\documentclass[12pt]{amsart}
\usepackage[T1]{fontenc}
\usepackage{mathptmx}

\usepackage{comment}
\usepackage[colorlinks=true,linkcolor=magenta,citecolor=blue]{hyperref}     
\usepackage{mathtools}
\usepackage{amsmath, amssymb,amsthm}
\numberwithin{equation}{section}
\usepackage[capitalise]{cleveref}   
\usepackage{array}
\usepackage{graphicx}
\usepackage{float}
\usepackage{caption,subcaption}
\usepackage{tikz}
\usepackage{tikz-cd}
\usetikzlibrary{matrix,arrows}
\usepackage[mathscr]{euscript}
\usepackage{xcolor}

\def\Abb{\mathbb{A}} 
\def\Cbb{\mathbb{C}} \def\Dbb{\mathbb{D}}
 
\def\Gbb{\mathbb{G}}

 \def\Pbb{\mathbb{P}}
\def\Qbb{\mathbb{Q}} \def\Rbb{\mathbb{R}}
 \def\Tbb{\mathbb{T}}

 \def\Zbb{\mathbb{Z}}

\def\Acal{\mathcal{A}} \def\Bcal{\mathcal{B}}
\def\Ccal{\mathcal{C}}  \def\Dcal{\mathcal{D}}
\def\Ecal{\mathcal{E}} \def\Fcal{\mathcal{F}}
 \def\Hcal{\mathcal{H}}
\def\Ical{\mathcal{I}} 
 \def\Lcal{\mathcal{L}}
\def\Mcal{\mathcal{M}} 
\def\Ocal{\mathcal{O}} \def\Pcal{\mathcal{P}}
\def\Qcal{\mathcal{Q}} 
\def\Scal{\mathcal{S}} \def\Tcal{\mathcal{T}}
\def\Ucal{\mathcal{U}} 
\def\Wcal{\mathcal{W}} 
 \def\Zcal{\mathcal{Z}}

\def\Iscr{\mathscr{I}}

 \def\Zscr{\mathscr{Z}}

\def\Sfrak{\mathfrak{S}}

\newtheorem{thm}{Theorem}[section]
\newtheorem{lem}[thm]{Lemma}
\newtheorem{ex}[thm]{Example}
\newtheorem{cor}[thm]{Corollary}
\newtheorem{prop}[thm]{Proposition}

\theoremstyle{definition}
\newtheorem{defi}[thm]{Definition}

\theoremstyle{remark}
\newtheorem{rem}[thm]{Remark}

\theoremstyle{plain}

\newcommand\supp{\mathrm{supp}}

\DeclareMathOperator{\ddc}{dd^c}

\DeclareMathOperator{\Poly}{Poly}
\DeclareMathOperator{\APoly}{APoly}
\DeclareMathOperator{\End}{End}
\DeclareMathOperator{\id}{id}
\newcommand{\Sa}{\Sigma^{\mathrm{aff}}}
\DeclareMathOperator{\Prep}{Prep}

\DeclareMathOperator{\Spec}{Spec}
\DeclareMathOperator{\PGL}{PGL}
\DeclareMathOperator{\Aut}{Aut}
\DeclareMathOperator{\Aff}{Aff}
\newcommand\hhat{\hat{h}}
\DeclareMathOperator{\dist}{dist}

\begin{document}
\author{Yugang Zhang}
\address{Yugang Zhang, Université Bourgogne Europe, CNRS, IMB UMR 5584, 21000 Dijon, France}
\email{yugang.zhang@ube.fr}
\urladdr{https://sites.google.com/view/yugangzhang}

\thanks{The author acknowledges support from the ATRACT programme of the
R\'egion Bourgogne--Franche--Comt\'e through the project ADYAUS
(RECH-ATRAC-000012), and from the French National Research Agency through the
project DynAtrois (ANR-24-CE40-1163).}
\title{Measure rigidity for regular polynomial endomorphisms and applications}

\begin{abstract}
We classify regular polynomial endomorphisms of $\Cbb^k$ with the same
equilibrium measure under conditions on the dynamics at infinity.
In particular, for $f$ in a Zariski dense open subset of the parameter space
and any $g$ of the same degree, $\mu_g=\mu_f$ if and only if $f=g$.
We also establish finiteness results for maps with a prescribed equilibrium
measure. We obtain stronger results in dimension two. Our approach rests
on the stable manifold structure near the hyperplane at infinity.
Applications include a Tits-type alternative, results on
iterated centralizers, and an arithmetic non-density theorem.
\end{abstract}

\maketitle

\setcounter{tocdepth}{1}
\tableofcontents

\section{Introduction}

\subsection{Motivation}

The equilibrium measure is one of the central objects in complex dynamics.
It encodes the chaotic part of a dynamical system and serves as a bridge from
holomorphic dynamics to ergodic theory and pluripotential theory, and,
through equidistribution and canonical heights, to arithmetic dynamics.

For a surjective endomorphism $f:\Pbb^k\to\Pbb^k$ of degree $d\geq2$, the
\emph{equilibrium measure} $\mu_f$ is the unique invariant probability measure of
maximal entropy, and its support is the \emph{Julia set} $J_f$.  In dimension one,
its construction goes back to Brolin for polynomials~\cite{Brolin1965}, and
to Lyubich~\cite{Lyubich1983} and Freire--Lopes--Ma\~n\'e for rational maps
\cite{FreireLopesMane1983}.  In higher dimensions, we refer to
Sibony's survey~\cite{sibonybook} and the lectures of Dinh--Sibony~\cite{DSbook}; see
also Briend--Duval~\cite{briendduval2}.

The inverse problem considered in this paper asks how much of the dynamics
can be recovered from the equilibrium measure.  For one-variable
polynomials, the model answer is Beardon's theorem.

\begin{thm}[Beardon {\cite{Beardon1990,Beardon1992}}]\label{thm: Beardon}
Let $f$ and $g$ be polynomials of degree $d\geq2$.  Then 
\[
\mu_f=\mu_g \Longleftrightarrow g=\sigma\circ f
\]
for an affine map $\sigma$ satisfying $\sigma(J_f)=J_f$.

Moreover, the group of affine symmetries $\sigma$ of $J_f$ is
finite cyclic unless $f$ is affinely conjugate to $z\mapsto z^d$.
\end{thm}

Closely related measure rigidity results were obtained by
Baker--Eremenko~\cite{BakerEremenko1987}, Schmidt--Steinmetz~\cite{SchmidtSteinmetz1995}, and Atela--Hu
\cite{AtelaHu1996}.  The situation for
rational maps of $\Pbb^1$ is
subtler: examples of Ritt~\cite{Ritt1923} and  Ye~\cite{Ye2015} show
that equality of measures need not come from M\"obius
postcomposition.  Although a comparably simple classification is therefore
unavailable, Levin and Przytycki obtained algebraic characterizations in
terms of functional relations between suitable iterates~\cite{LevinSymmetries90,LevinPrzytycki97}. See also related results of Dinh~\cite{DinhSameJulia2000}.

\medskip

The one-dimensional results above are now several decades old, but recent
developments in arithmetic dynamics have renewed interest in rigidity.
Arithmetic equidistribution~\cite{BakerRumely2006,Bilu1997,ChambertLoir2006,
ChambertLoirThuillier2009,FavreRiveraLetelier2006,
GauthierGoodHeights2026,KuhneEquidistribution2021,
MavrakiYe2023,Yuan2008,YZquasi} often yields proportionality between dynamically
defined canonical measures, and one wants to convert this analytic relation into
an algebraic relation. This strategy goes back to the equidistribution theorem of
Szpiro--Ullmo--Zhang~\cite{SzpiroUllmoZhang1997} and its use in the proofs
of the Bogomolov conjecture by Ullmo~\cite{Ullmo1998} and
Zhang~\cite{Zhang1998}.

We focus here on global rigidity. Local forms of measure rigidity were
studied by Dujardin--Favre--Gauthier
\cite{DujardinFavreGauthier2023} and Ji--Xie
\cite{JiXie_LocalRigidity}. This is a key input in Ji--Xie's proof of
the dynamical Andr\'e--Oort conjecture for curves
\cite{ji2023daocurves}. For polynomials, see also Luo~\cite{Luo2021Inhomogeneity}.

In this paper, we generalize to higher dimensions several measure-rigidity
results that play an important role in arithmetic dynamics. We highlight two
of them. The first is the relation between the equilibrium measure and the
set of preperiodic points, which is one of the key ingredients in
DeMarco--Mavraki's work on uniform bounds for common preperiodic
points~\cite{DMCompositio}. The second consists of Pakovich's generic
measure-rigidity results
\cite{pakovich-deg4,pakovich2026periodiccurvesgeneralendomorphisms}, which
provide an important ingredient in Ji--Xie's generic injectivity theorem for
the multiplier-spectrum morphism~\cite{JiXieMultiplierSpectrum}.

Some further measure rigidity results, mostly related to arithmetic dynamics,
can be found in
\cite{BianchiHe2026Joining,DKY2,
dujardin2023dynamicalmaninmumfordconjectureplane,FavreGauthier2022,
GauthierTaflinVigny2026,gong2026nonarchimedeanrigidityuniformitycommon,MSDuke}.
For polynomial automorphisms, see~\cite{abboud2026uniformboundcommonperiodic,
DujardinFavre2017,Vigny2026HenonSibony,
zhang2024arithmeticpropertiesfamiliesplane}.
For group-theoretic applications, which we will also generalize to higher
dimensions, see~\cite{beaumont2025centralizersendomorphismsprojectiveline,BHPTjems24}.

\medskip 

The one-dimensional results described above motivate similar questions for
regular polynomial endomorphisms in several variables. Their proofs,
however, do not seem to extend directly to this setting. The central difficulty
is to extract algebraic information
from the analytic identity $\mu_f=\mu_g$.  Our approach combines
pluripotential theory, laminar currents, and algebraic geometry; in
particular, it exploits a specifically higher-dimensional feature: the
variation of local stable manifolds along the hyperplane at infinity, which
has positive dimension, together with the dynamics induced on
that hyperplane.

\subsection{Main results}
We present here our main results. The main ideas of the proofs are outlined in
Section~\ref{sect:strategy}.

We fix the dimension $k\geq2$ and the degree $d\geq2$.
A polynomial endomorphism
$f:\Cbb^k\to\Cbb^k$ of degree $d$ is called \emph{regular} if it extends
holomorphically to an endomorphism of $\Pbb^k$. 
We denote the space of such maps by $\Poly^k_d$. Writing
$N_{d,k}:=k\binom{d+k}{k}$, this is the Zariski dense open subset of
$\Abb^{N_{d,k}}$ defined by the nonvanishing of the resultant of the leading
homogeneous coordinate functions; see also~\cite[\S~1.1, p.~169]{DSbook}.
Since the resultant has integer coefficients, $\Poly^k_d$ is defined over
$\Qbb$.
Write $f_h$ for the
degree-$d$ homogeneous part of $f$, and $f_\Pi:\Pi\simeq\Pbb^{k-1}\to\Pi$ for the induced endomorphism at infinity, where $\Pi:=\Pbb^k\setminus\Cbb^k$.

Denote by $\Aff(\Cbb^k)$ the set of affine automorphisms of $\Cbb^k$, and set
\[
   \Sa_f
   :=\{\sigma\in\Aff(\Cbb^k)\mid \sigma(J_f)=J_f\}.
\]

\begin{thm}\label{thm: main1}
Let $f,g\in\Poly^k_d$.
Suppose that $g_h=A\circ f_h$ for an invertible linear map
$A:\Cbb^k\to\Cbb^k$.  Then
\[
   \mu_f=\mu_g
   \quad\Longleftrightarrow\quad
   g=\sigma\circ f
\]
for some $\sigma\in\Sa_f$ whose linear part is $A$.
\end{thm}

For $e\geq2$, set
\[
\Sigma_f^{e\mathrm{poly}}
   :=\{g\in\Poly^k_e\mid \mu_g=\mu_f\},
   \qquad
   \Sigma_f^{\mathrm{poly}}
   :=\bigcup_{e\geq2}\Sigma_f^{e\mathrm{poly}}.
\]
Let $f^n$ denote the $n$-th iterate of $f$.

\begin{thm}\label{thm:generic-measure-rigidity}
There exists a Zariski dense open subset
$\Ucal_{d,k}\subset\Poly^k_d$, defined over $\Qbb$, such that, for every
$f\in\Ucal_{d,k}(\Cbb)$, $\Sigma_f^{d\mathrm{poly}}=\{f\}$.

If $k=2$, then $\Sigma_f^{\mathrm{poly}}= \{f^n \mid n\geq 1 \}$.
\end{thm}

Let us give some remarks on the assumption $g_h=A\circ f_h$ and show why
this is a natural condition and holds in several important cases.
\begin{itemize}
    \item If $\mu_f=\mu_g$, then $\mu_{f_\Pi}=\mu_{g_\Pi}$. Therefore,
    if the conclusion $g=\sigma\circ f$ holds on $\Cbb^k$, we have
    $g_\Pi=\sigma_\Pi\circ f_\Pi$ on $\Pi$ as well. However, when $k=2$,
    Ritt and Ye gave examples showing that there exist rational maps that
    have the same equilibrium measure on $\Pbb^1$ but do not differ by
    postcomposition with a M\"obius transformation; see~\cite{Ritt1923,Ye2015}.
    See Example~\ref{ex:leading-system-hypothesis} for a concrete example
    in $\Cbb^2$. Hence Theorem~\ref{thm: main1} is false in full
    generality, without any additional assumption.

    \item That being said, for general $f$, $\mu_f=\mu_g$ is indeed
    equivalent to $f=g$ when $\deg g=\deg f$ by
    Theorem~\ref{thm:generic-measure-rigidity}. We also give a class of
    polynomial maps, including products of one-variable polynomials of
    degree $d$, for which the conclusion holds without the assumption on
    the homogeneous parts; see Corollary~\ref{cor:coordinate-powers}.

    \item We choose to make the classification in terms of $\sigma\circ f$
    since this form is very convenient to work with, as its use throughout
    this work shows. However, one could make the assumption less restrictive
    by asking for weaker conclusions that are less convenient to manipulate
    (at least for the applications in our text). However, one still needs
    to exclude some exceptional classes in order to get reasonable results.
    In dimension one, a rational map $r$ is called \emph{special} if it is
    conjugate to $z\mapsto z^{\pm d}$, to a Chebyshev map $\pm T_d$, or is a
    Latt\`es map.
    In dimension two, one can say that $f$ is \emph{special} if $f$ is
    affinely conjugate to a homogeneous polynomial map or $f_\Pi$ is special.
    Then, for nonspecial $f$, one can characterise the equality $\mu_f=\mu_g$
    without any further assumptions, as in Theorem~\ref{thm: LP-C2}.
\end{itemize}

A point is \emph{preperiodic} if its $f$-orbit is finite. Write $\Prep(f)$
for the set of preperiodic points of $f$.

\begin{thm}\label{thm: LP-C2}
Let $f\in\Poly^2_d$ and $g\in\Poly^2_e$ for some $d,e\geq2$.
Assume that $f$ is nonspecial. Then the following are equivalent:
\begin{enumerate}
   \item $\mu_f=\mu_g$;
   \item there exist integers $a,b\geq1$ such that
   \[
      f^{2a}=f^a\circ g^b
      \qquad\text{and}\qquad
      g^{2b}=g^b\circ f^a;
   \]
   \item there exist integers $a,b\geq1$ such that the diagonal
   $\Delta\subset\Pbb^2\times\Pbb^2$ is preperiodic under $f^a\times g^b$;
   \item $\Prep(f)=\Prep(g)$;
   \item the semigroup $\langle f,g\rangle$ is not freely generated by $f$ and $g$.
\end{enumerate}
If $d=e$, these conditions are also equivalent to $\Delta$ being
preperiodic under $f\times g$.
\end{thm}

Some equivalences can be made more precise. Without assuming that $f$ is
nonspecial, we show in Theorem~\ref{thm:same-preper} that (1) is equivalent
to (4) for every $g\in\Poly^2$ if and only if $f$ is not affinely conjugate
to any of the maps listed in Theorem~\ref{thm:main 6}.
The equivalence $(4)\Leftrightarrow(5)$ holds without any additional
assumptions on $f$ and $g$; see Corollary~\ref{cor: prep-free}.
The equivalence of (1), (2), and (3), together with the assertion for
$d=e$, is proved in Section~\ref{sect:iterated-functional-relations}.
The equivalence of (1) and (2) is a two-dimensional analogue of the
Levin--Przytycki functional characterization of equal equilibrium
measures~\cite{LevinPrzytycki97}.
\medskip

Next we study the size of the group $\Sa_f$. If $f$ is homogeneous, then $f$ and $\lambda f$ have the same equilibrium
measure for every $\lambda\in S^1$.  We prove the useful 
characterization that $f$ is homogeneous if and only if $J_f$ is
$S^1$-invariant; see Lemma~\ref{lem: charcterizationHomogeneous}.

Set
\[
   \Sigma_f^0:=\{\sigma\in\Sa_f\mid \sigma_\Pi=\id\}.
\]

\begin{thm}\label{thm: main2}
Let $f\in\Poly^k_d$. Then $\Sigma_f^0$ is a central subgroup of $\Sa_f$.
It is finite if and only if $f$ is not affinely conjugate to a homogeneous
map. In this case, it is cyclic of order at most $d$.
\end{thm}

In dimension two, we obtain the following classification.

\begin{thm}\label{thm:main 6}
Let $f\in\Poly^2_d$. Then $\Sa_f$ is infinite if and only if $f$ is affinely
conjugate to one of the following maps:
\begin{enumerate}
\item a homogeneous polynomial map;
\item $(x,y)\mapsto(p(x),y^d)$, where $p$ is a monic polynomial of degree
$d$ that is not affinely conjugate to $z\mapsto z^d$;
\item
\begingroup
\postdisplaypenalty=10000
\[
   (x,y)\longmapsto
   \left(x^d,\ y^dP\!\left(\frac{x^s}{y^r}\right)\right),
\]
\endgroup
where $\gcd(r,s)=1$ and $1\leq s<r\leq d$, with
\[
   P\in\Cbb[t],\qquad P(0)=1,\qquad
   1\leq\deg P\leq\left\lfloor\frac{d}{r}\right\rfloor.
\]
\end{enumerate}
\end{thm}

Unlike the
one-variable polynomial case, a finite group $\Sa_f$ need
not be cyclic: Examples~\ref{ex:klein-four-centralizer} and
\ref{ex:dihedral-centralizer} realize groups containing the Klein four group
and the dihedral group $D_4$, respectively.

Finally, let us mention that Ueno studied pairs of nondegenerate polynomial
skew products on $\Cbb^2$ sharing the same equilibrium measure; see
\cite{Ueno2010,Ueno2020}.
His proof is still one-dimensional, since he uses B{\"o}ttcher coordinates
on the vertical fibers $\{x\}\times\Cbb\simeq\Cbb$, constructed in the
regular case by Jonsson~\cite{Jonsson1999} and in his setting from the
fiberwise Green functions of Favre--Guedj~\cite{FavreGuedj2001}.
Our measure rigidity proofs instead use local stable manifolds.
Let us describe them now in more detail.

\subsection{Strategy of the proofs of Theorems~\ref{thm: main1}
and~\ref{thm:generic-measure-rigidity}}\label{sect:strategy}

Our proofs begin by studying the dynamics near the hyperplane at infinity.
The starting point is Bedford--Jonsson's description of the Green current
near infinity in terms of stable disks~\cite{BedfordJonssonAJM00}.
Our use of local stable manifolds is also inspired by the work of
Dujardin--Favre--Ruggiero
\cite{dujardin2023dynamicalmaninmumfordconjectureplane}.

For Theorem~\ref{thm: main1}, equality of equilibrium measures first implies
equality of Green functions by Corollary~\ref{cor:mu-implies-green}.
We then show that $f$ and $g$ have the same local stable-manifold germs for
$\mu_{f_\Pi}$-almost every direction in $\Pi$. This is the key geometric input,
for which we give two proofs. In dimension two, the argument
uses Bedford--Jonsson's laminar decomposition in
Lemma~\ref{lem: BJ6.1}, together with the theory of geometric intersection
introduced by Bedford--Lyubich--Smillie~\cite{BS4} and further developed by
Dujardin~\cite{DujardinPM04}. In arbitrary dimension, we instead combine
Bedford--Jonsson's global finite-area disk representation
\cite[Theorem~6.10]{BedfordJonssonAJM00}, stated here as
Theorem~\ref{thm:BJ6.10}, with the result of Dinh--Nguy\^en--Sibony
\cite{DNS10JDG} that the trace measure $\|T_f^{k-1}\|_{\mathrm{FS}}$ is
locally moderate in the sense of Dinh--Sibony~\cite{DS03PolynomialLike}; see
Lemma~\ref{lem:moderate-near-infinity}. This yields
Proposition~\ref{prop:zariski-dense-common-stable-manifolds}.

The second geometric observation is the following. Take a local curve $C$
transverse to $\Pi$. Its image under $f$ is usually tangent to $\Pi$.
However, we have here local stable manifolds which are all transverse to
$\Pi$, and $f$ and $g$ map them to other local stable manifolds. This
observation leads to two consequences. The first one allows us to finish
the proof of Theorem~\ref{thm: main1}. Indeed, suppose for simplicity that
$g_h=f_h$. Then, along the common local stable manifolds, the tangential
coordinates of $f$ and $g$ differ by $O(t^d)$; see
\eqref{eq:tangential-image-estimates}, where $t=0$ locally defines $\Pi$.
Applying the divisibility lemma, Lemma~\ref{lem:stable-manifold-divisibility},
to the $2\times2$ minors of their homogeneous lifts then shows that $g-f$
is constant, finishing the proof. The second consequence is described below in
Step~1.

\medskip

Now we give the main ideas of the proof of Theorem~\ref{thm:generic-measure-rigidity}.
The proof occupies Sections~\ref{sect:generic-first-jet}--\ref{sect:generic-stabilizer}.

\smallskip
\noindent\textit{Step 1: Fiberwise admissibility and algebraization.
See Section~\ref{sect:generic-first-jet}.}
Write $F=f_h$ and $B=f_{d-1}$, where $B$ is the degree-$(d-1)$ homogeneous part of $f$.
Let $S_d=H^0(\Pbb^{k-1},\Ocal(d))$. We call a $k$-dimensional subspace
$L\subset S_d$ $(F,B)$-admissible if
\[
   \det DF\mid(P\nabla Q-Q\nabla P)\operatorname{adj}(DF)B
   \qquad\text{for every }P,Q\in L,
\]
where $\nabla$ denotes the row gradient and $\operatorname{adj}$ the
adjugate matrix; see Definition~\ref{def:FB-admissible}.

This definition is justified by Proposition~\ref{prop:measure-implies-admissible},
which shows that if $g\in\Poly^k_d$ satisfies $\mu_f=\mu_g$, then
$L_{g_h}:=\langle(g_h)_1,\ldots,(g_h)_k\rangle$ is $(F,B)$-admissible.
To see why, write a common local stable manifold through $[Z]\in\Pi$ as
$[Z+t\xi+O(t^2):t]$. The first $k$ homogeneous coordinates of its image
under $f$ are
\[
   F(Z)+t\bigl(DF(Z)\xi+B(Z)\bigr)+O(t^2).
\]
The transverse coordinate of the image is $O(t^d)$. Since the image lies
in a local stable manifold transverse to $\Pi$, its tangential coordinates
also differ from their values at $f_\Pi([Z])$ by $O(t^d)$. As $d\ge2$,
this gives
\[
   DF(Z)\xi+B(Z)=\lambda F(Z).
\]
Where $\det DF(Z)\ne0$, Euler's identity $DF(Z)Z=dF(Z)$ then yields
\[
   \xi=-DF(Z)^{-1}B(Z)+\nu Z.
\]

Write $G=g_h$ and $C=g_{d-1}$. The same calculation for $g$ gives
$DG(Z)\xi+C(Z)=\lambda_gG(Z)$. We obtain
\[
   \frac{(G_i\nabla G_j-G_j\nabla G_i)\operatorname{adj}(DF)B}
        {\det DF}
   =G_iC_j-G_jC_i.
\]
Since the right-hand
side is polynomial and the identity holds on a Zariski-dense set of
directions, the numerator is divisible by $\det DF$. By bilinearity, this
gives the admissibility condition for every $P,Q\in L_{g_h}$. Finally, note that,
tautologically, $L_F:=\langle F_1,\ldots,F_k\rangle$ is always
$(F,B)$-admissible; see Lemma~\ref{lem:tautological-solution}.

\smallskip
\noindent\textit{Step 2: The power-map fiber and spreading out.
See Section~\ref{sect:generic-admissible-uniqueness}.}
Denote by $\Qcal_{d,k}$ the parameter space of $(F,B)$.
We globalize the admissibility 
condition by constructing a projective incidence morphism
$\pi:\Iscr\to\Qcal_{d,k}$ such that, for every $(F,B)\in\Qcal_{d,k}(\Cbb)$, the complex points of the
fiber $\Acal(F,B)$ are precisely the $(F,B)$-admissible $k$-planes. We then study the fiber over a carefully
chosen pair $q_{\mathrm{pow}}=(F^{\mathrm{pow}},B^{\mathrm{pow}})$ whose leading term is the power map. An
explicit calculation shows that this fiber is the reduced point
$L_{\mathrm{pow}}$; see Lemma~\ref{lem:power-fiber}. Properness, together
with upper semicontinuity of the fiber length, spreads this property to a
Zariski dense open subset of $\Qcal_{d,k}$. Thus, for every complex point
$(F,B)$ of this open subset, $\Acal(F,B)$ is the reduced point $L_F$; see
Proposition~\ref{prop:generic-admissible-uniqueness}. Consequently, by Corollary~\ref{cor:leading-system-uniqueness}, if $f$ and $g$ have the same equilibrium measure, $L_{g_h}=L_F$.

\smallskip
\noindent\textit{Step 3: The stabilizer argument.
See Section~\ref{sect:generic-stabilizer}.}
The equality $L_{g_h}=L_F$ gives an invertible linear map $A$ such that
$g_h=A\circ F$. Theorem~\ref{thm: main1} then gives an affine symmetry
$\sigma(z)=Az+b$ of $J_f$ such that $g=\sigma\circ f$. It remains to show that
$\sigma=\id$ for a general $f$.

Let $H_F=\operatorname{Stab}_{\PGL_k}(L_F)$. If
$(d,k)\neq(2,3)$, then $H_F$ is trivial on a Zariski dense open subset
by Lemma~\ref{lem:nonexceptional-stabilizer}. This follows from the generic
stabilizer theorem in~\cite{GuralnickLawther2024}. For $k=2$ and $d\ge 4$, this is also proved by Pakovich in~\cite[Lemma~5.3]{pakovich2026periodiccurvesgeneralendomorphisms}.
The pair $(d,k)=(2,3)$ is exceptional for this generic-stabilizer argument.
In that case, we replace triviality by the weaker rigidity
condition~\eqref{eq:rigidity-condition}.
Proposition~\ref{prop:uniform-rigidity-condition} shows that this condition holds on a Zariski dense open subset in every
case. This condition is used to show that $A=\id$.

In outline:
\[
\mu_f=\mu_g\ \Longrightarrow\ L_{g_h}\ \text{is }(F,B)\text{-admissible}\ \Longrightarrow\ L_{g_h}=L_F\ \Longrightarrow\ g=\sigma\circ f\ \Longrightarrow\ \sigma=\id.
\]

The preceding three steps apply when $k\geq3$ (in fact, only
Step~2 needs $k\geq3$); see
Remark~\ref{rem:surface-power-fiber}. When $k=2$, the stronger
cross-degree assertion is proved separately in
Section~\ref{sect: first-consequences}, using
Theorem~\ref{thm: pakovich-measure} and
Lemma~\ref{lem:generic-trivial-affine-kernel}.

\subsection{The image--fiber principle and applications}
\label{sect:intro-image-fiber-applications}

Put
\[
\Poly^k:=\bigcup_{d\geq2}\Poly^k_d,\qquad \APoly^k:=\Poly^k\cup \Aff(\Cbb^k).
\]
Consider the restriction map
\[
   \rho:\APoly^k\longrightarrow\End(\Pi),
   \qquad h\longmapsto h_\Pi,
   \qquad \Pi\simeq\Pbb^{k-1}.
\]
For a dynamically defined family in $\APoly^k$, one can first control
its \emph{image} under $\rho$ using the dynamics at infinity by applying the
existing results, and then control
each \emph{fiber} using our main theorems. We refer to this two-step mechanism
as the \emph{image--fiber principle}:
\[
   \begin{gathered}
      \text{lower-dimensional rigidity at infinity}
      +\ \text{higher-dimensional rigidity in the fibers}\\
      \Longrightarrow\ \text{global rigidity on }\Cbb^k.
   \end{gathered}
\]

For instance,
\begin{itemize}
    \item The case $k=2$ in Theorem~\ref{thm:generic-measure-rigidity}
    $=$ the first assertion of Theorem~\ref{thm:generic-measure-rigidity}
    $+$ Pakovich's Theorem~\ref{thm: pakovich-measure}.
    \item The equivalence of $(1)$ and $(2)$ in Theorem~\ref{thm: LP-C2}
    $=$ Theorem~\ref{thm: main1}
    $+$ Theorem~\ref{thm: main2}
    $+$ the Levin--Przytycki Theorem~\ref{thm: Levin-P}.
    \item Corollary~\ref{cor:finite-symmetries-C2}
    $=$ Theorem~\ref{thm: main2}
    $+$ Levin's Theorem~\ref{thm: Levin}.
\end{itemize}

The same principle also applies to other types of results.

\subsubsection{Tits alternative}\label{sect:intro-tits}

Bell--Huang--Peng--Tucker~\cite[Question~5.2]{BHPTjems24} asked whether their
Tits-type alternative for semigroups of endomorphisms of $\Pbb^1$ extends to
higher dimension.  We answer the question for regular polynomial
endomorphisms of $\Cbb^2$.

Let us recall some notions; see \cite[Sect.~2]{BHPTjems24}.
Let $\Scal$ be a finitely generated semigroup, and let $S$ be a finite set of
generators. Denote by $S^{\leq n}$ the set of elements of $\Scal$ that can be
expressed as a product of elements of $S$ of length at most $n$. Then define
$d_S(n)=\#S^{\leq n}$ to be the cardinality of $S^{\leq n}$. One says that
$\Scal$ has \emph{polynomially bounded growth} if $d_S(n)=O(n^r)$ for some
$r>0$, and it has \emph{linear growth} if $d_S(n)=O(n)$. This does not depend
on the chosen generating set $S$. For any semigroup $\Tcal$ considered
below, we denote its subset of nonlinear elements by
$\Tcal^+:=\{h\in\Tcal\mid \deg h\geq2\}$.

\begin{thm}\label{thm: Tits1}
Let $\Scal$ be a finitely generated semigroup in $\APoly^2$.  Then:
\begin{enumerate}
   \item either $\Scal$ has polynomially bounded growth or $\Scal$ contains
   a nonabelian free semigroup;
   \item if there exists $f\in\Scal^+$ such that $f_\Pi$ is nonspecial, then
   either $\Scal$ has linear growth or $\Scal$ contains a nonabelian free
   semigroup.
\end{enumerate}
\end{thm}

The proof compares the growth of $\Scal$ with that of $\rho(\Scal)$. The
one-dimensional results recalled in Theorems~\ref{thm: BHPT1.3}
and~\ref{thm: Beaumont-tits} and Proposition~\ref{prop: BHPT4.10} control the
image, while Proposition~\ref{prop: finite_fiber}, based on measure rigidity,
gives a uniform bound for the fibers of $\rho|_\Scal$.

\subsubsection{Centralizers}

For $f\in\APoly^2$, define
\[
   C(f^n):=\{g\in\APoly^2\mid g\circ f^n=f^n\circ g\},
   \qquad
   C(f^\infty):=\bigcup_{n\geq1}C(f^n).
\]

The relevant one-dimensional finiteness results were established by
Pakovich~\cite{Pak20,Pak21} and recently strengthened by
Beaumont~\cite{beaumont2025centralizersendomorphismsprojectiveline}. Combining this theorem with the fiber control in
Lemma~\ref{lem:restriction-to-infinity-centralizer} yields the following.

\begin{thm}\label{thm: centralizer}
Let $f\in\Poly^2_d$ for some $d\geq2$.  Assume that $f$ is not
affinely conjugate to a homogeneous polynomial and that $f_\Pi$ is not
conjugate to a power map.
\begin{enumerate}
   \item There exists $N\geq1$ such that
   \[
      C(f^\infty)=C(f^N).
   \]
   \item If, moreover, $f_\Pi$ is nonspecial, then there exists a finite set
   $F\subset C(f^\infty)$ such that
   \[
      C(f^\infty)=F\langle f\rangle
        =\{q\circ f^m\mid q\in F,\ m\geq0\}.
   \]
\end{enumerate}
\end{thm}

In~\cite[Theorem~C]{Zhang2026IteratedCentralizer}, the author showed that, in any
dimension, $C(f^\infty)=\langle f\rangle$ for a general $f$.
In a complementary direction,  Dinh~\cite{DinhCommuting}, Dinh--Sibony~\cite{DinhSibony2002}, Kaufmann~\cite{Kaufmann2018}, and Ji--Song--Xie~\cite{JiSongXie2026} studied various commuting pairs of endomorphisms of $\Pbb^k$ that do not share a common iterate, extending to higher dimensions earlier results on $\Pbb^1$ by Fatou~\cite{Fatou1923}, Julia~\cite{Julia1922}, Ritt~\cite{Ritt1923}, and Eremenko~\cite{Eremenko1990}.

\subsubsection{Non-density of orbit intersections}

Hsia--Tucker~\cite{HsiaTucker} asked whether simultaneous orbit
intersections for two compositionally independent maps must lie in a proper
Zariski closed subset.  This question is rooted in earlier work of Bugeaud--Corvaja--Zannier~\cite{BCZ03},
who proved an upper bound for the common divisor of $a^n-1$ and
$b^n-1$ under the assumption that $a$ and $b$ are multiplicatively
independent integers at least $2$. We refer the reader to Yang--Zhong~\cite{yang2026dynamicalgcdproblemsvariant} for more background and recent developments.
We obtain a positive answer in the present setting.

\begin{thm}\label{thm: arith}
Let $f,g\in\Poly^2$ have degree at least $2$ and be compositionally
independent, equivalently, suppose that $\langle f,g\rangle$ is freely
generated by $f$ and $g$.  Let $c$ be an endomorphism of $\Cbb^2$ or $\Pbb^2$ that is not a
compositional power of $f$ or $g$.  Then the set of $z\in\Abb^2(\Cbb)$ for
which
\[
   f^m(z)=g^n(z)=c(z)
\]
for some positive integers $m,n$ is not Zariski dense.
\end{thm}

For regular polynomial skew products over a number field with $\deg f = \deg g$, this was proved by
Noytaptim--Zhong~\cite{noytaptim2024commonzerositeratedmorphisms}. Our proof is
modeled on theirs.

\subsection*{Conventions and organization of the paper}

The composition $f\circ g$ will usually be written multiplicatively as $fg$
when no confusion is possible. A property holds for a \emph{general} point
$x$ of an algebraic variety $X$ if it holds on a Zariski dense open subset of
$X$.

\medskip

Section~\ref{sect: basic} recalls the basic facts from regular polynomial
dynamics and establishes the two stable-manifold inputs,
Propositions~\ref{Prop: StableManifoldsSame} and
\ref{prop:zariski-dense-common-stable-manifolds}.

Section~\ref{sect: proof main} proves Theorems~\ref{thm: main1},
\ref{thm: main2}, and~\ref{thm:main 6}.
Section~\ref{sect: first-consequences} is devoted to further results in
dimension two. It proves the equivalences of $(1)$--$(3)$ in
Theorem~\ref{thm: LP-C2}, the second assertion of
Theorem~\ref{thm:generic-measure-rigidity}, and the characterization by
preperiodic sets in Theorem~\ref{thm:same-preper}.

Sections~\ref{sect:generic-first-jet}--\ref{sect:generic-stabilizer}
carry out in detail the three steps of the proof strategy for
Theorem~\ref{thm:generic-measure-rigidity} outlined in
Section~\ref{sect:strategy}, one step in each section.

Finally, Section~\ref{sect:further-applications} develops three further
applications of the image--fiber principle. Subsection~\ref{sect: tits}
proves the Tits alternative in Theorem~\ref{thm: Tits1},
Subsection~\ref{sect: cent} proves the centralizer theorem,
Theorem~\ref{thm: centralizer}, and Subsection~\ref{sect: arith} proves the
arithmetic application, Theorem~\ref{thm: arith}.

\subsection*{Acknowledgements}

The first breakthrough in this work arose from a conversation with Romain
Dujardin, and I am grateful to him for generously taking the time to
discuss mathematics with me. I also thank Marc Abboud, Alonso Beaumont, Zhuchao Ji, Johan
Taflin, and Gabriel Vigny for helpful discussions and for sharing their
insights on various aspects of this work.

\section{Regular polynomial endomorphisms and local stable manifolds}\label{sect: basic}
We first recall basic properties of regular polynomial endomorphisms of
$\Cbb^k$, including Green functions, equilibrium measures, Julia sets, and
local stable manifolds at infinity. All statements in this section are valid
in arbitrary dimension $k\geq2$, except
Proposition~\ref{Prop: StableManifoldsSame}, whose geometric-intersection
proof is specific to dimension two. The current-theoretic substitute in
arbitrary dimension is Lemma~\ref{lem:stable-manifolds-same-k}, which yields
Proposition~\ref{prop:zariski-dense-common-stable-manifolds}.

\subsection{Green functions, equilibrium measures and Julia sets}\label{section: Basic green}
Fix $k\geq2$ and let $f\in\Poly^k_d$, where $d\geq2$. We use the
homogeneous decomposition
\begin{equation}
   f=\sum_{m=0}^d f_m,
   \qquad
   f_m=(f_{m,1},\ldots,f_{m,k}),
   \qquad
   f_h:=f_d,
   \label{eq:homogeneous-decomposition}
\end{equation}
where each $f_m:\Cbb^k\to\Cbb^k$ is homogeneous of degree $m$. Thus the
subscript $m$ always denotes homogeneous degree, while the second subscript
denotes a coordinate component.

Using homogeneous coordinates $[Z:T]=[Z_1:\cdots:Z_k:T]$ on $\Pbb^k$, the
projective extension of $f$ is
\begin{equation*}
   \widehat f[Z:T]
   =\left[\sum_{m=0}^d T^{d-m}f_m(Z):T^d\right].
\end{equation*}
Regularity is equivalent to the coordinate polynomials of $f_h=f_d$ having
no common zero on $\Pbb^{k-1}$. Consequently, the hyperplane at infinity
$\Pi=\{T=0\}$ is totally invariant, and the induced map is the holomorphic
endomorphism
\[
   f_\Pi:\Pi\simeq\Pbb^{k-1}\longrightarrow\Pi\simeq\Pbb^{k-1},
   \qquad
   [Z]\longmapsto[f_h(Z)].
\]

The Green function of $f$ is defined by
\begin{equation*}
   G_f(z):=\lim_{n\to\infty}\frac{1}{d^n}\log^+\|f^n(z)\|.
\end{equation*}
It is a nonnegative continuous plurisubharmonic function on $\Cbb^k$.
The (logarithmically homogeneous) Green function of $f_h$ is defined by
\begin{equation*}
   G_{f_h}(z):=\lim_{n\to\infty}\frac{1}{d^n}\log\|f_h^n(z)\|,
   \qquad z\in\Cbb^k\setminus\{0\}.
\end{equation*}
We collect some standard properties; see, for instance,
Bedford--Jonsson~\cite{BedfordJonssonAJM00}.

\begin{prop}\label{Prop: basic-green}
The Green functions $G_f$ and $G_{f_h}$ satisfy the following properties.
\begin{enumerate}
\item \emph{(Invariance)}
      $G_f\circ f=dG_f$ and $G_{f_h}\circ f_h=dG_{f_h}$.
\item \emph{(Logarithmic growth)} There is a continuous Robin function
      $\rho_f:\Pi\to\Rbb$ such that
      \[
         G_{f_h}(z)=\log\|z\|+\rho_f([z]),
         \qquad
         G_f(z)=\log\|z\|+\rho_f([z])+o(1)
      \]
      as $\|z\|\to\infty$.
\item \emph{(Uniqueness)} If $u$ is a continuous nonnegative
      plurisubharmonic function on $\Cbb^k$ with logarithmic growth and
      $u\circ f=du$, then $u=G_f$.
\end{enumerate}
\end{prop}

The zero set $K_f:=\{G_f=0\}$ is the \emph{filled Julia set}; it is precisely
the set of points with bounded
forward orbit. Set
\[
   T_f:=\ddc G_f,
   \qquad
   \mu_f:=T_f^k=(\ddc G_f)^k,
   \qquad
   J_f:=\supp(\mu_f).
\]
The Green current extends canonically to $\Pbb^k$, and its restriction to
$\Pi$ is the Green current $T_{f_\Pi}$ of $f_\Pi$. In particular,
\[
   \mu_{f_\Pi}=T_{f_\Pi}^{k-1}=(T_f|_\Pi)^{k-1}.
\]

There is a strong relation between $K_f$ and $J_f$. We briefly recall some notions from pluripotential theory; see, for example, the books of Klimek~\cite{Klimek} and Dinh and Sibony~\cite{sibonybook,DSbook}.

Let $K\subset\Cbb^k$ be compact. Its \emph{Shilov boundary} $\partial_SK$ is
the unique smallest closed subset $E\subset K$ such that
$\max_E|P|=\max_K|P|$ for every polynomial $P$. Its \emph{polynomial hull} is
\[
   \widehat K:=\{z\in\Cbb^k\mid |P(z)|\leq\sup_K|P|
      \text{ for every polynomial }P\}.
\]
The compact set $K$ is \emph{polynomially convex} if $\widehat K=K$. In that
case $\widehat{\partial_SK}=K$.

Finally, the \emph{pluricomplex Green function} of $K$ is defined by
\[
   G_K(z):=\sup\{u(z)\mid u\text{ is psh on }\Cbb^k,\ u\leq0\text{ on }K,
      \ u(z)\leq\log\|z\|+O(1)\}.
\]
The compact set $K$ is \emph{regular} if $G_K$ is continuous. In this case,
$\supp(\ddc G_K)^k=\partial_SK$; see
Bedford--Taylor~\cite{BTJFA87}.

Returning to dynamics, the Green function $G_f$ is precisely the pluricomplex Green function of $K_f$. Hence the Julia set $J_f$ is the Shilov boundary of the polynomially convex compact set $K_f$, and therefore $\widehat{J_f}=K_f$.

We shall use the following uniqueness theorem of Bedford--Taylor for plurisubharmonic functions of logarithmic growth~\cite{BT89}.

\begin{prop}[Bedford--Taylor]\label{prop:BT-uniqueness}
Let $u,v$ be continuous nonnegative plurisubharmonic functions on $\Cbb^k$
with logarithmic growth at infinity. If
$(\ddc u)^k=(\ddc v)^k$, then $u-v$ is constant. If, moreover, both functions
vanish somewhere, then $u=v$.
\end{prop}

Applying Proposition~\ref{prop:BT-uniqueness} to Green functions yields the following.

\begin{cor}\label{cor:mu-implies-green}
Let $f\in\Poly^k_d$ and $g\in\Poly^k_e$, where $d,e\geq2$.
If $\mu_f=\mu_g$, then
\[
   G_f=G_g,
   \qquad
   T_f=T_g,
   \qquad
   G_{f_h}=G_{g_h},
   \qquad
   \mu_{f_\Pi}=\mu_{g_\Pi}.
\]
\end{cor}

\begin{proof}
Bedford--Taylor uniqueness gives $G_f=G_g$, and hence $T_f=T_g$. Comparing
the logarithmic asymptotics at infinity gives equality of the Robin functions,
which yields the last two equalities.
\end{proof}

We also record the converse implication used throughout the paper.

\begin{prop}\label{prop:easy-direction}
Let $f,g\in\Poly^k_d$. If $g=\sigma f$ for an affine
automorphism $\sigma$ satisfying $\sigma(J_f)=J_f$, then $\mu_g=\mu_f$.
\end{prop}

\begin{proof}
Since affine automorphisms commute with polynomial hulls,
$\sigma(J_f)=J_f$ implies
\[
   \sigma(K_f)=\sigma(\widehat{J_f})
   =\widehat{\sigma(J_f)}=\widehat{J_f}=K_f.
\]
Therefore $G_f\circ\sigma=G_f$, and $G_f\circ g=G_f\circ(\sigma f)=dG_f.$
Proposition~\ref{Prop: basic-green}(3) gives $G_g=G_f$, and consequently
$\mu_g=\mu_f$.
\end{proof}

\subsection{Local stable manifolds and the two-dimensional argument}
\label{section: stable manifold}
We write $J_\Pi:=J_{f_\Pi}$ and $\mu_\Pi:=\mu_{f_\Pi}$.
By Pesin theory, for $\mu_\Pi$-almost every $a\in J_\Pi$ there exists
$\delta(a)>0$ such that the \emph{local stable manifold}
\[
   W^s_{f,\mathrm{loc}}(a)
   :=\{x\in\Pbb^k\mid
       \dist(f^j(x),f^j(a))<\delta(a)
       \text{ for every }j\geq0\}
\]
is a one-dimensional embedded complex disk.

Fix $\eta>0$. Set
\[
   \Omega_f:=\Pbb^k\setminus K_f
   =\Pi\cup\{z\in\Cbb^k\mid G_f(z)>0\},
\]
and, for $R>0$, set
\[
   \Omega_{f,R}
   :=\Pi\cup\{z\in\Cbb^k\mid G_f(z)>R\}.
\]
By Pesin theory (see
\cite[pp.~178--180]{BedfordJonssonAJM00} or
\cite[Corollary~5.3]{PS89}), there are a compact set
$E_f=E(f,\eta)\subset J_\Pi$ and $R_0>0$ such that the following properties hold:
\begin{itemize}
\item $\mu_\Pi(E_f)\geq1-\eta$ and $E_f$ has no isolated points;
\item for every $a\in E_f$, the disk
      $W^s_{f,R_0}(a):=W^s_{f,\mathrm{loc}}(a)\cap\Omega_{f,R_0}$ is properly
      embedded in $\Omega_{f,R_0}$;
\item the family $\{W^s_{f,R_0}(a)\mid a\in E_f\}$ forms a Riemann surface
      lamination in $\Omega_{f,R_0}$.
\end{itemize}
Moreover, Bedford--Jonsson proved the following laminar decomposition:
\begin{lem}[Bedford--Jonsson {\cite[Lemma~6.1]{BedfordJonssonAJM00}}]\label{lem: BJ6.1}
We have
\[
   T_f^{k-1}|_{\Omega_{f,R_0}}
   =\int_{E_f}[W^s_{f,R_0}(a)]\,d\mu_\Pi(a)+S_f,
\]
where $S_f$ is a positive closed current of bidimension $(1,1)$ and mass at
most $\eta$.
\end{lem}

From this we can deduce the following.
\begin{prop}\label{Prop: StableManifoldsSame}
    Fix $d\geq 2.$ Let $f,g\in \Poly^2_d$ and
    suppose $\mu=\mu_f=\mu_g$. Then for $\mu_\Pi$ almost every $a\in J_\Pi$
    the local stable manifolds $W^s_{f,\mathrm{loc}}(a)$ and
    $W^s_{g,\mathrm{loc}}(a)$ coincide in a neighborhood of $a$.
\end{prop}
\begin{proof}
Fix $\eta>0$ and apply the preceding Pesin construction to both $f$ and $g$,
obtaining compact sets $E_f,E_g$ and a common level $R_0$ with
$\Omega_{R_0}:=\Omega_{f,R_0}=\Omega_{g,R_0}$.
We have $\mu_\Pi(E_f\cap E_g)\geq1-2\eta$.
Let $T:=T_f=T_g$ on $\Omega_{R_0}$, and set
\begin{align*}
L_f:=\int_{E_f}[W^s_{f,R_0}(a)]\,d\mu_\Pi(a),
\qquad
L_g:=\int_{E_g}[W^s_{g,R_0}(b)]\,d\mu_\Pi(b).
\end{align*}
By Lemma~\ref{lem: BJ6.1}, we have
$T|_{\Omega_{R_0}}=L_f+S_f=L_g+S_g$, where $S_f$ and $S_g$ are positive closed
currents on $\Omega_{R_0}$. Since $T\wedge T=0$ on $\Omega_{R_0}$,
positivity implies $L_f\wedge L_g=0.$

By Bedford--Lyubich--Smillie~\cite{BS4} and Dujardin~\cite{DujardinPM04}, this wedge product is geometric:
\begin{align*}
L_f\wedge L_g
=
\int_{E_f\times E_g}
[W^s_{f,R_0}(a)\cap W^s_{g,R_0}(b)]\,d\mu_\Pi(a)\,d\mu_\Pi(b),
\end{align*}
where the intersection is the sum of isolated intersection points counted with multiplicities.
Suppose by contradiction that there exists $a_0\in E_f\cap E_g$ such that
$W^s_{f,R_0}(a_0)\neq W^s_{g,R_0}(a_0)$.
Choose a sufficiently small box $\Omega$ for the $g$-lamination near $a_0$. In holomorphic coordinates $(\zeta,\xi)\in \Dbb\times \tau$ on
$\Omega$, the plaques of the $g$-lamination are
\begin{align*}
W^s_{g,R_0}(b)\cap \Omega=\{\xi=b\},
\qquad b\in \tau\cap E_g.
\end{align*}
After shrinking $\Omega$, the component of $W^s_{f,R_0}(a_0)\cap \Omega$ is the graph of a
holomorphic function $\xi=\varphi(\zeta)$ for $\zeta\in \Dbb.$
Since $W^s_{f,R_0}(a_0)\neq W^s_{g,R_0}(a_0)$, the function $\varphi$ is nonconstant. Hence,
by the open mapping theorem, $\varphi(\Dbb)$ contains a neighborhood $V$ of $a_0$ in $\tau$.

Now, since $\{W^s_{f,R_0}(a)\}_{a\in E_f}$ is a Riemann surface lamination and $E_f$ has no
isolated points, after shrinking $\Omega$ again there exists a relatively open neighborhood
$U$ of $a_0$ in $E_f$ such that for every $a\in U$, the component of
$W^s_{f,R_0}(a)\cap \Omega$ is the graph of a holomorphic function $\xi=\varphi_a(\zeta), \zeta\in \Dbb,$
and $\varphi_a\to\varphi$ uniformly on compact subsets of $\Dbb$ as $a\to a_0$.
Shrinking $U$ and $V$ if necessary, Rouch\'e's theorem implies that
\begin{align*}
W^s_{f,R_0}(a)\cap W^s_{g,R_0}(b)\neq\varnothing
\qquad\text{for every }(a,b)\in U\times (V\cap E_g).
\end{align*}
Since the two relatively open sets have positive $\mu_\Pi$-measure, it follows that
\begin{align*}
\int_{E_f\times E_g}
[W^s_{f,R_0}(a)\cap W^s_{g,R_0}(b)]\,d\mu_\Pi(a)\,d\mu_\Pi(b)>0,
\end{align*}
contradicting $L_f\wedge L_g=0$.
Therefore $W^s_{f,R_0}(a)=W^s_{g,R_0}(a)$ for all $a\in E_f\cap E_g$.
Thus the two local stable manifolds coincide in a neighborhood of each such $a$. Since
$\eta>0$ was arbitrary, this holds for $\mu_\Pi$-almost every $a$.
\end{proof}

\subsection{Stable-disk decompositions in arbitrary dimension}
\label{section:stable-manifold-k}

The preceding proof relies on the geometric intersection theory of laminar
currents. A comparable theorem is not presently available for currents of
arbitrary codimension; see~\cite[Question~1.3]{DujardinICM2022}. We therefore
give an alternative proof valid in every dimension.

For $\mu_\Pi$-almost every $a\in J_\Pi$, Bedford--Jonsson construct an embedded holomorphic disk
$\Wcal_f(a)\subset\Omega_f$, contained in the global stable manifold of $a$,
such that $\Wcal_f(a)\cap\Pi=\{a\}.$
The disk need not be closed, and hence need not be properly embedded, in
$\Omega_f$. 

\begin{thm}[Bedford--Jonsson {\cite[Theorem~6.10]{BedfordJonssonAJM00}}]
\label{thm:BJ6.10}
We have the following geometric representation
\[
   T_f^{k-1}|_{\Omega_f}
   =\int[\Wcal_f(a)]\,d\mu_\Pi(a).
\]
In particular, $\Wcal_f(a)$ has finite area for $\mu_\Pi$-almost every
$a$.
\end{thm}

Let $\omega$ be a Hermitian form on $\Pbb^k$, say the Fubini--Study form. For any positive closed bidimension $(1,1)$-current $T$, we define its \emph{trace measure} by $\|T\|_{\omega}:=T\wedge \omega$.

\begin{lem}\label{lem:moderate-near-infinity}
There exist constants $C,\alpha>0$ such that
$\|T_f^{k-1}\|_{\mathrm{FS}}(\Omega_{f,R})\leq Ce^{-\alpha R}$ for every
sufficiently large $R>0$.
\end{lem}

\begin{proof}
Cover a neighbourhood of $\Pi$ by finitely many coordinate charts in which $\Pi$ is
locally given by $t=0$. Then $G_f=-\log|t|+O(1)$. By
\cite[Corollary~1.3]{DNS10JDG}, the trace measure
$\|T_f^{k-1}\|_{\mathrm{FS}}$ is locally moderate. This means that, since
$\log|t|$ is plurisubharmonic,
\begin{align*}
    \|T_f^{k-1}\|_{\mathrm{FS}}\bigl(\{\log|t|<-M\}\bigr)\le C e^{-\alpha M}
   \qquad 
\end{align*}
for any $M\ge0$ (sufficiently large so that we are in a small compact neighborhood of $\Pi$).
In other words, $\|T_f^{k-1}\|_{\mathrm{FS}}(\Omega_{f,R})\le C e^{-\alpha R}$ for large $R>0$.
\end{proof}

\begin{lem}\label{lem: finite_sum_mass}
For large $R>0$, we have
$\sum_{n\ge0}\|T_f^{k-1}\|_{(f^n)^*\omega}(\Omega_{f,R})<\infty$.
\end{lem}

\begin{proof}
Recall that $f_*T^{k-1}_f=d T_f^{k-1}$ and
$f(\Omega_{f,r})\subset\Omega_{f,dr}$ for any $r>0$.
The positive current $(f^n)_*(\mathbf 1_{\Omega_{f,R}} T^{k-1}_f)$ is supported
on $\Omega_{f,d^nR}$ and is dominated by
$(f^n)_*T^{k-1}_f=d^nT^{k-1}_f$.  Therefore
\[
\begin{aligned}
   \|T_f^{k-1}\|_{(f^n)^*\omega}(\Omega_{f,R})
   &=\left\langle (f^n)_*(\mathbf 1_{\Omega_{f,R}} T^{k-1}_f),\omega\right\rangle  \\
   &\le d^n\|T_f^{k-1}\|_\omega(\Omega_{f,d^nR}).
\end{aligned}
\]
The last term is bounded by $Cd^n e^{-\alpha R d^n}$ by
Lemma~\ref{lem:moderate-near-infinity},
and this is summable in $n$.
\end{proof}

Assume now that the current $T^{k-1}_f|_{\Omega_f}$ has another representation
\[
   T^{k-1}_f|_{\Omega_f}=\int [D_a] d\mu_\Pi(a).
\]
Here the $D_a\subset\Omega_f$ are embedded holomorphic disks satisfying
$a\in D_a$ and $D_a\cap\Pi=\{a\}$.
If an intersection with $\Omega_{f,R}$ is disconnected, we retain the connected
component containing $a$; we use the same convention for
$\Wcal_f(a)\cap\Omega_{f,R}$.

Areas of holomorphic images are counted with multiplicity:
for a holomorphic map $h$ defined on a disk $D$, we set
$\operatorname{Area}_\omega(h(D)):=\int_D h^*\omega$.

\begin{lem}\label{lem:local-stable-area}
For sufficiently large $R>0$, we have
$\operatorname{Area}_\omega\bigl(f^n(D_a\cap\Omega_{f,R})\bigr)\to0$ for
$\mu_\Pi$-a.e. $a$.
\end{lem}

\begin{proof}
We have
\begin{align*}
    \|T_f^{k-1}\|_{(f^n)^*\omega}(\Omega_{f,R})
    &\geq\int \left(\int_{D_a\cap\Omega_{f,R}}(f^n)^*\omega\right)d\mu_\Pi(a)\\
    &= \int \operatorname{Area}_\omega\bigl(f^n(D_a\cap\Omega_{f,R})\bigr) d\mu_\Pi(a)
\end{align*}
Tonelli's theorem and Lemma~\ref{lem: finite_sum_mass} imply that
\[
   \int \sum_{n\ge0}\operatorname{Area}_\omega\bigl(f^n(D_a\cap\Omega_{f,R})\bigr) d\mu_\Pi(a)<\infty.
\]
Therefore, for $\mu_\Pi$-almost every $a$,
\[
   \sum_{n\geq0}\operatorname{Area}_\omega
      \bigl(f^n(D_a\cap\Omega_{f,R})\bigr)<\infty.
\]
In particular,
$\operatorname{Area}_\omega\bigl(f^n(D_a\cap\Omega_{f,R})\bigr)\to0$.
\end{proof}

\begin{lem}\label{lem:stable-manifolds-same-k}
For $\mu_\Pi$-almost every $a\in\Pi$, there is a level $R(a)>0$ such that
\[
   \Wcal_f(a)\cap\Omega_{f,R(a)}=D_a\cap\Omega_{f,R(a)}.
\]
\end{lem}

\begin{proof}
By Lemma~\ref{lem:local-stable-area}, for sufficiently large $R>0$,
$\operatorname{Area}_\omega\bigl(f^n(D_a\cap\Omega_{f,R})\bigr)\to0$.
By the Briend--Duval area--diameter estimate
\cite[Lemma~1.55]{DSbook} (see also the original lemma
\cite[Appendix]{briendduval2}), applied in a local parametrization
of $D_a$ at $a$, there is a disk
$U\subset D_a\cap\Omega_{f,R}$ containing $a$,
independent of $n$, such that
\[
\sup_{z\in U}
\dist\bigl(f^n(z),f_\Pi^n(a)\bigr)\to0.
\]
After shrinking $U$ to control the finitely many initial iterates,
every point of $U$ therefore lies on the local stable manifold
through $a$. Since the disks are one-dimensional embedded complex
submanifolds through $a$, the identity principle shows that they
have the same germ at $a$. Increasing the level once more yields
$\Wcal_f(a)\cap\Omega_{f,R(a)}
=D_a\cap\Omega_{f,R(a)}$.
\end{proof}

\begin{prop}
\label{prop:zariski-dense-common-stable-manifolds}
Let $f,g\in\Poly^k_d$ satisfy $\mu_f=\mu_g$. Then there exists a set
$E_0\subset\Pi$ of full $\mu_\Pi$-measure such that, for every $a\in E_0$,
$f$ and $g$ have the same local stable-manifold germ at $a$.
\end{prop}

\begin{proof}
The Bedford--Jonsson disk representation associated to $g$ is 
another representation of the current
$T_f^{k-1}|_{\Omega_f}$.
Applying Lemma~\ref{lem:stable-manifolds-same-k} with $D_a=\Wcal_g(a)$ gives 
the conclusion.
\end{proof}

\section{Measure rigidity and  finiteness}\label{sect: proof main}

\subsection{Proof of Theorem~\ref{thm: main1}}\label{subsec:proof-main1}

We first fix coordinates near the hyperplane at infinity.  For
$1\le i\le k$, let
\[
   \widetilde U_i:=\{[Z:T]\in\Pbb^k \mid Z_i\ne0\},
   \qquad
   U_i:=\widetilde U_i\cap\Pi.
\]
The open set $\widetilde U_i$ is an affine chart, isomorphic to $\Cbb^k$,
with \emph{tangential coordinates} $z_r^{(i)}:=\frac{Z_r}{Z_i}, r\ne i$ and \emph{transverse coordinate} $t_i:=\frac{T}{Z_i}$.

In this section, we use the ambient chart with $i=k$ and abbreviate $w_r:=\frac{Z_r}{Z_k}$ and $t:=\frac{T}{Z_k}$.
Let $\Dcal_f\subset\Pi$ be the full-$\mu_{f_\Pi}$-measure set of
directions at which the local stable manifold of $f$ is defined.  Define
\begin{equation}\label{eq:fixed-chart-good-set}
   \Dcal_f^{(k)}
   :=\Dcal_f\cap f_\Pi^{-1}(\Dcal_f)
     \cap U_k\cap f_\Pi^{-1}(U_k).
\end{equation}
Since $\mu_{f_\Pi}$ is invariant under
$f_\Pi$, both $\Dcal_f$ and $f_\Pi^{-1}(\Dcal_f)$ have full measure.
Moreover, $\Pi\setminus U_k$ and $f_\Pi^{-1}(\Pi\setminus U_k)$ are proper
algebraic subsets of $\Pi$, and hence have zero $\mu_{f_\Pi}$-measure.
Consequently, $\mu_{f_\Pi}(\Dcal_f^{(k)})=1.$

Let $x\in\Dcal_f^{(k)}$. Write $x=[a:1]$, where
$a=(a_1,\ldots,a_{k-1})\in\Cbb^{k-1}$. In the ambient chart
$\widetilde U_k$, the transversality of the local stable manifold to
$\Pi=\{t=0\}$ gives a unique normalized graph representation
\begin{equation}\label{eq:normalized-stable-parametrization}
   \varphi_{f,x}(t)=[a+h_{f,x}(t):1:t],
   \qquad
   h_{f,x}(t)=O(t).
\end{equation}
The normalized graph in \eqref{eq:normalized-stable-parametrization} is unique
once the affine coordinates are fixed. Later in Section~\ref{sect:generic-first-jet}, we need a more precise (the first-order) expansion of $h_{f,x}$. 

We begin with an elementary vanishing lemma along local stable manifolds.

\begin{lem}
\label{lem:stable-manifold-divisibility}
Let $P(w,t)\in\Cbb[w_1,\ldots,w_{k-1},t]$, let $d\ge1$, and let
$E\subset\Cbb^{k-1}$ be Zariski dense.  Suppose that, for every $a\in E$,
there is a holomorphic map $h_a(t)=O(t)$ near $t=0$ such that $P(a+h_a(t),t)=O(t^d).$
Then $t^d$ divides $P(w,t)$.
\end{lem}

\begin{proof}
Write $P(w,t)=\sum_{m=0}^N t^m P_m(w).$
We show inductively that $P_m=0$ for $0\le m<d$.  Suppose that
$P_0=\cdots=P_{m-1}=0$.  Since $h_a(t)=O(t)$, we have
\[
   P(a+h_a(t),t)=t^mP_m(a)+O(t^{m+1}).
\]
For $m<d$, the hypothesis implies $P_m(a)=0$ for every $a\in E$.
The Zariski density of $E$ gives $P_m\equiv0$.  This proves the induction and
hence $t^d\mid P$.
\end{proof}

\begin{lem}
\label{lem:unitarize-leading-relation}
Let $f,g\in\Poly^k_d$ satisfy $\mu_f=\mu_g$. If $g_h=Af_h$
for some $A\in\operatorname{GL}_k(\Cbb)$, then $A$ is conjugate to a unitary
matrix.
\end{lem}

\begin{proof}
Since $G_h\circ (Af_h)
   =G_h\circ g_h
   =dG_h
   =G_h\circ f_h$,
we have $G_h\circ A^n(z)=G_h(z)$ for every $n\in\Zbb$.
Choose $C>0$ such that
$|G_h(z)-\log\|z\||\leq C$ for $z\neq0$. For every unit vector
$z$ and every $n\in\Zbb$,
\[
   \log\|A^nz\|
   \leq G_h(A^nz)+C
   =G_h(z)+C
   \leq2C.
\]
The closure of $\{A^n\}$ is therefore a compact subgroup of
$\operatorname{GL}_k(\Cbb)$. Hence $A$ is conjugate to a unitary matrix.
\end{proof}

\begin{proof}[Proof of Theorem~\ref{thm: main1}]
Assume that $\mu_f=\mu_g$ and that
$g_h=Af_h$.

By Lemma~\ref{lem:unitarize-leading-relation}, choose
$P\in\operatorname{GL}_k(\Cbb)$ such that $U:=P^{-1}AP$ is unitary, and set $\widetilde f=P^{-1}fP$, $\widetilde g=P^{-1}gP.$
Then $\mu_{\widetilde f}=(P^{-1})_*\mu_f=(P^{-1})_*\mu_g
=\mu_{\widetilde g}$ and
$\widetilde g_h=U\widetilde f_h$. Since a unitary matrix is conjugate to a diagonal matrix,
we may suppose 
that $A=\operatorname{diag}(\alpha_1,\ldots,\alpha_k)$
and that $A$ is unitary.
We now use the first $d-1$ transverse coefficients of the common local stable
manifolds to show that $g-Af$ is constant.

\smallskip
\textbf{Step 1: Common stable manifolds and local coordinates.}
By Proposition~\ref{prop:zariski-dense-common-stable-manifolds},
$\mu_{f_\Pi}=\mu_{g_\Pi}=: \mu_\Pi$, and there is a full-$\mu_\Pi$-measure
set $E_0\subset\Pi$ on which the local stable manifolds of $f$ and $g$
coincide sufficiently close to $\Pi$. By
\eqref{eq:fixed-chart-good-set}, both $\Dcal_f^{(k)}$ and
$\Dcal_g^{(k)}$ have full $\mu_\Pi$-measure. Hence
\[
   E:=E_0\cap\Dcal_f^{(k)}\cap\Dcal_g^{(k)}
\]
has full $\mu_\Pi$-measure and is therefore Zariski dense in $\Pi$.

Let
\[
   \widehat f[Z:T]
   =[\widehat f_1(Z,T):\cdots:\widehat f_k(Z,T):T^d],
   \qquad
   \widehat g[Z:T]
   =[\widehat g_1(Z,T):\cdots:\widehat g_k(Z,T):T^d]
\]
be the projective extensions. Choose homogeneous coordinates
$[Z_1:\cdots:Z_k:T]$ and work in the fixed ambient chart
$\widetilde U_k$.
By the definition of $\Dcal_f^{(k)}$ and $\Dcal_g^{(k)}$, the points
$x$, $f_\Pi(x)$, and $g_\Pi(x)$ all belong to
$U_k\subset\widetilde U_k$ for every $x\in E$, and all the relevant local
stable manifolds are defined.  The affine-coordinate
image
\[
   E_{\mathrm{aff}}
   :=\{a\in\Cbb^{k-1}\mid [a:1]\in E\}
\]
is Zariski dense in $\Cbb^{k-1}$.

Write
\[
   \widehat f[w:1:t]
   =[\Phi_{f,1}(w,t):\cdots:\Phi_{f,k}(w,t):t^d],
\]
\[
   \widehat g[w:1:t]
   =[\Phi_{g,1}(w,t):\cdots:\Phi_{g,k}(w,t):t^d].
\]
The homogeneous-part relation gives
\begin{equation}\label{eq:homogeneous-chart-relation}
   \Phi_{g,i}(w,0)=\alpha_i\Phi_{f,i}(w,0),
   \qquad 1\le i\le k.
\end{equation}
For $x\in E$, write
\[
   x=[a:1],
   \qquad
   y:=f_\Pi(x)=[b:1],
   \qquad
   y':=g_\Pi(x)=[b':1],
\]
where $a,b,b'\in\Cbb^{k-1}$.  Use the normalized parametrization
\eqref{eq:normalized-stable-parametrization}.  Since the two local stable
manifolds coincide, we may write
\begin{equation*}
   \varphi_x:=\varphi_{f,x}=\varphi_{g,x},
   \qquad
   \varphi_x(t)=[a+h_x(t):1:t],
   \qquad
   h_x(t)=O(t).
\end{equation*}
By \eqref{eq:homogeneous-chart-relation},
\begin{equation}\label{eq:target-label-relation}
   b'_j=\frac{\alpha_j}{\alpha_k}b_j,
   \qquad 1\le j\le k-1.
\end{equation}

Parametrize the relevant target local stable manifolds by
\[
   \varphi_{f,y}(u)=[b+h_{f,y}(u):1:u],
   \qquad h_{f,y}(u)=O(u),
\]
\[
   \varphi_{g,y'}(u)=[b'+h_{g,y'}(u):1:u],
   \qquad h_{g,y'}(u)=O(u).
\]

\smallskip
\textbf{Step 2: Divisibility of the projective minors.}
The transverse parameters induced by $f$ and $g$ on the common local stable
manifold are
\begin{equation}\label{eq:transverse-parameters}
   \widetilde f_x(t)
   =\frac{t^d}{\Phi_{f,k}(a+h_x(t),t)},
   \qquad
   \widetilde g_x(t)
   =\frac{t^d}{\Phi_{g,k}(a+h_x(t),t)}.
\end{equation}
The denominators in \eqref{eq:transverse-parameters} are nonzero at $t=0$;
hence $\widetilde f_x(t)=O(t^d)$ and $\widetilde g_x(t)=O(t^d)$.
After shrinking the common local stable manifold at $x$, its images under
$f$ and $g$ are contained in the corresponding target local stable
manifolds. Therefore, for $1\le j\le k-1$ we obtain
\begin{equation}\label{eq:tangential-image-estimates}
\begin{aligned}
   \frac{\Phi_{f,j}(a+h_x(t),t)}
        {\Phi_{f,k}(a+h_x(t),t)}
   &=b_j+h_{f,y,j}(\widetilde f_x(t))
   =b_j+O(t^d),\\
   \frac{\Phi_{g,j}(a+h_x(t),t)}
        {\Phi_{g,k}(a+h_x(t),t)}
   &=b'_j+h_{g,y',j}(\widetilde g_x(t))
   =b'_j+O(t^d).
\end{aligned}
\end{equation}
Using \eqref{eq:target-label-relation}, it follows that
\[
   \alpha_k
   \frac{\Phi_{g,j}(a+h_x(t),t)}
        {\Phi_{g,k}(a+h_x(t),t)}
   -\alpha_j
   \frac{\Phi_{f,j}(a+h_x(t),t)}
        {\Phi_{f,k}(a+h_x(t),t)}
   =O(t^d).
\]
Therefore the polynomial
\begin{equation*}
   M_j(w,t)
   :=\alpha_k\Phi_{g,j}(w,t)\Phi_{f,k}(w,t)
     -\alpha_j\Phi_{f,j}(w,t)\Phi_{g,k}(w,t)
\end{equation*}
satisfies
\[
   M_j(a+h_x(t),t)=O(t^d)
   \qquad (x=[a:1]\in E).
\]
Applying Lemma~\ref{lem:stable-manifold-divisibility} to
$E_{\mathrm{aff}}$ yields
\begin{equation}\label{eq:minor-divisibility}
   t^d\mid M_j(w,t),
   \qquad 1\le j\le k-1.
\end{equation}

Since $M_j$ is the dehomogenization of a homogeneous polynomial of degree
$2d$, homogenizing \eqref{eq:minor-divisibility} gives
\begin{equation}\label{eq:homogeneous-minor-divisibility}
   T^d\mid
   \bigl(\alpha_k\widehat g_j\widehat f_k
         -\alpha_j\widehat f_j\widehat g_k\bigr),
   \qquad 1\le j\le k-1.
\end{equation}

\smallskip
\textbf{Step 3: The affine relation and the converse.}
Put $H_i(Z,T):=\widehat g_i(Z,T)-\alpha_i\widehat f_i(Z,T)$.
By the homogeneous-part relation, $H_i(Z,0)=0$, and hence
\begin{equation}\label{eq:H-expansion}
   H_i(Z,T)=\sum_{m=1}^d T^mH_{i,m}(Z),
   \qquad \deg H_{i,m}=d-m.
\end{equation}
Equation \eqref{eq:homogeneous-minor-divisibility} becomes
\begin{equation}\label{eq:H-minor-divisibility}
   T^d\mid
   \bigl(\alpha_kH_j\widehat f_k-\alpha_j\widehat f_jH_k\bigr),
   \qquad 1\le j\le k-1.
\end{equation}

We prove by induction on $m<d$ that $H_{i,m}=0$ for every $i$.  Assume this has
already been proved for all smaller indices.  The coefficient of $T^m$ in
\eqref{eq:H-minor-divisibility} is
\begin{equation}\label{eq:coefficient-syzygy}
   \alpha_kH_{j,m}(f_h)_k
   -\alpha_j(f_h)_jH_{k,m}=0,
   \qquad 1\le j\le k-1.
\end{equation}
Set $V_{i,m}:=\alpha_i^{-1}H_{i,m}$.  Then
\eqref{eq:coefficient-syzygy} says $ V_{j,m}(f_h)_k=(f_h)_jV_{k,m}.$
Consequently, over the fraction field of $\Cbb[Z_1,\ldots,Z_k]$, there is a
rational function $q_m$ such that $(V_{1,m},\ldots,V_{k,m})
   =q_m\bigl((f_h)_1,\ldots,(f_h)_k\bigr).$
Write $q_m=p_m/r_m$ in lowest terms. Since
$q_m(f_h)_i=V_{i,m}$ is a polynomial for every $i$, the denominator $r_m$
divides every $(f_h)_i$. The regularity of $f$ means that
$(f_h)_1,\ldots,(f_h)_k$ have no common zero on $\Pbb^{k-1}$; in particular,
they have no nonconstant common factor. Hence $r_m$ is constant and $q_m$ is
a polynomial. If $q_m\ne0$, then for every $i$ with $(f_h)_i\ne0$ the polynomial
$q_m(f_h)_i=V_{i,m}$ has degree at least $d$, whereas $V_{i,m}$ is homogeneous
of degree $d-m<d$.  This is impossible.  Hence $q_m=0$.  Therefore
$H_{i,m}=0$ for every $i$.  The induction is complete.

It follows from \eqref{eq:H-expansion} that $H_i(Z,T)=c_iT^d$ for
constants $c_i\in\Cbb$.  Hence
\[
   \widehat g_i(Z,T)=\alpha_i\widehat f_i(Z,T)+c_iT^d,
   \qquad 1\le i\le k.
\]
On the affine chart $T=1$, this is precisely $g(z)=Af(z)+c$, where
$c=(c_1,\ldots,c_k)$. Now set $   \sigma(z):=Az+c.$
Then $g=\sigma f$.

It remains to show that $\sigma$ preserves the Julia set. By
Corollary~\ref{cor:mu-implies-green}, $G_g=G_f$. Since $g=\sigma f$, we have
$G_f\circ\sigma=G_f$.  Consequently,
\[
   \sigma^*\mu_f
   =\bigl(\ddc(G_f\circ\sigma)\bigr)^k
   =(\ddc G_f)^k
   =\mu_f,
\]
and hence $\sigma(J_f)=J_f$.  The converse implication is
Proposition~\ref{prop:easy-direction}.
\end{proof}

\begin{cor}\label{cor:coordinate-powers}
Let $f\in\Poly_d^k$, where $d,k\geq2$. Suppose that
\[
   f_h(z)=\bigl(a_1z_{\pi(1)}^d,\ldots,a_kz_{\pi(k)}^d\bigr),
   \qquad a_i\ne0,\quad \pi\in S_k.
\]
Then, for every $g\in\Poly_d^k$,
\[
   \mu_f=\mu_g
   \quad\Longleftrightarrow\quad
   g=\sigma f\quad\text{for some }\sigma\in\Sa_f.
\]
\end{cor}

For example,
$f(x,y)=(x^d+y,y^d+x)$ satisfies the hypothesis, although it is not a
product map.

\begin{proof}[Proof of Corollary~\ref{cor:coordinate-powers}]
A diagonal change of coordinates allows us to assume $a_i=1$ for every $i$.
In these coordinates, $\|f_h(z)\|_\infty=\|z\|_\infty^d$, so
$G_{f_h}(z)=\log\|z\|_\infty$ for $z\ne0$. Let $g\in\Poly_d^k$
and suppose that $\mu_f=\mu_g$. Put $H:=g_h$.
Corollary~\ref{cor:mu-implies-green} gives $G_H=G_{f_h}$. Using
$G_H\circ H=dG_H$, we obtain
\[
   \max_i|H_i(z)|=\left(\max_j|z_j|\right)^d
   \qquad\text{for every }z\in\Cbb^k.
\]

Let $e_j$ be the $j$th coordinate vector. The preceding identity gives
$|H_i(e_j)|=1$ for some $i$, so the coefficient of $z_j^d$ in $H_i$ has
absolute value $1$. On the other hand, $|H_i|\leq1$ on the unit
polydisc. Write $H_i(z)=\sum_{|\alpha|=d}c_\alpha z^\alpha$.
Integrating on the unit torus $\Tbb^k=(S^1)^k$ with its normalized
Haar measure $m$ gives
\[
   \sum_{|\alpha|=d}|c_\alpha|^2
   =\int_{\Tbb^k}|H_i|^2\,dm\leq1.
\]
Thus all coefficients other than that of $z_j^d$ vanish, and
$H_i=c_i z_j^d$ with $|c_i|=1$. We obtain
\[
   g_h(z)=\bigl(c_1z_{\tau(1)}^d,\ldots,c_kz_{\tau(k)}^d\bigr),
   \qquad |c_i|=1,\quad \tau\in S_k.
\]
This gives $g_h=Af_h$ for some invertible linear map $A$.
Theorem~\ref{thm: main1} then gives $g=\sigma f$ for some
$\sigma\in\Sa_f$. The converse is Proposition~\ref{prop:easy-direction}.
\end{proof}

\begin{prop}\label{prop:C2-rigidity-applications}
Let $f,g\in\Poly^2_d$ for some $d\geq 2$. Suppose that one of the
following conditions holds:
\begin{enumerate}
\item[(a)] $f_\Pi$ and $g_\Pi$ are polynomials in the same coordinate;
\item[(b)] $f_\Pi$ is polynomial and is not conjugate to the power map.
\end{enumerate}
Then $\mu_f=\mu_g$ if and only if  $g=\sigma f$
for some $\sigma\in\Sa_f$.
\end{prop}

\begin{proof}
The reverse implication follows from Proposition~\ref{prop:easy-direction}.
Assume $\mu_f=\mu_g$.

\smallskip
\noindent(a)
Up to the same conjugation on $f$ and $g$, we may assume that $f_\Pi$
is centered and that
\begin{align*}
    f_h(x,y)=a_f\left(y^df_\Pi\!\left(\frac{x}{y}\right),y^d\right),\qquad
    g_h(x,y)=a_g\left(y^dg_\Pi\!\left(\frac{x}{y}\right),y^d\right)
\end{align*}
for some $a_f,a_g\in\Cbb^*$. By Theorem~\ref{thm: Beardon},
$g_\Pi=\phi f_\Pi$ for an affine symmetry
$\phi(z)=\lambda z+b$ of $J_{f_\Pi}$. Since $f_\Pi$ is centered, $b=0$. Hence
\[
   g_h=Af_h,
   \qquad
   A:=\frac{a_g}{a_f}\operatorname{diag}(\lambda,1).
\]
We conclude by Theorem~\ref{thm: main1}.

\smallskip
\noindent(b)
Since $\mu_{f_\Pi}=\mu_{g_\Pi}$, it follows from
Okuyama--Stawiska~\cite{YusukeMagorzata} that $g_\Pi$ is also a polynomial with
respect to the same exceptional point at infinity. Thus we are reduced to case~(a).
\end{proof}

\begin{ex}\normalfont\label{ex:leading-system-hypothesis}
The hypothesis $g_h=Af_h$ in Theorem~\ref{thm: main1} cannot be omitted,
even for maps that are not affinely conjugate to homogeneous maps.
The following construction is based on Ye's example~\cite[\S~3]{Ye2015}.
Let $\omega\ne1$ satisfy $\omega^3=1$. For $c\in\Cbb$, put
$P_c(x,y):=x^3-3xy^2+c$ and define
\[
\begin{aligned}
   f_c(x,y)&:=\bigl(P_c(x,y)^2+y^6,\;P_c(x,y)y^3\bigr),\\
   g_c(x,y)&:=\bigl(\omega P_c(x,y)^2+\omega^2y^6,\;P_c(x,y)y^3\bigr).
\end{aligned}
\]
Their leading coordinates have no common zero on $\Pbb^1$, so
$f_c,g_c\in\Poly^2_6$. A direct computation gives
$f_cf_c=f_cg_c$, which implies $G_{f_c}=G_{g_c}$, and hence
$\mu_{f_c}=\mu_{g_c}$.

However, $g_c\ne\sigma f_c$ for every affine automorphism $\sigma$.
Indeed, such a relation would imply
$(g_c)_\Pi=\sigma_\Pi(f_c)_\Pi$,
contrary to Ye's example~\cite[\S~3]{Ye2015}.

For $c=0$, both maps are homogeneous. For $c\ne0$, neither is affinely
conjugate to a homogeneous map. To see this, let $u$ be either $f_c$ or
$g_c$. Such a conjugacy would give a fixed point $p$ for which
$u(p+z)-p$ is homogeneous of degree $6$. Since $u$ has no terms of degree
$5$, comparison of these terms gives $Du_h(z)p\equiv0$.
Euler's identity gives $6u_h(p)=Du_h(p)p=0$, so regularity forces $p=0$.
But $f_c(0)=(c^2,0)$ and $g_c(0)=(\omega c^2,0)$ are
nonzero when $c\ne0$, contradicting the fixed-point condition.
Thus $c=0$ gives homogeneous counterexamples, while every $c\ne0$
gives counterexamples that are not affinely conjugate to homogeneous maps.
\end{ex}

\subsection{Proofs of Theorems~\ref{thm: main2} and~\ref{thm:main 6}}
We have the restriction map
\begin{equation*}
    \rho_f^{\mathrm{aff}}:\Sa_f
    \longrightarrow\Aut(\Pi),
    \qquad \sigma\longmapsto\sigma_\Pi.
\end{equation*}
Thus $\Sigma_f^0=\ker\rho_f^{\mathrm{aff}}$.

\begin{proof}[Proof of Theorem~\ref{thm: main2}]
Let $\ell_f:\Sa_f\to\operatorname{GL}_k(\Cbb)$ denote the homomorphism
taking linear parts. This homomorphism is injective since $J_f$ is compact. 
Every element of $\Sigma_f^0$ has scalar linear part. Thus, for
$\tau\in\Sigma_f^0$ and $\sigma\in\Sa_f$, we have
$\ell_f(\tau\sigma)=\ell_f(\sigma\tau)$, so injectivity gives
$\tau\sigma=\sigma\tau$.
Hence $\Sigma_f^0$ is contained in the center of $\Sa_f$.

If $f=L\circ h\circ L^{-1}$ for an affine automorphism $L$ and a
homogeneous map $h$, then $R_\lambda(z):=\lambda z$ preserves $J_h$ for
every $\lambda\in S^1$. Thus $L\circ R_\lambda\circ L^{-1}\in\Sigma_f^0$,
so $\Sigma_f^0$ is infinite.

Assume that $f$ is not affinely conjugate to a homogeneous map.
If $\Sigma_f^0=\{\id\}$, the assertion follows. Otherwise, choose
$\tau(z)=\lambda z+a\in\Sigma_f^0\setminus\{\id\}$.
Injectivity of $\ell_f$ gives $\lambda\ne1$,
so $\tau$ has a unique fixed point $p$. Since every element of $\Sa_f$
commutes with $\tau$, it fixes $p$. Conjugating by the translation
$z\mapsto z-p$, we may therefore assume that every element of
$\Sigma_f^0$ is $R_\lambda(z):=\lambda z$.

For $R_\lambda\in\Sigma_f^0$, we have  $\mu_{fR_\lambda}=\mu_f$ and $(fR_\lambda)_h=\lambda^d f_h$, Theorem~\ref{thm: main1} gives
$fR_\lambda=R_\lambda^d f$. Thus $f(\lambda z)=\lambda^d f(z)$.
Write $f=\sum_{j=0}^d f_j$ in these coordinates. Since $f$ is not
affinely conjugate to a homogeneous map, some $f_j$ with $0\leq j<d$
is nonzero. Comparing the degree-$j$ terms gives
$\lambda^{d-j}=1$.
Consequently, $\Sigma_f^0$ is a finite cyclic group of order at most
$d-j\leq d$.

The other implication is clear.
\end{proof}

We next establish a criterion for a regular polynomial endomorphism to be
homogeneous, which may be of independent interest.

\begin{lem}\label{lem: charcterizationHomogeneous}
Let $f\in \Poly^k_d$ for some $d\ge 2$. Then its
Julia set $J_f$ is $S^1$-invariant (or equivalently, $\lambda J_f=J_f$ for
infinitely many $\lambda \in S^1$) if and only if $f$ is homogeneous.
\end{lem}

\begin{proof}
The stabilizer $\{\lambda\in S^1\mid \lambda J_f=J_f\}$
is a closed subgroup of $S^1$. If it is infinite, it equals $S^1$, which
justifies the equivalence in the statement. The implication from right to left
is immediate. For the converse, assume that $J_f$ is $S^1$-invariant. Since
$K_f=\widehat{J_f}$, the filled Julia set $K_f$ is also $S^1$-invariant. Hence
$G_f(\lambda z)=G_f(z)$ for every $z\in \Cbb^k$ and every $\lambda\in S^1$.

For any $\lambda\in S^1$, define
$f_\lambda(z):=\lambda^{-d}f(\lambda z)$. Then $f_\lambda$ and $f$ have the
same homogeneous part, and
\[
   G_f\circ f_\lambda(z)
   =G_f\bigl(\lambda^{-d}f(\lambda z)\bigr)
   =dG_f(z).
\]
This gives $G_{f_\lambda}=G_f$. By Theorem~\ref{thm: main1}, $f_\lambda=f$ for every $\lambda\in S^1$. Hence $f$ is homogeneous.
\end{proof}

We recall the one-dimensional finiteness theorem used to control the image of
the restriction map.

\begin{thm}[Levin {\cite[Theorem~2]{LevinSymmetries90}}]\label{thm: Levin}
Fix $d_1,d_2\geq2$, and let $r$ be a rational map of degree $d_1$.  The set of
rational maps of degree $d_2$ sharing the equilibrium measure of $r$ is finite
unless $r$ is conjugate to $z\mapsto z^{\pm d_1}$.
\end{thm}

\begin{cor}\label{cor:finite-symmetries-C2}
Let $f\in\Poly^2_d$ be a map that is not affinely conjugate to a
homogeneous polynomial map. If $f_\Pi$ is not conjugate to
$z\mapsto z^{\pm d}$, then $\Sa_f$ and $\Sigma_f^{e\mathrm{poly}}$
are finite for every $e\geq2$.
\end{cor}

\begin{proof}
If $g\in \Sigma_f^{e\mathrm{poly}}$, then $\mu_{g_\Pi}=\mu_{f_\Pi}$. Thus $\rho(\Sigma_f^{e\mathrm{poly}})$ is finite by Theorem~\ref{thm: Levin}.
Fix $r\in \rho(\Sigma_f^{e\mathrm{poly}})$, and choose
$g_0\in \Sigma_f^{e\mathrm{poly}}$ with $(g_0)_\Pi=r$. Since
$\mu_{g_0}=\mu_f$, we have $J_{g_0}=J_f$. The map $g_0$ is not affinely
conjugate to a homogeneous polynomial map: otherwise, after the same affine
conjugacy, Lemma~\ref{lem: charcterizationHomogeneous} and the equality of the
Julia sets would imply that $f$ is affinely conjugate to a homogeneous map as
well. By Theorem~\ref{thm: main1}, every element of
\[
   \{g\in\Sigma_f^{e\mathrm{poly}}\mid g_\Pi=r\}
   =\{g\in\Sigma_{g_0}^{e\mathrm{poly}}\mid g_\Pi=(g_0)_\Pi\}
\]
is of the form $\sigma g_0$ for some $\sigma\in\Sigma_{g_0}^0$.
Theorem~\ref{thm: main2}, applied to $g_0$,
therefore shows that this fiber is finite. Thus $\Sigma_f^{e\mathrm{poly}}$
is finite. In particular, for
$e=d$, the injective map $\sigma\mapsto\sigma f$ shows that
$\Sa_f$ is finite as well.
\end{proof}

\begin{proof}[Proof of Theorem~\ref{thm:main 6}]
Assume that $\Sa_f$ is infinite and that $f$ is not affinely conjugate
to a homogeneous polynomial map. By Corollary~\ref{cor:finite-symmetries-C2},
$f_\Pi$ is conjugate to $z\mapsto z^d$ or $z\mapsto z^{-d}$.
We may
assume that
\[
   f_h(x,y)=(x^d,y^d)\qquad\text{or}\qquad f_h(x,y)=(y^d,x^d).
\]

Every $\sigma\in\Sa_f$ fixes the barycenter $b_f:=\int z\,d\mu_f(z)$.
Translating $b_f$ to the origin makes every element of $\Sa_f$ linear. Since $G_f\circ\sigma=G_f$, we get
$\|\sigma z\|_\infty=\|z\|_\infty$.
Such a matrix is diagonal or antidiagonal, with its nonzero entries
of modulus one.
Consequently, $\Sa_f$ is a compact Lie group. Its identity component $T$ has
positive dimension and its elements are all diagonal. Thus $T\subset(S^1)^2$. It cannot have dimension two,
since then $T=(S^1)^2$ would contain the scalar circle, contradicting
the finiteness of $\Sigma_f^0$ in Theorem~\ref{thm: main2}.
It follows that
\[
   T=\{R_\lambda:=\operatorname{diag}(\lambda^r,\lambda^s)
       \mid\lambda\in S^1\},
   \qquad \gcd(r,s)=1,
\]
for integers $r,s$. Moreover, $r\ne s$, again by the finiteness of
$\Sigma_f^0$.

We first exclude $f_h=(y^d,x^d)$. For every $\lambda\in S^1$,
 Theorem~\ref{thm: main1} gives
\[
   fR_\lambda=A_\lambda f,
   \qquad A_\lambda=\operatorname{diag}(\lambda^{ds},\lambda^{dr})
   \in\Sa_f.
\]
The continuous homomorphism $\lambda\mapsto A_\lambda$ has image in $T$.
Since $\lambda\mapsto R_\lambda$ is an isomorphism, there is
an integer $q$ such that $A_\lambda=R_{\lambda^q}$.
Hence $ds=qr$ and $dr=qs$, which imply $r=-s$.
We may thus assume $(r,s)=(1,-1)$. This implies that $f=(y^d,x^d)$ is homogeneous, a contradiction.

We are left with $f_h=(x^d,y^d)$. The same application of
Theorem~\ref{thm: main1} now gives
\[
   fR_\lambda=R_{\lambda^d}f\qquad(\lambda\in S^1).
\]
We may assume $r>s$ and $r>0$. This implies that
$f(x,y)=(x^d,Q(x,y))$ is a regular polynomial skew product.

Every nonzero monomial $x^a y^b$ of $Q$ satisfies
\[
   ra=s(d-b),\qquad a+b\leq d.
\]
If $s<0$, the left side is nonnegative and the right side is nonpositive,
so $a=0$ and $b=d$.
Then $Q=y^d$, contradicting the assumption that $f$ is not affinely
conjugate to a homogeneous map.
If $s=0$, then $a=0$ and $Q=p(y)$ for a monic polynomial $p$ of
degree $d$. Interchanging the coordinates gives assertion~(2).
The polynomial $p$ cannot be affinely conjugate to $z^d$, since otherwise
$f$ would be affinely conjugate to a homogeneous map.
Finally, suppose $s>0$. Since $r$ and $s$ are coprime, we have $a=sj$
and $b=d-rj$ for an integer $0\leq j\leq\lfloor d/r\rfloor$.
Since $Q_h=y^d$, we can write
\begingroup
\postdisplaypenalty=10000
\[
   Q(x,y)=y^dP\!\left(\frac{x^s}{y^r}\right)
\]
\endgroup
for some $P\in\Cbb[t]$ with $P(0)=1$ and $\deg P\leq\lfloor d/r\rfloor$.
The polynomial $P$ is nonconstant, since otherwise $f=(x^d,y^d)$ would
be homogeneous. In particular, $r\leq d$, which gives assertion~(3).

Conversely, for $\lambda\in S^1$, set
\[
   R_\lambda(x,y):=
   \begin{cases}
      (\lambda x,\lambda y),&\text{in case (1)},\\
      (x,\lambda y),&\text{in case (2)},\\
      (\lambda^r x,\lambda^s y),&\text{in case (3)}.
   \end{cases}
\]
In each case, $fR_\lambda=R_{\lambda^d}f$ and $G_f\circ R_\lambda=G_f$.
Hence $R_\lambda\in\Sa_f$ for every $\lambda\in S^1$.
Thus $\Sa_f$ is infinite.
\end{proof}

When the group $\Sa_f$ is finite, its  structure can still be complicated.
\begin{ex}[Finite $\Sa_f$ containing the Klein four group ]\normalfont\label{ex:klein-four-centralizer}
Consider the regular polynomial
\begin{align*}
f(x,y):=\bigl(x^3+xy^2+x,\ y^3+2x^2y+y\bigr).
\end{align*}
In the affine coordinate $u=x/y$ on
$\Pi\simeq \Pbb^1$, we have $f_\Pi(u)=u\frac{u^2+1}{2u^2+1}.$
It admits $0$ as a fixed parabolic point, so $f_\Pi$ is nonspecial. Moreover,
$f$ is not affinely conjugate to a homogeneous polynomial, since no fixed
point of $f$ has zero differential. Corollary~\ref{cor:finite-symmetries-C2} implies that
$\Sa_f$ is finite. Now define
$\tau_1(x,y)=(-x,y)$ and $\tau_2(x,y)=(x,-y)$.
A direct computation shows that $V_4=\{\id,\tau_1,\tau_2,\tau_1\tau_2\}\subset C(f)$. Thus $\Sa_f$ need not be cyclic. 
\end{ex}

\begin{ex}[Finite $\Sa_f$ containing the dihedral group $D_4$ ]\normalfont\label{ex:dihedral-centralizer}
Consider the regular polynomial
\begin{align*}
f(x,y):=\bigl(x^3+2xy^2+x,\ y^3+2x^2y+y\bigr).
\end{align*}
In the affine coordinate $u=x/y$ on
$\Pi\simeq \Pbb^1$, we have $f_\Pi(u)=u\frac{u^2+2}{2u^2+1}.$
It is not of monomial type since it has four
distinct simple critical values. It is not Latt\`es since it has two attracting fixed points $\pm 1$.
Moreover, $f$ is not affinely conjugate to a homogeneous polynomial, since no
fixed point of $f$ has zero differential. Corollary~\ref{cor:finite-symmetries-C2} implies
that $\Sa_f$ is finite.
Finally, define $s(x,y)=(-y,x)$ and $t(x,y)=(x,-y).$
A direct computation shows $D_4=\langle s,t\rangle\subset C(f)$.
\end{ex}

\section{Further measure rigidity in \texorpdfstring{$\Cbb^2$}{C²}}\label{sect: first-consequences}
\subsection{Functional relations and preperiodicity of the diagonal}\label{sect:iterated-functional-relations}
We begin with the one-dimensional characterization that controls the image at
infinity.

\begin{thm}[Levin~{\cite[Theorem~3]{LevinSymmetries90}},
Levin--Przytycki~{\cite[Theorem~A]{LevinPrzytycki97}}]\label{thm: Levin-P}
Let $r$ and $s$ be nonspecial rational maps of degrees at least $2$.  Then
$\mu_r=\mu_s$ if and only if there exist iterates $R$ and $S$ of $r$ and $s$,
respectively, such that
\[
   RR=RS
   \qquad\text{and}\qquad
   S R=S S.
\]
\end{thm}

\begin{proof}[Proof of the equivalence of $(1)$--$(3)$ in Theorem~\ref{thm: LP-C2}]
We first prove $(1)\Rightarrow(2)$. Assume that $\mu_f=\mu_g$.
Set $r:=f_\Pi$ and $s:=g_\Pi$.
Since $\mu_r=\mu_s$ and $r$ is nonspecial, $s$ is nonspecial;
see~\cite[\S~1.2]{JiXie_LocalRigidity}.
Theorem~\ref{thm: Levin-P} gives integers
$a,b\geq1$ such that, with $r_0:=r^a$ and $s_0:=s^b$, one has
$r_0r_0=r_0s_0$ and $s_0r_0=s_0s_0$.
Taking degrees gives $d^a=e^b$.

Put $u:=f^a$ and $v:=g^b$. By induction,
$r_0s_0^n=r_0^{n+1}$ for every $n\geq1$.
Theorem~\ref{thm: main1}, applied to $u^{n+1}$ and
$uv^n$, gives $\alpha_n\in\Sigma_f^0$ such that $uv^n=\alpha_nu^{n+1}$.
Similarly, there is $\beta_n\in\Sigma_f^0$ such that
$vu^n=\beta_nv^{n+1}$.

By Theorem~\ref{thm: main2}, the group
$\Sigma_f^0$ is finite. Hence the sequence
$(\alpha_n,\beta_n)_{n\geq 1}$
takes only finitely many values, and some value occurs infinitely often.
We may therefore choose $1\leq m<n$ with $n\geq2m+1$ such that
$\alpha_m=\alpha_n$ and $\beta_m=\beta_n$.
Set $p:=n-m$ and $q:=m+1$, so that $p\geq q$.
Composing $uv^m=\alpha_mu^{m+1}$ on the right by $v^p$ gives
$uv^n=\alpha_mu^{m+1}v^p$.
Since $\alpha_m=\alpha_n$, cancellation gives
$u^qv^p=u^{q+p}$. Similarly, $v^qu^p=v^{q+p}$.
Composing these identities on the left by $u^{p-q}$ and $v^{p-q}$,
respectively, gives
\[
   u^{2p}=u^pv^p
   \qquad\text{and}\qquad
   v^{2p}=v^pu^p.
\]
Replacing $a,b$ by $ap,bp$ proves $(2)$.

We next prove $(2)\Rightarrow(3)$. Put $u:=f^a$, $v:=g^b$, and
$\Psi:=u\times v$. Since $u^2=uv$ and $v$ is surjective on $\Pbb^2$,
\[
   \Psi^2(\Delta)
   =\{(u(v(z)),v(v(z)))\mid z\in\Pbb^2\}
   =\Psi(\Delta).
\]

Finally, suppose $(3)$ holds. Put
\[
   u:=f^a,\qquad v:=g^b,\qquad D:=\deg u=\deg v,
   \qquad\text{and}\qquad \Gamma_j:=(u\times v)^j(\Delta).
\]
By logarithmic growth,
$M:=\sup_{z\in\Cbb^2}|G_u(z)-G_v(z)|<\infty$.
Set $\Gamma_j^{\mathrm{aff}}:=\Gamma_j\cap(\Cbb^2\times\Cbb^2)$.
The Green function equations therefore give
\[
   \sup_{(x,y)\in\Gamma_j^{\mathrm{aff}}}|G_u(x)-G_v(y)|
   =\sup_{z\in\Cbb^2}|G_u(u^j(z))-G_v(v^j(z))|=D^jM.
\]
Since $\Delta$ is preperiodic, we have
$\Gamma_m^{\mathrm{aff}}=\Gamma_n^{\mathrm{aff}}$ for some $0\leq m<n$. This gives
$D^mM=D^nM$, hence $M=0$. Thus $G_u=G_v$, hence $G_f=G_g$ and
consequently $\mu_f=\mu_g$.
This proves $(3)\Rightarrow(1)$.

If $d=e$, comparing degrees in $(2)$ gives $a=b$. The proof of
$(2)\Rightarrow(3)$ then gives
$(f\times g)^{2a}(\Delta)=(f\times g)^a(\Delta)$, so $\Delta$ is
preperiodic under $f\times g$. Conversely, preperiodicity under
$f\times g$ is condition $(3)$ with $a=b=1$.
\end{proof}

\subsection{Generic rigidity in dimension two}
In this subsection we prove the two-dimensional case of Theorem~\ref{thm:generic-measure-rigidity}.
We first isolate the one-dimensional inputs.

\begin{thm}[Pakovich]\label{thm: pakovich-measure}
For every $d\geq2$, there exists a Zariski dense open subset
$\Omega_d\subset\End_d(\Pbb^1)$, defined over $\Qbb$, such
that the following assertions hold for every
$r\in\Omega_d(\Cbb)$ and every rational map $s$ of degree at least
$2$. When $d=2$, denote by $\iota_r$ the unique nontrivial deck involution
satisfying $r\iota_r=r$.
    \begin{enumerate}
        \item If $d=2,$ then $\mu_s=\mu_r
   \Longrightarrow
   s=r^n$ or $s=\iota_r r^n$ for some $n\geq1$.
        \item If $d\ge 3$, then $\mu_s=\mu_r
   \Longrightarrow
   s=r^n$ for some $n\ge 1$.  
    \end{enumerate}
\end{thm}
\begin{proof}
    The case $d\ge 4$ is proved by Pakovich in \cite[Theorem~1.3]{pakovich-deg4}. He then extended the main ingredients of the proof to $d=2$ and $d=3$ in~\cite{pakovich2026periodiccurvesgeneralendomorphisms}. Since the conclusion for the case $d=2$ is different and is not explicitly stated in~\cite{pakovich2026periodiccurvesgeneralendomorphisms}, we provide here the argument for the convenience of the reader. It is based on Pakovich's new results in~\cite{pakovich2026periodiccurvesgeneralendomorphisms} and is slightly different from the argument for the case $d\ge 4$ in~\cite{pakovich-deg4}. 

    Let $r$ be a general quadratic rational map and let $s$ be a rational map of degree at least $2$ with $\mu_s=\mu_r$.  By Theorem~\ref{thm: Levin-P}, there exist $a,b\geq1$ such that $r^a r^a=r^a s^b$. Apply \cite[Theorem~5.8]{pakovich2026periodiccurvesgeneralendomorphisms} with
\begin{align*}
   A=r,
   \qquad X=r^a,
   \qquad B=s^b,
   \qquad m=a.
\end{align*}
 We obtain $r^a=r^\ell\nu$ and $s^b=\nu^{-1} r^a\nu$.
Comparing degrees in $r^a=r^\ell\nu$ gives $\ell=a$, hence $\nu\in \{\id, \iota_r\}$ by \cite[Lemma~3.4 and Corollary~3.7]{pakovich2026periodiccurvesgeneralendomorphisms}.
Taking degrees in $s^b=\nu^{-1} r^a\nu$ gives $\deg s=2^n$ for some $n\geq1$
and $a=nb$.
Put $\widetilde s:=\nu s\nu^{-1}.$ Applying \cite[Theorem~5.8]{pakovich2026periodiccurvesgeneralendomorphisms} once again to $r^{nb}\widetilde s=\widetilde sr^{nb}$
gives $\widetilde s=r^n\alpha$ and $r^{nb}=\alpha^{-1} r^{nb}\alpha$.
The second identity implies $\alpha=\id$ by \cite[Corollary~3.10]{pakovich2026periodiccurvesgeneralendomorphisms}.
Thus $\widetilde s=r^n$.  Undoing the conjugation gives $s=\nu^{-1} r^n\nu.$
If $\nu=\id$, then $s=r^n$.  If $\nu=\iota_r$, then $s=\iota_rr^n\iota_r=\iota_r r^n .$

Let $U\subset\End_d(\Pbb^1_{\Cbb})$ be a nonempty Zariski open subset
of nonspecial maps on which the conclusions hold. By
Theorem~\ref{thm: Levin-P}, equality of maximal-entropy measures for
nonspecial maps is equivalent to algebraic identities between iterates.
Since specialness and the quadratic deck involution are also preserved
under field automorphisms, the conclusions hold on every $U^\tau$, where
$\tau\in\Aut(\Cbb/\Qbb)$. Thus $\Omega_d:=\bigcup_{\tau\in\Aut(\Cbb/\Qbb)}U^\tau$
is a nonempty invariant Zariski open subset on which the conclusions hold,
and hence is defined over~$\Qbb$.
\end{proof}

We next establish the generic fiberwise input needed to lift this
one-dimensional result.

\begin{lem}\label{lem:generic-trivial-affine-kernel}
For every $d\geq2$, there exists a Zariski dense open subset
$\Ucal_{d,\mathrm{ker}}\subset\Poly^2_d$, defined over $\Qbb$, such that $\Sigma_f^0=\{\id\}$
for every $f\in\Ucal_{d,\mathrm{ker}}(\Cbb)$.
\end{lem}

\begin{proof}
We use dimension count arguments.
Fix $d\geq2$. We show that the locus of maps
$f\in\Poly^2_d$ for which $\Sigma_f^0$ is nontrivial
is contained in a proper algebraic subset.

Suppose that $\Sigma_f^0$ is nontrivial. By the fixed-point
argument in the proof of Theorem~\ref{thm: main2}, there
is a translation $T_p(z)=z-p$ such that $\widetilde f:=T_pfT_{-p}$ has
a nontrivial scalar symmetry $z\mapsto\lambda z$, with
$\lambda\in S^1\setminus\{1\}$. The calculation in the proof of
Lemma~\ref{lem: charcterizationHomogeneous} gives
$\widetilde f(\lambda z)=\lambda^d\widetilde f(z)$.
Comparing the degree-$(d-1)$ parts yields
$\lambda^{d-1}\widetilde f_{d-1}=\lambda^d\widetilde f_{d-1}$.
Since $\lambda\ne1$, we obtain $(T_pfT_{-p})_{d-1}=0$.

Consider the incidence subset
\[
   \Zcal_d:=\{(u,p)\in\Poly^2_d\times\Abb^2
      \mid (T_puT_{-p})_{d-1}=0\}.
\]
For fixed $p$, writing $u=\sum_{m=0}^d u_m$ gives
\[
   (T_puT_{-p})_{d-1}(z)=u_{d-1}(z)+Du_d(z)p.
\]
The vector space of degree-$(d-1)$ homogeneous polynomial maps
$\Cbb^2\to\Cbb^2$ has dimension $2d$. Consequently,
\[
   \dim\Zcal_d
   \leq \dim\Poly^2_d+2-2d
   <\dim\Poly^2_d.
\]
Let $\Bcal_d:=\overline{\operatorname{pr}_1(\Zcal_d)}^{\,\mathrm{Zar}}
\subsetneq\Poly^2_d$. Hence
\[
   \Ucal_{d,\mathrm{ker}}
      :=\Poly^2_d\setminus\Bcal_d
\]
is the required Zariski dense open subset defined over $\Qbb$.
\end{proof}

\begin{lem}\label{lem:deck-free-open}
There exists a Zariski dense open subset
$\Ucal_{2,\mathrm{deck}}\subset\Poly^2_2$, defined over $\Qbb$, such that,
for every $f\in\Ucal_{2,\mathrm{deck}}(\Cbb)$, there is no nontrivial affine
involution $\sigma\in\Aff(\Cbb^2)$ with nontrivial action on $\Pi$
satisfying $f\sigma=f$.
\end{lem}

\begin{proof}
We use dimension count arguments.
Let $\Pcal_{\leq2}$ be the affine space, defined over $\Qbb$, of polynomial
maps $\Abb^2\to\Abb^2$ of degree at most two. Thus
\[
\Pcal_{\leq 2}(\Cbb)=
   \left\{
      u=(u_1,u_2):\Cbb^2\to\Cbb^2
      \mid
      \deg u_i\leq2
   \right\}.
\]
This vector space of polynomial maps has dimension $12$.
The space $\Poly^2_2$ of regular ones is a Zariski dense open subset of
$\Pcal_{\leq 2}$.

An affine involution $\sigma$ with nontrivial action on $\Pi$ has a fixed point.
After translating it to the origin, we may suppose that $\sigma$ is linear.
It has eigenvalues $1$ and $-1$, and we may thus suppose that
$\sigma(x,y)=(x,-y)$.

The vector space $\Pcal_{\leq2}^{\sigma}:=
   \{u\in\Pcal_{\leq2}(\Cbb)\mid u\sigma=u\}$
has dimension $8$.
The affine group $\Aff(\Cbb^2)$ has dimension 6.
The centralizer of $\sigma$ in $\Aff(\Cbb^2)$ has dimension $3$.
Let $\Ccal_\sigma$ denote the affine conjugacy class of $\sigma$. Then
\begin{align*}
   \dim\Ccal_\sigma
   =\dim \Aff(\Cbb^2)-\dim C_{\Aff(\Cbb^2)}(\sigma)=6-3=3.
\end{align*}

We now form the locally closed incidence subset
\begin{align*}
   \Ical_{\mathrm{deck}}:=
   \left\{
      (u,\tau)\in \Pcal_{\leq2}\times\Aff(\Cbb^2)
      \mid
      \tau^2=\id,\ \tau_\Pi\neq\id,\ u\tau=u
   \right\}
\end{align*}
which is defined over $\Qbb$. Let $\pi_1$ and $\pi_2$ be the two
projections.
The projection $\pi_2$ has image $\Ccal_\sigma$ and fibers of dimension $8$. Indeed, if
$\tau=\gamma\sigma\gamma^{-1}$, then $u\tau=u$ is equivalent to
$(u\gamma)\sigma=u\gamma$. Thus precomposition by $\gamma$ identifies
$\pi_2^{-1}(\tau)$ with $\Pcal_{\leq2}^{\sigma}$. Consequently,
\[
   \dim\Ical_{\mathrm{deck}}
   =\dim\Ccal_\sigma+\dim\Pcal_{\leq2}^{\sigma}
   =3+8=11.
\]

Since $\dim\Pcal_{\leq2}=12$, put $\Dcal_{\mathrm{deck}}
      :=\overline{\pi_1(\Ical_{\mathrm{deck}})}^{\,\mathrm{Zar}}
      \subsetneq\Pcal_{\leq2}$.
Therefore
\[
   \Ucal_{2,\mathrm{deck}}
      :=\Poly^2_2\setminus
        \bigl(\Dcal_{\mathrm{deck}}\cap\Poly^2_2\bigr)
\]
is the required Zariski dense open subset defined over $\Qbb$.
\end{proof}

\begin{proof}[Proof of 
Theorem~\ref{thm:generic-measure-rigidity} when $k=2$]
The restriction-to-infinity morphism
\[
   \rho\colon
   \Poly^2_d\longrightarrow\End_d(\Pbb^1),
   \qquad f\longmapsto f_\Pi,
\]
is surjective and defined over $\Qbb$. For $d\geq3$, set
\[
   \Ucal_{d,2}
      :=\Ucal_{d,\mathrm{ker}}
        \cap\rho^{-1}(\Omega_d),
\]
and, when $d=2$, set
\[
   \Ucal_{2,2}
      :=\Ucal_{2,\mathrm{ker}}
        \cap\rho^{-1}(\Omega_2)
        \cap\Ucal_{2,\mathrm{deck}}.
\]
Choose $f\in\Ucal_{d,2}(\Cbb)$. By
Lemma~\ref{lem:generic-trivial-affine-kernel}, we have
$\Sigma_f^0=\{\id\}$. Let $g\in \Poly_e^2$ for some $e\ge 2$ such that $\mu_f=\mu_g$. Put $r=f_\Pi$ and $s=g_\Pi$. Then
$\mu_s=\mu_r$. By Theorem~\ref{thm: pakovich-measure}, if $d=2$, there is
$n\geq1$ such that either $s=r^n$ or $s=\iota_rr^n$; if $d\geq3$,
there is $n\geq1$ such that $s=r^n$.

Suppose first that $s=r^n$. Then $\deg g=\deg f^n$ and
$g_\Pi=(f^n)_\Pi$. Theorem~\ref{thm: main1},
applied to $f^n$ and $g$, gives
$g=\sigma f^n$ for some $\sigma\in\Sigma_{f^n}^0$. Since
$J_{f^n}=J_f$, we have
$\Sigma_{f^n}^0=\Sigma_f^0=\{\id\}$. Thus
$g=f^n$.

Suppose now that $d=2$ and $s=\iota_rr^n$. Choose and scale a linear
lift $A$ of $\iota_r$ so that $g_h=A(f^n)_h.$
Applying Theorem~\ref{thm: main1} directly gives $ g=\sigma f^n$
for some $\sigma\in\Sa_f$ satisfying
$\sigma_\Pi=\iota_r$.

Since $G_f\circ\sigma=G_f$, the maps $f\sigma$ and $f$ have the same
equilibrium measure. At infinity,
\[
   (f\sigma)_\Pi=r\iota_r=r=f_\Pi.
\]
Theorem~\ref{thm: main1} therefore gives
$f\sigma=\eta f$ for some $\eta\in\Sigma_f^0$. Thus
$\eta=\id$ and $f\sigma=f$. Also
$\sigma^2\in\Sigma_f^0$, so $\sigma^2=\id$. Hence $\sigma$ is a
nontrivial affine involution acting nontrivially on $\Pi$, contrary to
Lemma~\ref{lem:deck-free-open}. The second case is therefore impossible.
\end{proof}

\subsection{Preperiodic sets}
Yuan--Zhang~\cite{YZhodge17} (see also
Baker--DeMarco~\cite{BakerDeMarco2011} in dimension one) proved that equality
of preperiodic sets implies equality of equilibrium measures.  Our method
gives the converse.

\begin{thm}\label{thm:same-preper}
Let $f\in\Poly^2_d$, where $d\geq2$. Then $\Sa_f$ is finite if and only
if, for every $e\geq2$ and every $g\in\Poly^2_e$,
\[
   \mu_f=\mu_g
   \quad\Longleftrightarrow\quad
   \Prep(f)=\Prep(g).
\]
\end{thm}

The proof uses notation from Section~\ref{sect:intro-tits} and
a result recalled in Section~\ref{sect: tits}.
\begin{proof}
The implication from equality of preperiodic sets to equality of measures
always holds by Yuan--Zhang~\cite{YZhodge17}.

Suppose first that $\Sa_f$ is finite, and put $M:=\#\Sa_f$.
Let $g\in\Poly_e^2$, $e\geq2$, satisfy $\mu_g=\mu_f$.
Set $\Scal:=\langle f,g\rangle$ and $S:=\{f,g\}$.
By Corollary~\ref{cor:mu-implies-green}, $G_f=G_g=:G$.
Every $h\in\Scal$ satisfies $G\circ h=(\deg h)G$, so it has Green
function $G$. In particular, $\mu_h=\mu_f$ and $\Sa_h=\Sa_f$.

Assume first that $f_\Pi$ is not conjugate to $z\mapsto z^{\pm d}$.
All elements of $\rho(\Scal)$ have the same equilibrium measure and hence
the same preperiodic set; see~\cite[Remark~2]{LevinPrzytycki97}.
Proposition~\ref{prop: BHPT4.10} therefore gives polynomially bounded
growth of $\rho(\Scal)$.
If $h_1,h_2\in\Scal$ have the same restriction to $\Pi$, Theorem~\ref{thm: main1}
gives $h_2=\sigma h_1$ for some $\sigma\in\Sa_f$.
Thus every fiber of $\rho|_\Scal$ has cardinality at most $M$, and
\[
   d_S(n)\leq M\,d_{\{f_\Pi,g_\Pi\}}(n).
\]
It follows that $\Scal$ has polynomially bounded growth.

Now assume that $f_\Pi$ is conjugate to $z\mapsto z^{\pm d}$.
After a linear change of coordinates, we may assume
\[
   f_h(x,y)=(x^d,y^d)\qquad\text{or}\qquad f_h(x,y)=(y^d,x^d).
\]
For any $h\in\Scal$ of degree $D$, $\|h_h(z)\|_\infty=\|z\|_\infty^D.$
The proof of Corollary~\ref{cor:coordinate-powers} shows that
\[
h_h(x,y)=(c_1x^D,c_2y^D)
\quad \mbox{or} \quad
h_h(x,y)=(c_1y^D,c_2x^D),
\]
where $|c_1|=|c_2|=1$.
Thus any two elements $h_1,h_2\in\Scal$ of the same degree
satisfy $(h_2)_h=A(h_1)_h$ for some $A\in\mathrm{GL}_2(\Cbb)$.
Theorem~\ref{thm: main1} therefore gives $h_2=\sigma h_1$
for some $\sigma\in\Sa_{h_1}=\Sa_f$. Hence there are at most $M$ elements of $\Scal$
of any given degree. 
Thus $\Scal$ again has polynomially bounded growth.
In both cases, Theorem~\ref{thm: BHPT1.3} implies
$\Prep(f)=\Prep(g)$.

Conversely, suppose that $\Sa_f$ is infinite.
By Theorem~\ref{thm:main 6}, after an affine
conjugacy, $f$ has one of the three normal forms listed there.
Let $R_\lambda$, $\lambda\in S^1$, be the corresponding symmetries
defined in the proof of Theorem~\ref{thm:main 6}. Then
$fR_\lambda=R_{\lambda^d}f$ and $R_\lambda\in\Sa_f$.

There is a fixed point $a$ of $f$ such that $R_\lambda(a)=a$
for only finitely many $\lambda\in S^1$. In case (1), take a nonzero fixed point, obtained by
lifting a fixed point of $f_\Pi$ and rescaling. In case (2), take
$a=(a_0,1)$ with $p(a_0)=a_0$. In case (3), writing $f=(x^d,Q)$,
take $a=(1,b_0)$ with $Q(1,b_0)=b_0$.
Choose $\lambda\in S^1$ that is not a root of unity and set
$g:=R_\lambda f$. Proposition~\ref{prop:easy-direction} gives
$\mu_g=\mu_f$, whereas
\[
   g^n(a)=R_{\lambda^{1+d+\cdots+d^{n-1}}}(a)
\]
are pairwise distinct.
Thus $a\in\Prep(f)\setminus\Prep(g)$, which proves the converse.
\end{proof}

\section{Admissible homogeneous systems}
\label{sect:generic-first-jet}

The case $k=2$ of Theorem~\ref{thm:generic-measure-rigidity} was proved in
Section~\ref{sect: first-consequences} by means of the image--fiber
principle. We now turn to the remaining case $k\geq3$, whose proof occupies
Sections~\ref{sect:generic-first-jet}--\ref{sect:generic-stabilizer}.

\subsection{First-order behavior of local stable manifolds}

Fix $f\in\Poly^k_d$, and write $F:=f_h=f_d$ and $B:=f_{d-1}$ for
its two leading homogeneous terms. Put $V:=\Cbb^k$ and identify
$\Pbb(V\oplus\Cbb)$ with $\Pbb^k$.
Recall the standard affine charts $\widetilde U_i$ and their boundary parts
$U_i$ introduced at the beginning of Section~\ref{sect: proof main}, together
with the tangential coordinates $z_r^{(i)}=Z_r/Z_i$ and the transverse coordinate
$t_i=T/Z_i$.

Let $x\in\Pbb(V)$ be a direction at which the local stable manifold of $f$
is defined, and choose $i$ such that $x\in U_i$.  In $\widetilde U_i$ we
normalize by $Z_i=1$.  If
$\widetilde x^{(i)}\in V$ is the representative of $x$ whose $i$-th
coordinate is $1$, the local stable manifold has a unique normalized
expansion
\[
   [\widetilde x^{(i)}
      +t_i\widetilde\xi_f^{(i)}(x)+O(t_i^2):t_i],
   \qquad
   (\widetilde\xi_f^{(i)}(x))_i=0.
\]
Writing $\ell_x\subset V$ for the line represented by $x$, define
\[
   j_f^{(i)}(x):=[\widetilde\xi_f^{(i)}(x)]\in V/\ell_x.
\]

\begin{lem}
\label{lem:intrinsic-stable-jet}
The class $j_f^{(i)}(x)$ is well-defined and independent of the standard ambient affine chart
$\widetilde U_i$. Thus we denote it simply by $j_f(x)$.
\end{lem}

\begin{proof}
It is enough to check compatibility on the overlaps of these standard
ambient affine charts.  Suppose that $x\in U_i\cap U_r$.  In the
$i$-th normalization, put
\[
   c:=(\widetilde x^{(i)})_r\ne0,
   \qquad
   b:=(\widetilde\xi_f^{(i)}(x))_r.
\]
The $r$-th tangential coordinate along the local stable manifold is
$c+bt_i+O(t_i^2)$.  Passing from the normalization $Z_i=1$ to $Z_r=1$
therefore gives
\[
\begin{aligned}
   t_r
   &=\frac{t_i}{c+bt_i+O(t_i^2)}
     =\frac{t_i}{c}+O(t_i^2),\\
   \widetilde x^{(r)}+t_r\widetilde\xi_f^{(r)}(x)+O(t_r^2)
   &=\frac{\widetilde x^{(i)}
      +t_i\widetilde\xi_f^{(i)}(x)+O(t_i^2)}
        {c+bt_i+O(t_i^2)}.
\end{aligned}
\]
Since $t_i=ct_r+O(t_r^2)$, expansion in the new transverse coordinate yields
\[
   \widetilde x^{(r)}=\frac{\widetilde x^{(i)}}{c},
   \qquad
   \widetilde\xi_f^{(r)}(x)
      =\widetilde\xi_f^{(i)}(x)-\frac{b}{c}\widetilde x^{(i)}.
\]
In particular, $\widetilde x^{(i)}$ and $\widetilde x^{(r)}$ span the same
line $\ell_x$, while the two first coefficients differ by an element of
$\ell_x$.  They consequently define the same class in $V/\ell_x$.
Since the $U_i$ cover $\Pbb(V)$, the local classes glue to a well-defined
class $j_f(x)$.
\end{proof}

We now return to the fixed ambient chart $\widetilde U_k$. Recall that
$\Dcal_f^{(k)}$, defined in
\eqref{eq:fixed-chart-good-set}, is the full-$\mu_{f_\Pi}$-measure set on
which both $x$ and $f_\Pi(x)$ lie in $U_k$ and the relevant local stable
manifolds are defined.

\begin{lem}
\label{lem:first-stable-jet}
Let $x\in\Dcal_f^{(k)}$ and write $x=[a:1]$. With $F=f_h$ and
$B=f_{d-1}$ as above, abbreviate
$\widetilde x:=\widetilde x^{(k)}=(a,1)$ and
$\widetilde\xi_f(x):=\widetilde\xi_f^{(k)}(x)$.
Then there is a scalar $\lambda_f(x)\in\Cbb$ such that
\begin{equation}
   DF(\widetilde x)\widetilde\xi_f(x)+B(\widetilde x)
      =\lambda_f(x)F(\widetilde x).
   \label{eq:first-stable-radial-equation}
\end{equation}
If $\det DF(\widetilde x)\ne0$, then
\begin{equation}\label{eq:jfx}
    j_f(x)
 =\left[-DF(\widetilde x)^{-1}B(\widetilde x)\right]
 =\left[-\frac{\operatorname{adj}(DF(\widetilde x))B(\widetilde x)}
                 {\det DF(\widetilde x)}\right]
 \quad\text{in }V/\Cbb\widetilde x. 
\end{equation}

\end{lem}

\begin{proof}
The projective extension of $f$ has the form 
\[
 \widehat f[Z:T]
 =\bigl[F(Z)+TB(Z)+O(T^2):T^d\bigr].
\]
Using
\eqref{eq:normalized-stable-parametrization}, write the local stable manifold
through $[\widetilde x:0]$ as $[\widetilde x+t\widetilde\xi_f(x)+O(t^2):t].$
Its image is therefore $[Y(t):t^d]$, where
\[
 Y(t)=F(\widetilde x)
 +t\bigl(DF(\widetilde x)\widetilde\xi_f(x)+B(\widetilde x)\bigr)
 +O(t^2).
\]
Since $F_k(\widetilde x)\ne0$, normalization gives
the new transverse parameter
\[
 \tau=\frac{t^d}{Y_k(t)}=u(t)t^d,
 \qquad u(0)\ne0.
\]
Set $v_0:=F(\widetilde x)$ and $v_1:=DF(\widetilde x)\widetilde\xi_f(x)+B(\widetilde x).$
Write $y:=f_\Pi(x)=[b:1]$ in $U_k$.  The image lies in the target local stable
manifold, whose normalized graph has the form
\[
[b+h_{f,y}(\tau):1:\tau] \qquad h_{f,y}(\tau)=O(\tau).
\]
Hence, for
$1\le j<k$,
\[
 \frac{Y_j(t)}{Y_k(t)}
 =\frac{(v_0)_j}{(v_0)_k}+O(\tau)
 =\frac{(v_0)_j}{(v_0)_k}+O(t^d).
\]
Since $d\ge2$, comparison with the first-order expansion
\[
 \frac{Y_j(t)}{Y_k(t)}
 =\frac{(v_0)_j}{(v_0)_k}
 +t\frac{(v_0)_k(v_1)_j-(v_0)_j(v_1)_k}{(v_0)_k^2}+O(t^2)
\]
shows that $v_1=\lambda_f(x)v_0$, where
$\lambda_f(x)=(v_1)_k/(v_0)_k$. Consequently,
\[
   DF(\widetilde x)\widetilde\xi_f(x)+B(\widetilde x)
      =\lambda_f(x)F(\widetilde x).
\]
If moreover $DF(\widetilde x)$ is
invertible, Euler's identity
$DF(\widetilde x)\widetilde x=dF(\widetilde x)$ gives
\[
\widetilde\xi_f(x)
      =-DF(\widetilde x)^{-1}B(\widetilde x)
        +\frac{\lambda_f(x)}{d}\,\widetilde x.
\]
Passing to $V/\Cbb\widetilde x$ proves \eqref{eq:jfx}.
\end{proof}

\begin{rem}\label{rem:rational-stable-jet-section}
The pointwise formula in Lemma~\ref{lem:first-stable-jet} has a global
interpretation. The map
\[
   [Z]\longmapsto\left[-DF(Z)^{-1}B(Z)\right]\in V/\Cbb Z
\]
defines a rational section of
$T\Pbb(V)\otimes\Ocal_{\Pbb(V)}(-1)$. Indeed, since $F$ and $B$ are
homogeneous, $-DF(Z)^{-1}B(Z)$ is a degree-zero rational vector-valued function,
and its class in $V/\Cbb Z$ depends only on $[Z]$.  The Euler sequence
\[
   0\longrightarrow\Ocal_{\Pbb(V)}(-1)
   \longrightarrow V\otimes\Ocal_{\Pbb(V)}
   \longrightarrow T\Pbb(V)\otimes\Ocal_{\Pbb(V)}(-1)
   \longrightarrow0
\]
identifies the fiber of $T\Pbb(V)\otimes\Ocal_{\Pbb(V)}(-1)$ at $[Z]$
with $V/\Cbb Z$.
\end{rem}

\subsection{The fiberwise admissibility condition}
\label{subsec:fiberwise-admissibility}

We now regard $F$ and $B$ as independent parameters.
For $m\geq0$, let $S_m:=H^0\bigl(\Pbb^{k-1},\Ocal_{\Pbb^{k-1}}(m)\bigr)$
be the space of homogeneous polynomials of degree $m$ on $V$, and put $\mathbf G_{d,k}:=\operatorname{Gr}(k,S_d).$
Let
\[
 \Fcal_{d,k}:=\left\{F=(F_1,\ldots,F_k)\in S_d^{k} \mid
 F_1,\ldots,F_k\text{ have no common zero on }\Pbb^{k-1}\right\}.
\]
This is a Zariski dense open subset of $S_d^k$ whose points are ordered
basepoint-free degree-$d$ tuples. We now fix $F\in\Fcal_{d,k}$ and
$B\in S_{d-1}^k$. Set
$L_F:=\langle F_1,\ldots,F_k\rangle\in\mathbf G_{d,k}$.

Use linear coordinates
$Z=(Z_1,\ldots,Z_k)$ on $V$, and put
\[
   \Delta_F:=\det DF,
   \qquad
   \eta_{F,B}:=\operatorname{adj}(DF)B,
   \qquad
   D_{d,k}:=k(d-1).
\]
Our convention is that $F$, $B$, and $\eta_{F,B}$ are column
vectors. We have $\Delta_F\in S_{D_{d,k}}$, and
$(\eta_{F,B})_i\in S_{D_{d,k}}.$

On $\{\Delta_F\ne0\}$, define 
\begin{equation*}
 \Theta_{F,B}(Z):=-\frac{\eta_{F,B}(Z)}{\Delta_F(Z)}
 =-\frac{\operatorname{adj}(DF(Z))B(Z)}{\det DF(Z)},
 \qquad
 \theta_{F,B}([Z]):=[\Theta_{F,B}(Z)]\in V/\Cbb Z.
\end{equation*}
Lemma~\ref{lem:first-stable-jet} says that, for every regular
polynomial endomorphism $f$ with $f_h=F$ and $f_{d-1}=B$, we have
$j_f(x)=\theta_{F,B}(x)$ for every
$x\in\Dcal_f^{(k)}\cap\{\Delta_F\ne0\}$.

For $P,Q\in S_d$, define the polynomial
\begin{equation}\label{eq:first-order-numerator}
   M^{F,B}_{P,Q}:=(P\nabla Q-Q\nabla P)\eta_{F,B}=\sum_{r=1}^k
 \left(P\frac{\partial Q}{\partial Z_r}
       -Q\frac{\partial P}{\partial Z_r}\right)(\eta_{F,B})_r
       \in S_{D_{d,k}+2d-1},
\end{equation}
where $\nabla P=(\partial_1P,\ldots,\partial_kP)$ denotes the row gradient
of $P$, and similarly for $Q$.
It is bilinear and alternating in $P,Q$.  The rational lift
$\Theta_{F,B}$ defines the directional derivative
\[
 \partial_{F,B}P:=DP(Z)\bigl(\Theta_{F,B}(Z)\bigr)
 =\nabla P(Z)\Theta_{F,B}(Z).
\]
Adding a radial vector field to the lift does not affect the alternating
expression below, since
$(P\nabla Q-Q\nabla P)Z=dPQ-dQP=0$.  More precisely,
\begin{equation}\label{eq:first-order-ratio-numerator}
 Q\partial_{F,B}P-P\partial_{F,B}Q
 =\frac{M^{F,B}_{P,Q}}{\Delta_F}.
\end{equation}

Thus the quotient in \eqref{eq:first-order-ratio-numerator} depends only on
the quotient class $\theta_{F,B}$.

\begin{defi}[$(F,B)$-admissibility]
\label{def:FB-admissible}
A $k$-plane $L\subset S_d$ is \emph{$(F,B)$-admissible} if
\begin{equation}
 \Delta_F\mid M^{F,B}_{P,Q}
   \qquad\text{for every }P,Q\in L.   \label{eq:FB-admissibility}
\end{equation}
\end{defi}

\begin{lem}
\label{lem:tautological-solution}
For every $F\in\Fcal_{d,k}$ and every $B\in S_{d-1}^k$, the plane
$L_F$ is $(F,B)$-admissible.
\end{lem}

\begin{proof}
The adjugate identity and the fact that the $j$-th row of $DF$ is
$\nabla F_j$ give
\[
\begin{aligned}
 DF\,\eta_{F,B}
    &=DF\operatorname{adj}(DF)B=\Delta_FB,\\
 \nabla F_j\,\eta_{F,B}
    &=(\Delta_FB)_j=\Delta_FB_j.
\end{aligned}
\]
Consequently, for all $i,j$,
\begin{align*}
 M^{F,B}_{F_i,F_j}
 =(F_i\nabla F_j-F_j\nabla F_i)\eta_{F,B}
 =F_i\Delta_FB_j-F_j\Delta_FB_i
 =\Delta_F(F_iB_j-F_jB_i).
\end{align*}
Bilinearity therefore gives
$\Delta_F\mid M^{F,B}_{P,Q}$ for all $P,Q\in L_F$.
\end{proof}

\begin{prop}
\label{prop:measure-implies-admissible}
Let $f,g\in\Poly^k_d$, and put $F:=f_h$ and $B:=f_{d-1}$. If
$\mu_f=\mu_g$, then
$L_{g_h}:=\langle(g_h)_1,\ldots,(g_h)_k\rangle$ is
$(F,B)$-admissible.
\end{prop}

\begin{proof}
Put $C:=g_{d-1}$. By
Proposition~\ref{prop:zariski-dense-common-stable-manifolds}, the local stable
manifolds of $f$ and $g$ coincide on a set $E_0\subset\Pi$ of full measure
for their common induced equilibrium measure. The sets
$\Dcal_f^{(k)}$ and $\Dcal_g^{(k)}$ also have full measure.
Since this measure does not charge proper
algebraic subsets and $\Delta_F\not\equiv0$, the set
$E:=E_0\cap\Dcal_f^{(k)}\cap\Dcal_g^{(k)}
\cap\{\Delta_F\ne0\}$ is Zariski dense in
$\Pi$.

Fix $x=[a:1]\in E$ and put
$\widetilde x:=(a,1)\in V$. The two local stable manifolds coincide as germs
and are normalized in the same fixed ambient chart $\widetilde U_k$. Hence
they have a common
normalized graph
\[
   [a+h_x(t):1:t],
   \qquad
   h_x(t)=t\xi(x)+O(t^2).
\]
Set $\widetilde\xi(x):=(\xi(x),0)\in V$. We first compare the radial
equations associated with this common graph. Applying
\eqref{eq:first-stable-radial-equation} to $g$, we obtain a scalar
$\lambda_g(x)$ such that
\begin{equation}
   Dg_h(\widetilde x)\widetilde\xi(x)+C(\widetilde x)
   =\lambda_g(x)g_h(\widetilde x).
   \label{eq:candidate-radial-equation}
\end{equation}
The second conclusion of
Lemma~\ref{lem:first-stable-jet}, applied to $f$ on
$\{\Delta_F\ne0\}$, gives a scalar $\nu(x)$ such that
\[
   \widetilde\xi(x)
   =\Theta_{F,B}(\widetilde x)+\nu(x)\widetilde x.
\]
Euler's identity $Dg_h(Z)Z=dg_h(Z)$ changes
\eqref{eq:candidate-radial-equation} into
\[
   Dg_h(\widetilde x)\Theta_{F,B}(\widetilde x)+C(\widetilde x)
   =\Lambda(x)g_h(\widetilde x),
   \qquad \Lambda(x):=\lambda_g(x)-d\nu(x).
\]
Thus, for $P=\sum_i\gamma_i(g_h)_i\in L_{g_h}$, put
$C_P:=\sum_i\gamma_iC_i$. Then
\[
   (\partial_{F,B}P)(\widetilde x)+C_P(\widetilde x)
   =\Lambda(x)P(\widetilde x).
\]
We now eliminate the common scalar $\Lambda(x)$ by alternating. For
$P,Q\in L_{g_h}$, we obtain
\begin{equation*}
   Q\partial_{F,B}P-P\partial_{F,B}Q=PC_Q-QC_P,
\end{equation*}
whose right-hand side is polynomial. By
\eqref{eq:first-order-ratio-numerator}, the left-hand side is
$M^{F,B}_{P,Q}/\Delta_F$. Since $E$ is Zariski dense, the equality holds
as an identity of rational functions.
Thus $\Delta_F\mid M^{F,B}_{P,Q}$ for all $P,Q\in L_{g_h}$, so
$L_{g_h}$ is $(F,B)$-admissible.
\end{proof}

\section{The power-map fiber and generic uniqueness of the admissible system}
\label{sect:generic-admissible-uniqueness}

The purpose of this section is to prove Corollary~\ref{cor:leading-system-uniqueness}.
Throughout this section, $d\ge2$ and $k\ge3$. We retain the notation
$V$, $S_m$, $\mathbf G_{d,k}$, $\Fcal_{d,k}$, and $D_{d,k}=k(d-1)$ from the preceding
section. We begin by globalizing the fiberwise admissibility condition of
Definition~\ref{def:FB-admissible}.
Define $\Qcal_{d,k}:=\Fcal_{d,k}\times S_{d-1}^k$. The map
$f\mapsto(f_h,f_{d-1})$ is a surjective morphism
from $\Poly^k_d$ to $\Qcal_{d,k}$. Thus every
Zariski dense open subset constructed below pulls back to a Zariski dense
open subset of the polynomial parameter space.

These parameter spaces are defined over $\Qbb$. Indeed,
$\Fcal_{d,k}$ is defined by the nonvanishing of the resultant, which has
integer coefficients, and $\Qcal_{d,k}=\Fcal_{d,k}\times S_{d-1}^k$.

\subsection{The admissibility incidence scheme}

Admissibility in \eqref{eq:FB-admissibility} is equivalent to the vanishing
of $[M^{F,B}_{P,Q}]$ in
$S_{D_{d,k}+2d-1}/(\Delta_FS_{2d-1})$. As $(F,B)$ varies, multiplication by
$\Delta_F$ gives a morphism of trivial vector bundles
\begin{equation*}
 \operatorname{mult}_\Delta:S_{2d-1}\otimes\Ocal_{\Qcal_{d,k}}
 \longrightarrow S_{D_{d,k}+2d-1}\otimes\Ocal_{\Qcal_{d,k}},
 \qquad R\longmapsto \Delta_FR.
\end{equation*}
Since $\Delta_F\not\equiv0$ for every $F\in\Fcal_{d,k}$, this morphism is
injective on every fiber and hence has constant rank. Its cokernel is
therefore a vector bundle, denoted by $\Ecal_{\Qcal}$; its fiber at $(F,B)$ is
$(\Ecal_{\Qcal})_{(F,B)}\simeq
S_{D_{d,k}+2d-1}/(\Delta_FS_{2d-1})$.
Let $\Ucal_{\mathbf G}\subset S_d\otimes\Ocal_{\mathbf G_{d,k}}$
be the tautological rank-$k$ subbundle whose fiber at $L$ is
$L$.

Put $X:=\Qcal_{d,k}\times\mathbf G_{d,k}$,
write $\operatorname{pr}_{\Qcal}:X\to\Qcal_{d,k}$ and
$\operatorname{pr}_{\mathbf G}:X\to\mathbf G_{d,k}$ for the projections,
and set $\Ucal:=\operatorname{pr}_{\mathbf G}^*\Ucal_{\mathbf G}$
and $\Ecal:=\operatorname{pr}_{\Qcal}^*\Ecal_{\Qcal}$.

Equation~\eqref{eq:first-order-numerator} defines a vector-bundle morphism
\begin{equation*}
 \alpha:\bigwedge\nolimits^2\Ucal
 \longrightarrow\Ecal,
 \qquad
 P\wedge Q\longmapsto[M^{F,B}_{P,Q}].
\end{equation*}
At a complex point $\zeta=((F,B),L)$, one has
\[
 \alpha(\zeta)=0
 \quad\Longleftrightarrow\quad
 [M^{F,B}_{P,Q}]=0\text{ for all }P,Q\in L
 \quad\Longleftrightarrow\quad
 L\text{ is $(F,B)$-admissible}.
\]

Set $\Hcal:=\Hcal\!om(\bigwedge^2\Ucal,\Ecal)$ and regard
$\alpha$ as a global section of $\Hcal$. Denote by $\Iscr$ the
closed zero scheme of $\alpha$. Its projection
$\pi:\Iscr\to\Qcal_{d,k}$ is proper because $\mathbf G_{d,k}$ is
projective.

For $(F,B)\in\Qcal_{d,k}(\Cbb)$, denote the scheme-theoretic fiber by
\[
 \Acal(F,B):=
 \Iscr\times_{\Qcal_{d,k},\,(F,B)}
 \Spec\Cbb.
\]
Since zero schemes commute with base change, this is the zero scheme on
$\mathbf G_{d,k}$ of
\[
 \alpha_{F,B}:\bigwedge\nolimits^2\Ucal_{\mathbf G}
 \longrightarrow
 \bigl(S_{D_{d,k}+2d-1}/(\Delta_FS_{2d-1})\bigr)
 \otimes\Ocal_{\mathbf G_{d,k}}.
\]

Set
\[
 F^{\mathrm{pow}}:=(Z_1^d,\ldots,Z_k^d),\qquad
 B^{\mathrm{pow}}:=(Z_2^{d-1},\ldots,Z_k^{d-1},Z_1^{d-1}),
\]
where the components of $B^{\mathrm{pow}}$ are indexed cyclically.
Put
\[
 q_{\mathrm{pow}}:=(F^{\mathrm{pow}},B^{\mathrm{pow}}),\qquad
 L_{\mathrm{pow}}:=\langle Z_1^d,\ldots,Z_k^d\rangle,\qquad
 \Acal_{\mathrm{pow}}:=\Acal(F^{\mathrm{pow}},B^{\mathrm{pow}}).
\]
The cyclic choice of $B^{\mathrm{pow}}$ is designed so that
$Z_i\nmid B_i^{\mathrm{pow}}$ for every $i$.

\subsection{The power-map equations and torus degeneration}

For $P,Q\in S_d$, write
\[
 H_i(P,Q):=P\,\partial_iQ-Q\,\partial_iP.
\]

\begin{lem}
\label{lem:power-equations}
Let $F=F^{\mathrm{pow}}$. For any $B\in S_{d-1}^k$ satisfying
$Z_i\nmid B_i$ for every $i$, a complex point
$L\in\mathbf G_{d,k}(\Cbb)$ is $(F,B)$-admissible if and only if
\begin{equation}
 Z_i^{d-1}\mid H_i(P,Q)
 \qquad(P,Q\in L,\ 1\le i\le k),
 \label{eq:power-divisibility}
\end{equation}
\end{lem}

\begin{proof}
For the power-map tuple,
\[
\begin{aligned}
 DF&=d\,\operatorname{diag}(Z_1^{d-1},\ldots,Z_k^{d-1}),\\
 \Delta_F&=d^k\prod_{j=1}^kZ_j^{d-1},\qquad
 (\eta_{F,B})_i=d^{k-1}B_i\prod_{j\ne i}Z_j^{d-1}.
\end{aligned}
\]
Consequently,
\begin{equation}
 M^{F,B}_{P,Q}
 =d^{k-1}\sum_{i=1}^k
 H_i(P,Q)B_i\prod_{j\ne i}Z_j^{d-1}.
 \label{eq:power-numerator}
\end{equation}
Fix $i$.  Modulo $Z_i^{d-1}$, every summand in
\eqref{eq:power-numerator} except the $i$-th vanishes. Since
$Z_i\nmid B_i$, the factor $B_i\prod_{j\ne i}Z_j^{d-1}$
is not divisible by $Z_i$. It may therefore be canceled, and
\[
 Z_i^{d-1}\mid M^{F,B}_{P,Q}
 \quad\Longleftrightarrow\quad
 Z_i^{d-1}\mid H_i(P,Q),
\]
for every $i$. Since the powers
$Z_1^{d-1},\ldots,Z_k^{d-1}$ are pairwise coprime, these simultaneous
divisibilities are equivalent to
$\Delta_F\mid M^{F,B}_{P,Q}$. This proves the lemma.
\end{proof}

\begin{lem}
\label{lem:compatible-monomials}
Assume $k\ge3$. The only $k$-plane in $S_d$ that is
spanned by monomials and satisfies
\eqref{eq:power-divisibility} is
$L_{\mathrm{pow}}=\langle Z_1^d,\ldots,Z_k^d\rangle$.
\end{lem}

\begin{proof}
Let $Z^{\mathbf a}$ and $Z^{\mathbf b}$ be distinct degree-$d$ monomials
satisfying \eqref{eq:power-divisibility}. If $a_i\ne b_i$, then
$H_i(Z^{\mathbf a},Z^{\mathbf b})
 =(b_i-a_i)Z^{\mathbf a+\mathbf b-\mathbf e_i}$, where $\mathbf e_i$ is the
$i$-th standard basis vector, so $a_i+b_i\ge d$.
Since $|\mathbf a|=|\mathbf b|=d$, the exponent vectors differ in at least
two coordinates, say $p,q$. Hence
\[
  2d\le (a_p+b_p)+(a_q+b_q)
     \le |\mathbf a|+|\mathbf b|=2d.
\]
Equality throughout forces the two monomials to be
$Z_p^aZ_q^{d-a}$ and $Z_p^{d-a}Z_q^a$ for some $0\le a\le d$.
A mixed monomial therefore has at most one distinct monomial with which
it satisfies \eqref{eq:power-divisibility}, since its support determines
$\{p,q\}$. Since $k\ge3$, every basis monomial must be a pure power.
The converse is clear.
\end{proof}

Let $\Tbb=(\Gbb_m)^k$. For $t=(t_1,\ldots,t_k)\in\Tbb$, define
$(t\cdot P)(Z_1,\ldots,Z_k):=P(t_1Z_1,\ldots,t_kZ_k)$. This gives an
algebraic action on $\mathbf G_{d,k}$ by
$t\cdot L:=\{t\cdot P\mid P\in L\}$. For $L\in\mathbf G_{d,k}$, set
$Y_L:=\overline{\Tbb\cdot L}$. It is a projective irreducible
$\Tbb$-stable variety.

\begin{lem}
\label{lem:torus-fixed-point}
For every $L\in\Acal_{\mathrm{pow}}(\Cbb)$, we have $L_{\mathrm{pow}} \in Y_L$.
\end{lem}

\begin{proof}
The chain rule gives $\partial_i(t\cdot P)=t_i\bigl(t\cdot(\partial_iP)\bigr)$.
Thus we have
\[
   H_i(t\cdot P,t\cdot Q)=t_i\bigl(t\cdot H_i(P,Q)\bigr).
\]
Thus $Z_i^{d-1}\mid H_i(P,Q)$ implies
$Z_i^{d-1}\mid H_i(t\cdot P,t\cdot Q)$.
Let $L\in\Acal_{\mathrm{pow}}(\Cbb)$.
Lemma~\ref{lem:power-equations} gives
$t\cdot L\in\Acal_{\mathrm{pow}}(\Cbb)$ for every $t\in\Tbb$.
Since $\Acal_{\mathrm{pow}}$ is closed, $Y_L\subset\Acal_{\mathrm{pow}}$.

Borel's fixed point theorem
(cf.~\cite[\S21.2, Theorem]{Humphreys1975}) therefore gives a point
$L_*\in Y_L$ fixed by $\Tbb$. By the weight-space decomposition
(cf.~\cite[Theorem~12.12]{Milne2017AlgebraicGroups}), $L_*$ is a direct sum
of monomial lines.
Lemma~\ref{lem:compatible-monomials} now gives $L_*=L_{\mathrm{pow}}$. Thus
$L_{\mathrm{pow}}\in Y_L=\overline{\Tbb\cdot L}$, as claimed.
\end{proof}

\subsection{From the power-map fiber to generic uniqueness}

\begin{lem}
\label{lem:power-tangent}
The Zariski tangent space
$T_{L_{\mathrm{pow}}}\Acal_{\mathrm{pow}}$ is zero.
\end{lem}

\begin{proof}
Throughout this proof, set $F:=F^{\mathrm{pow}}$ and
$B:=B^{\mathrm{pow}}$.
Let $M_{\mathrm{mix}}\subset S_d$ be the span of the mixed degree-$d$
monomials. The splitting $S_d=L_{\mathrm{pow}}\oplus M_{\mathrm{mix}}$
identifies
\[
 T_{L_{\mathrm{pow}}}\mathbf G_{d,k}
 \simeq\operatorname{Hom}(L_{\mathrm{pow}},M_{\mathrm{mix}}).
\]
Put $D_\varepsilon:=\Cbb[\varepsilon]/(\varepsilon^2)$
and let $\tau:\Spec D_\varepsilon\to\Acal_{\mathrm{pow}}$
be a tangent vector based at $L_{\mathrm{pow}}$.
In the corresponding graph chart, the deformed plane has a
$D_\varepsilon$-basis
\begin{equation}
 P_i(\varepsilon)
 =Z_i^d+\varepsilon\Phi_i,
 \qquad \Phi_i\in M_{\mathrm{mix}},
 \qquad 1\le i\le k.
 \label{eq:tangent-graph-basis}
\end{equation}
Thus it suffices to show that $\Phi_i=0$ for all $i$.

The proof of Lemma~\ref{lem:power-equations} also applies over
$D_\varepsilon$.
Thus the defining equations of $\Acal_{\mathrm{pow}}$ give
\[
 Z_r^{d-1}\mid
 H_r(P_i(\varepsilon),P_j(\varepsilon))
 \qquad(1\le i,j,r\le k).
\]
Fix $i\ne r$ and choose $j\notin\{i,r\}$, using $k\ge3$.
Taking the coefficient of $\varepsilon$ gives
\[
 Z_r^{d-1}\mid
 Z_i^d\partial_r\Phi_j-Z_j^d\partial_r\Phi_i.
\]
No monomial from $Z_i^d\partial_r\Phi_j$ can cancel a monomial
from $Z_j^d\partial_r\Phi_i$. Indeed, a monomial
appearing in both would be divisible by $Z_i^dZ_j^d$, and hence would
have degree at least $2d$, whereas both summands are homogeneous of
degree $2d-1$. Since $j\ne r$, it follows that
$Z_r^{d-1}\mid\partial_r\Phi_i$.
Since $\Phi_i$ contains only mixed monomials, every monomial of
$\partial_r\Phi_i$ has $Z_r$-exponent at most $d-2$.
Thus
$\partial_r\Phi_i=0$ for every $r\ne i$.
Consequently, $\Phi_i$ is a multiple of $Z_i^d$.
As $\Phi_i$ contains no pure-power terms, it is zero.
This holds for every $i$, so
$T_{L_{\mathrm{pow}}}\Acal_{\mathrm{pow}}=0$.
\end{proof}

\begin{lem}
\label{lem:power-fiber}
The scheme $\Acal_{\mathrm{pow}}$ is the reduced point
$L_{\mathrm{pow}}$.
\end{lem}

\begin{proof}
Let $R=\Ocal_{\Acal_{\mathrm{pow}},L_{\mathrm{pow}}}$, with maximal ideal
$\mathfrak m$. From $T_{L_{\mathrm{pow}}}\Acal_{\mathrm{pow}}
 \simeq\operatorname{Hom}_{\Cbb}
       (\mathfrak m/\mathfrak m^2,\Cbb)$ and Lemma~\ref{lem:power-tangent}, we infer that
$\mathfrak m/\mathfrak m^2=0$.  Nakayama's lemma yields
$\mathfrak m=0$, and hence $R\simeq\Cbb$.  In particular,
$L_{\mathrm{pow}}$ is reduced and isolated.

Let $L$ be any closed point of $\Acal_{\mathrm{pow}}$.  By Lemma
\ref{lem:torus-fixed-point}, $Y_L$ contains $L_{\mathrm{pow}}$. The singleton
$\{L_{\mathrm{pow}}\}$ is
both open and closed, so irreducibility forces
$\overline{\Tbb\cdot L}=\{L_{\mathrm{pow}}\}$, and hence
$L=L_{\mathrm{pow}}$. This proves that
$\Acal_{\mathrm{pow}}$ is the reduced point $L_{\mathrm{pow}}$.
\end{proof}

\begin{rem}
\label{rem:surface-power-fiber}
When $k=2$ and $d\ge3$, the plane
$\langle Z_1Z_2^{d-1},\,Z_1^{d-1}Z_2\rangle$
also satisfies the power-map equations. Thus the power-map fiber $\Acal_{\mathrm{pow}}$ is no longer a
singleton, which explains the restriction $k\ge3$ in the argument above.
\end{rem}

\begin{prop}
\label{prop:generic-admissible-uniqueness}
There exists a Zariski dense open subset
$\Qcal_{d,k}^{(1)}\subset\Qcal_{d,k}$, defined over $\Qbb$, such that, for every
complex point $(F,B)\in\Qcal_{d,k}^{(1)}(\Cbb)$, the scheme
$\Acal(F,B)$ is the reduced point $L_F$.
\end{prop}

\begin{proof}
By construction, the morphism $\pi:\Iscr\to\Qcal_{d,k}$ is defined over $\Qbb$.
Let $W\subset\Iscr$ be the quasi-finite locus of $\pi$ and put
$\Zscr:=\Iscr\setminus W$. Since $\pi$ is proper, $\pi(\Zscr)$ is closed.
The point $q_{\mathrm{pow}}$ and the plane $L_{\mathrm{pow}}$ are
$\Qbb$-rational, and Lemma~\ref{lem:power-fiber} shows that the geometric
fiber over $q_{\mathrm{pow}}$ is a reduced point. Hence
$q_{\mathrm{pow}}\notin\pi(\Zscr)$. Over the nonempty open
subset
\[
   U^{(1)}
      :=\Qcal_{d,k}\setminus\pi(\Zscr),
\]
the restricted morphism
$\pi^{(1)}:\Iscr_{U^{(1)}}\to U^{(1)}$ is proper
and quasi-finite, hence finite. Therefore
$\Mcal:=(\pi^{(1)})_*\Ocal_{\Iscr_{U^{(1)}}}$ is coherent.

For $q\in U^{(1)}$, let
$\Iscr_q:=\Iscr_{U^{(1)}}
\times_{U^{(1)}}
\Spec\kappa(q)$, and define
\begin{equation*}
   \ell(q)
   :=\dim_{\kappa(q)}
      \bigl(\Mcal_q
         \otimes_{\Ocal_{U^{(1)},q}}\kappa(q)\bigr)
   =\operatorname{length}_{\kappa(q)}(\Iscr_q).
\end{equation*}
By \cite[Exercise II.5.8]{Hartshorne1977}, $\ell$ is upper
semicontinuous. Lemma~\ref{lem:power-fiber} gives
$\ell(q_{\mathrm{pow}})=1$. Hence
\[
   \Qcal_{d,k}^{(1)}
      :=\{q\in U^{(1)}\mid \ell(q)\leq1\}
\]
is a Zariski dense open subset defined over $\Qbb$.

By Lemma~\ref{lem:tautological-solution}, every complex fiber contains $L_F$,
so every fiber over $(F,B)\in\Qcal_{d,k}^{(1)}(\Cbb)$ has length one. It is
therefore the reduced point $L_F$.
\end{proof}

\begin{cor}
\label{cor:leading-system-uniqueness}
Let $d\ge2$ and $k\ge3$. Let $f,g\in\Poly^k_d$, put
$F:=f_h$ and $B:=f_{d-1}$, and assume that
$(F,B)\in\Qcal_{d,k}^{(1)}(\Cbb)$. If $\mu_f=\mu_g$, then
$L_F=L_{g_h}$.
\end{cor}

\begin{proof}
By Proposition~\ref{prop:measure-implies-admissible}, the plane $L_{g_h}$ is
$(F,B)$-admissible. Thus $L_{g_h}\in\Acal(F,B)(\Cbb)$. By Proposition
\ref{prop:generic-admissible-uniqueness}, this fiber consists only of $L_F$.
Therefore $L_{g_h}=L_F$.
\end{proof}

\section{The stabilizer argument and generic measure rigidity}
\label{sect:generic-stabilizer}

Throughout this section, we assume that $d\geq2$ and $k\geq3$.
Recall the notation introduced at the beginning of
Section~\ref{subsec:fiberwise-admissibility}.

For $A\in\operatorname{GL}(V)$ and
$P\in S_d$, write $(A^*P)(z):=P(Az)$.
If $\alpha=[A]\in\PGL(V)$ and $L\subset S_d$, the plane $A^*L$ is
independent of the choice of the lift $A$; we denote it by $\alpha^*L$.
Define
\[
   H_F:=\operatorname{Stab}_{\PGL(V)}(L_F)
      =\{\alpha\in\PGL(V)\mid \alpha^*L_F=L_F\}.
\]
Let $\alpha\in H_F$ and choose a lift $A\in\operatorname{GL}(V)$.
There is a unique
$\widehat A_F\in\operatorname{GL}(V)$ such that
\begin{equation}
   F(Az)=\widehat A_F F(z).
   \label{eq:kappa-relation}
\end{equation}
The identities
$\widehat{cA}_F=c^d\widehat A_F$ and, for lifts $A_1,A_2$ of elements of
$H_F$, $\widehat{A_1A_2}_F=\widehat{A_1}_F\widehat{A_2}_F$ show that
\eqref{eq:kappa-relation} defines a homomorphism
\[
   \kappa_F\colon H_F\longrightarrow\PGL(V),
   \qquad \alpha=[A]\longmapsto[\widehat A_F],
\]
independent of all choices of lifts.
Put $\widehat H_F:=\kappa_F(H_F)$.

The rigidity condition needed in the final proof is
\begin{equation*}
   \kappa_F^{-1}(H_F)=\{1\},
   \tag{$\dagger$}\label{eq:rigidity-condition}
\end{equation*}

We use the generic stabilizer theorem of Guralnick--Lawther
\cite[Theorem~4 and Table~1.4]{GuralnickLawther2024}. For $d\geq2$ and
$k\geq3$, the natural action of $\PGL(V)$ on $\mathbf G_{d,k}$ has
finite generic stabilizer, which is trivial unless $(d,k)=(2,3)$.
In their notation, this action corresponds, up to graph automorphisms,
to the large higher quadruple $(A_{k-1},d\omega_1,\infty,k)$.
By their Theorem~4, every large higher quadruple not listed in
Table~1.4 has trivial generic stabilizer, i.e., $H_F=\{1\}$ for
general $F$ in the corresponding cases of our family.
The only member of this family appearing in Table~1.4 is
$(A_2,2\omega_1,\infty,3)$. In this case the generic stabilizer $H$
fits into a short exact sequence
\[
   1\longrightarrow C_3\longrightarrow H\longrightarrow\Sfrak_3
      \longrightarrow1,
\]
where $C_3$ is cyclic of order $3$ and $\Sfrak_3$ is the symmetric
group on three letters.

Consider the dominant morphism $\rho_{d,k}:\Fcal_{d,k}\longrightarrow\mathbf G_{d,k}$, $F\longmapsto L_F$.

\begin{lem}
\label{lem:nonexceptional-stabilizer}
Assume $d\geq2$, $k\geq3$, and $(d,k)\neq(2,3)$. There is a Zariski
dense open subset $\Fcal_{d,k}^{(2)}\subset\Fcal_{d,k}(\Cbb)$
such that $H_F=\{1\}$ for every $F\in\Fcal_{d,k}^{(2)}$.
In particular, the rigidity condition~\eqref{eq:rigidity-condition} holds on this subset.
\end{lem}

\begin{proof}
By the preceding Guralnick--Lawther result, there is a Zariski dense open
subset $U\subset\mathbf G_{d,k}(\Cbb)$ on which the stabilizer in
$\PGL(V)$ is trivial. It suffices to set
$\Fcal_{d,k}^{(2)}:=\rho_{d,k}^{-1}(U)$.
\end{proof}

\begin{rem}
\label{rem:stabilizer-range}
More generally, $H_F=\{1\}$ for general $F\in\Fcal_{d,k}(\Cbb)$
whenever $d,k\geq2$ and $(d,k)\notin\{(2,2),(3,2),(2,3)\}$,
by \cite[Theorem~4 and Table~1.4]{GuralnickLawther2024}.
In this section we only need the range
$d\geq2$, $k\geq3$.
\end{rem}

\begin{lem}
\label{lem:exceptional-stabilizer}
There is a Zariski dense open subset $\Fcal_{2,3}^{(3)}\subset\Fcal_{2,3}(\Cbb)$ such that
$H_F$ is finite and $\kappa_F$ is injective for every $F\in \Fcal_{2,3}^{(3)}$.
\end{lem}

\begin{proof}
By the preceding Guralnick--Lawther result, choose a Zariski dense open
subset $U\subset\mathbf G_{2,3}(\Cbb)$ on which every stabilizer fits
into the stated short exact sequence. Then
$\Fcal_{2,3}^{(3)}:=\rho_{2,3}^{-1}(U)$ is a Zariski dense open subset of
$\Fcal_{2,3}(\Cbb)$. Since linear postcomposition preserves $L_F$,
it also preserves $\Fcal_{2,3}^{(3)}$.

Let $F\in \Fcal_{2,3}^{(3)}$ and put $K:=\ker\kappa_F\triangleleft H_F$.
The group $H_F$ is finite, and every $\alpha\in K$ satisfies
$F\alpha=F$ as morphisms of $\Pbb^2$. Each nonidentity element
of $K$ has a proper fixed-point locus. Since $F$ is finite, a general
fiber avoids all these loci, so $K$ acts freely on it. The morphism
$F:\Pbb^2\to\Pbb^2$ has degree $4$, and hence $\#K\mid4$.

Let $\pi:H_F\to\Sfrak_3$ be the quotient map in the short exact sequence
above. Then $\pi(K)$ is a normal $2$-subgroup of $\Sfrak_3$, and is
therefore trivial. Thus $K\subset\ker\pi\simeq C_3$.
Together with $\#K\mid4$, this gives $K=\{1\}$, proving that
$\kappa_F$ is injective.
\end{proof}

\begin{lem}
\label{lem:relative-position}
Let $A_0,B_0\subset\PGL(V)$ be finite subgroups. There is a Zariski
dense open subset $\Omega_{A_0,B_0}\subset\PGL(V)$
such that $A_0\cap TB_0T^{-1}=\{1\}$ for every
$T\in\Omega_{A_0,B_0}$.
\end{lem}

\begin{proof}
For $a\in A_0\setminus\{1\}$ and $b\in B_0\setminus\{1\}$, put
\[
   Z_{a,b}:=\{T\in\PGL(V)\mid a=TbT^{-1}\}.
\]
This set is empty unless $a$ and $b$ are conjugate.  If it is nonempty, it is a coset
of the centralizer of $b$. That centralizer is proper because the center of
$\PGL(V)$ is trivial and $b\neq1$. Thus
\[
   \Omega_{A_0,B_0}
      :=\PGL(V)\setminus
        \bigcup_{\substack{a\in A_0\setminus\{1\}\\
                            b\in B_0\setminus\{1\}}} Z_{a,b}
\]
is a Zariski dense open subset with the required property.
\end{proof}

\begin{prop}
\label{prop:uniform-rigidity-condition}
For every $d\geq2$ and $k\geq3$, there is a Zariski dense open subset $\Fcal_{d,k}^{(4)}\subset\Fcal_{d,k}$,
defined over $\Qbb$, such that the rigidity condition~\eqref{eq:rigidity-condition} holds for every
$F\in\Fcal_{d,k}^{(4)}(\Cbb)$.
\end{prop}

\begin{proof}
Put $G:=\PGL(V)$ and consider the locally closed incidence subset
\[
   \Zcal_{d,k}:=\bigl\{(F,\alpha,\gamma)\in\Fcal_{d,k}\times G^2
      \mid F\alpha=\gamma F,\,\gamma^*L_F=L_F,\,\alpha\ne1\bigr\}.
\]
Here $F\alpha=\gamma F$ denotes equality of the induced morphisms of
$\Pbb(V)$. By Chevalley's theorem,
$\Bcal_{d,k}:=\operatorname{pr}_{\Fcal}(\Zcal_{d,k})$ is a
constructible subset defined over $\Qbb$, with
\[
   F\in\Bcal_{d,k}(\Cbb)
   \quad\Longleftrightarrow\quad
   \kappa_F^{-1}(H_F)\ne\{1\}.
\]

For $(d,k)\ne(2,3)$, Lemma~\ref{lem:nonexceptional-stabilizer}
immediately shows that $\Bcal_{d,k}(\Cbb)$ is not Zariski dense.
Suppose now that $(d,k)=(2,3)$ and that $\Bcal_{2,3}(\Cbb)$ is dense.
By constructibility, it contains a nonempty Zariski open subset
$O\subset\Fcal_{2,3}(\Cbb)$. Choose
$F\in O\cap\Fcal_{2,3}^{(3)}$. For general
$T\in\operatorname{GL}(V)$, we have $TF\in O\cap\Fcal_{2,3}^{(3)}$.
As
\[
   H_{TF}=H_F,\qquad
   \widehat H_{TF}=[T]\widehat H_F[T]^{-1},
\]
Lemmas~\ref{lem:exceptional-stabilizer} and~\ref{lem:relative-position}
imply that $TF\notin\Bcal_{2,3}(\Cbb)$ for general $T$, contradicting
$O\subset\Bcal_{2,3}(\Cbb)$.

Thus $\Bcal_{d,k}$ is not dense in any case. It suffices to set
$\Fcal_{d,k}^{(4)}:=\Fcal_{d,k}\setminus\overline{\Bcal_{d,k}}^{\,\mathrm{Zar}}$.
\end{proof}

Let $\Qcal_{d,k}^{(4)}$ be the inverse image of
$\Fcal_{d,k}^{(4)}$ under $(F,B)\mapsto F$.

Define $\Qcal_{d,k}^{\circ}:=\Qcal_{d,k}^{(1)}\cap\Qcal_{d,k}^{(4)}$ and
\[
   \Ucal_{d,k}
      :=\bigl\{f\in\Poly_d^k\mid(f_h,f_{d-1})\in\Qcal_{d,k}^{\circ}\bigr\}.
\]
Both are Zariski dense open subsets defined over $\Qbb$.

\begin{proof}[Proof of Theorem~\ref{thm:generic-measure-rigidity} for $k\geq3$]
Assume $k\geq3$, let $f\in\Ucal_{d,k}(\Cbb)$, and let
$g\in\Poly^k_d$ satisfy $\mu_g=\mu_f$.  Write
\[
   F:=f_h,\qquad B:=f_{d-1},
   \qquad
   f=F+B+\sum_{j=0}^{d-2}f_j,
   \qquad
   (F,B)\in\Qcal_{d,k}^{\circ}(\Cbb).
\]
Corollary~\ref{cor:leading-system-uniqueness} gives
$L_{g_h}=L_F$.
Thus there is a unique $A\in\operatorname{GL}(V)$ such that $g_h=AF$.
Theorem~\ref{thm: main1} gives an affine symmetry
$\sigma\in\Sa_f$ and some $b\in V$ such that
\begin{equation}
   g=\sigma f,
   \qquad
   \sigma(z)=Az+b,
   \qquad
   G_f\circ\sigma=G_f
   \label{eq:generic-g-sigma-f}
\end{equation}

Set $u:=f\sigma$. Its leading part is $u_h(z)=F(Az)$. From  \eqref{eq:generic-g-sigma-f}, $G_f=G_u$, and thus $\mu_f=\mu_u$.
Applying Corollary~\ref{cor:leading-system-uniqueness} to $(f,u)$ yields
$A^*L_F=L_F$, so $[A]\in H_F$.
Since $F$ is fixed, write $\widehat A:=\widehat A_F$ for the unique matrix
such that
\begin{equation}
   F(Az)=\widehat A F(z).
   \label{eq:generic-induced-A}
\end{equation}
Since $\mu_u=\mu_f$ and $u_h=\widehat A F$,
Theorem~\ref{thm: main1}, applied to $(f,u)$, gives an affine symmetry
$\tau\in\Sa_f$ of the form $\tau(z)=\widehat A z+c$
for some $c\in V$, such that
\begin{equation}
   u=f\sigma=\tau f,
   \qquad
   G_f\circ\tau=G_f.
   \label{eq:generic-u-tau-f}
\end{equation}

Repeating the above argument for $f\tau$, we get
$\widehat A^*L_F=L_F$, and therefore $[\widehat A]\in H_F$.
The rigidity condition~\eqref{eq:rigidity-condition}, proved in
Proposition~\ref{prop:uniform-rigidity-condition}, now gives
$[A]=1$. Thus $A=\lambda I_k$ for some $\lambda\in\Cbb^*$, and
\eqref{eq:generic-induced-A} gives $\widehat A=\lambda^dI_k$.

Comparing the degree-$(d-1)$ homogeneous parts in
\eqref{eq:generic-u-tau-f} gives
\[
   DF(z)b=(\lambda-1)B(z).
\]
If $\lambda\ne1$, put $v:=b/(\lambda-1)$. Then $B=DF\,v$, so
$\eta_{F,B}=\operatorname{adj}(DF)B=\Delta_Fv$ and
\[
   M^{F,B}_{P,Q}
      =\Delta_F(P\nabla Q-Q\nabla P)v
      \qquad(P,Q\in S_d).
\]
Thus $\alpha_{F,B}=0$ and $\Acal(F,B)=\mathbf G_{d,k}$, contradicting
Proposition~\ref{prop:generic-admissible-uniqueness}, since
$(F,B)\in\Qcal_{d,k}^{(1)}(\Cbb)$ and $\dim\mathbf G_{d,k}>0$.
Hence $\lambda=1$. Now $DF\,b=0$ and $\Delta_F\not\equiv0$ imply $b=0$.
Consequently $\sigma=\id$, and \eqref{eq:generic-g-sigma-f} yields
$g=f$.
\end{proof}

\section{Further applications of the image--fiber principle}
\label{sect:further-applications}

\subsection{Tits alternative}\label{sect: tits}
In this subsection we prove Theorem~\ref{thm: Tits1}. 

We start by showing the following uniform bound on the size of the fibers. For
$k\geq2$, recall that we have the restriction map
$\rho:\APoly^k\to\End(\Pi)$.

\begin{prop}\label{prop: finite_fiber}
Let $k\geq2$, and let $\Scal$ be a finitely generated semigroup of regular
polynomial endomorphisms of $\Cbb^k$ of degree $\ge 1$. Assume that
$\Scal^+\neq\varnothing$ and that all elements of $\Scal^+$ have the
same preperiodic set. Then there
exists $M\ge 1$ such that
\begin{align*}
\#\bigl(\rho^{-1}(u)\cap \Scal\bigr)\le M
\qquad\text{for every }u\in \rho(\Scal).    
\end{align*}
\end{prop}

Let us prepare some lemmas before proceeding with the proof of Proposition~\ref{prop: finite_fiber}.
\begin{lem}\label{lem:azd-same-prep}
Let $d\ge 2$ and $a,b\in \Cbb^*$. If the polynomials
\begin{align*}
z\longmapsto az^d
\qquad\text{and}\qquad
z\longmapsto bz^d
\end{align*}
have the same preperiodic points in $\Abb^1(\Cbb)$, then $b/a$ is a root of unity.
\end{lem}

\begin{proof}
Choose $\alpha_a,\alpha_b\in\Cbb^*$ such that
$\alpha_a^{d-1}=a^{-1}$ and $\alpha_b^{d-1}=b^{-1}$. Conjugation by
$z\mapsto\alpha_a z$ identifies $az^d$ with $z^d$, and similarly for $b$.
Consequently,
\[
   \Prep(az^d)\cap\Cbb^*=\alpha_a\mu_\infty,
   \qquad
   \Prep(bz^d)\cap\Cbb^*=\alpha_b\mu_\infty,
\]
where $\mu_\infty$ is the group of all roots of unity. Equality of these sets
implies $\alpha_b/\alpha_a\in\mu_\infty$, and hence $b/a$ is a root of unity.
\end{proof}

\begin{lem}\label{lem:common-conjugacy-to-homogeneous}
Under the assumptions of Proposition~\ref{prop: finite_fiber}, suppose moreover that
one element of $\Scal^+$ is affinely conjugate to a homogeneous polynomial map.
Then there exists a single affine map $\sigma$ of $\Cbb^k$ such that
$\sigma g\sigma^{-1}$ is homogeneous for every $g\in \Scal^+$.
\end{lem}

\begin{proof}
Choose $f\in \Scal^+$ and an affine map $\sigma$ such that $\sigma  f\sigma ^{-1}$ is
homogeneous. If $J$ denotes the Julia set of $f$, then $\sigma(J)$ is
$S^1$-invariant.
For any $g\in \Scal^+$, since $J_g=J$, we have $J_{\sigma g\sigma^{-1}}=\sigma(J),$
which is $S^1$-invariant as well. By Lemma~\ref{lem: charcterizationHomogeneous},
$\sigma g\sigma ^{-1}$ is homogeneous.
\end{proof}

\begin{lem}\label{lem:fibers-finite-homogeneous}
Under the assumptions of Proposition~\ref{prop: finite_fiber}, suppose moreover that every element of $\Scal^+$ is homogeneous. Then there exists $M\ge 1$ such that
\begin{align*}
\#\bigl(\rho^{-1}(u)\cap \Scal\bigr)\le M
\qquad\text{for every }u\in \rho(\Scal).
\end{align*}
\end{lem}

\begin{proof}
Choose a finitely generated field $K\subset \Cbb$ over which a finite set of generators
of $\Scal$ is defined. Then every element of $\Scal$ is defined over $K$, and $K$ has
only finitely many roots of unity.

\textbf{Case 1.} Let $u\in \rho(\Scal)^+$.
Choose $f,g\in \rho^{-1}(u)\cap \Scal^+$. Since $f$ and $g$ are homogeneous of the same degree, the equality
$f_\Pi=g_\Pi$ implies $g=\lambda f$
for some $\lambda\in K^*$.

Choose a periodic point $a\in \Pi$ of $u$, say of period $m\ge 1$, and let
$L_a\subset \Cbb^k$ be the complex line corresponding to $a$.
Then $f^m(L_a)=g^m(L_a)=L_a$.
In a linear coordinate $z$ on $L_a$, we can write
\begin{align}\label{eq: restriction-La}
    f^m|_{L_a}(z)=\beta z^D,
\qquad
g^m|_{L_a}(z)=\lambda^{1+d+\cdots+d^{m-1}}\beta z^D
\end{align}
for some $\beta\in \Cbb^*$ and $D=d^m$.

Since $\Prep(f)=\Prep(g)$ by the assumption, we have $\Prep(f^m)=\Prep(g^m),$
and restricting to the invariant line $L_a$ gives equality of the preperiodic sets of the
    two one-variable maps in Eq.~\eqref{eq: restriction-La}. Lemma~\ref{lem:azd-same-prep} therefore shows that $\lambda^{1+d+\cdots+d^{m-1}}$
is a root of unity, hence $\lambda$ itself is a root of unity.
Since $K$ contains only finitely many roots of unity, the fiber
$\rho^{-1}(u)\cap \Scal^+$ is uniformly finite.

\textbf{Case 2.} Now let $u\in \rho(\Scal)$ have degree $1$, and let $\sigma,\tau\in \rho^{-1}(u)\cap \Scal$.
Fix $F\in \Scal^+$. Since every element of $\Scal^+$ is homogeneous, both
$\sigma F$ and $\tau F$ are homogeneous. This forces $\sigma$ and $\tau$ to
be linear. Therefore, there exists $\lambda\in K^*$ such that $\sigma=\lambda\tau.$

Applying the argument of Case 1 to $f=\tau F$ and $g=\sigma F$, we conclude that $\lambda$ is a root of unity.
Thus the degree-one fibers are also uniformly finite. This proves the lemma.
\end{proof}

\begin{proof}[Proof of Proposition~\ref{prop: finite_fiber}]
Suppose first that some $f_0\in\Scal^+$ is not affinely conjugate to a
homogeneous polynomial map. Lemma~\ref{lem:common-conjugacy-to-homogeneous}
then shows that no element of $\Scal^+$ is affinely conjugate to a homogeneous
map. Since all elements of $\Scal^+$ have the same preperiodic set, they have
the same equilibrium measure by Yuan--Zhang~\cite{YZhodge17}, and hence the
same Julia set. In particular, their affine symmetry groups and restriction
kernels agree. By Theorem~\ref{thm: main2},
this common kernel is finite. Put
\[
   M_0:=\#\Sigma_{f_0}^0<\infty,
\]

Let $u\in\rho(\Scal)$ have degree at least $2$, and choose a basepoint
$h_u\in\rho^{-1}(u)\cap\Scal$. For every
$h\in\rho^{-1}(u)\cap\Scal$, the equality $h_\Pi=(h_u)_\Pi$ implies that
the leading homogeneous parts of $h$ and $h_u$ differ by a nonzero scalar.
Theorem~\ref{thm: main1} therefore
gives $h=\sigma_h h_u$ for some
$\sigma_h\in\Sa_{h_u}$. Passing to infinity and using the
surjectivity of $u$, we obtain $(\sigma_h)_\Pi=\id$. Hence
$\sigma_h\in\Sigma_{h_u}^0=
\Sigma_{f_0}^0$. Moreover, $\sigma_h$ is unique because $h_u$
is dominant. Thus this fiber injects into a set of cardinality $M_0$.

If $u$ has degree one, right composition by $f_0$ defines an injection
\[
   \rho^{-1}(u)\cap\Scal
   \longrightarrow
   \rho^{-1}\bigl(u(f_0)_\Pi\bigr)\cap\Scal^+,
   \qquad h\longmapsto hf_0.
\]
This map is injective because $f_0$ is dominant. Hence the same bound $M_0$
applies to degree-one fibers.

It remains to consider the case in which every element of $\Scal^+$ is
affinely conjugate to a homogeneous polynomial map. By
Lemma~\ref{lem:common-conjugacy-to-homogeneous}, one affine conjugacy makes all
of them homogeneous. Lemma~\ref{lem:fibers-finite-homogeneous} then gives a
uniform fiber bound in this case as well.
\end{proof}

Let us recall the following results of Bell--Huang--Peng--Tucker and Beaumont.
\begin{thm}[{\cite[Theorem 1.3]{BHPTjems24}}]\label{thm: BHPT1.3}
Let $V$ be a projective variety over a field $K$, and let $f, g : V \to V$ be endomorphisms polarized by the same line bundle $\Lcal$. If $\Prep(f) \neq \Prep(g)$, then, for every sufficiently large positive integer $j$, the semigroup $\langle f^j, g^j \rangle$ is a free semigroup on two generators.
\end{thm}
Beaumont strengthened this result.
\begin{thm}[{\cite[Theorem 1.1]{beaumont2025uniformtitsalternativeendomorphisms}}]\label{thm: Beaumont-tits}
    One can take $j=1$ in Theorem~\ref{thm: BHPT1.3}.
\end{thm}

\begin{prop}[{\cite[Proposition 4.10]{BHPTjems24}}]\label{prop: BHPT4.10}
 Let $\Sfrak$ be a finitely generated semigroup of endomorphisms of $\Pbb^1$ over $\Cbb$. Suppose that for any $f,g \in \Sfrak^+$, we have $\Prep(f) = \Prep(g)$. Then $\Sfrak$ has polynomially bounded growth, and if $\Sfrak^+$ contains at least one nonspecial rational function, then $\Sfrak$ has linear growth.
\end{prop}

We are now in a position to prove Theorem~\ref{thm: Tits1}.
\begin{proof}[Proof of Theorem~\ref{thm: Tits1}]
If $\Scal^+ =\varnothing$, then this is a result of Okni\'nski--Salwa~\cite[Theorem 1.5]{zbMATH00825311}. So we may assume $\Scal^+\neq \varnothing$.
Let $S$ be a finite generating set of $\Scal$, and let $S_\Pi:=\rho(S)$, which generates
$\rho(\Scal)\subset \End(\Pi)$.
If $\Scal$ contains a nonabelian free semigroup, there is nothing to prove. So assume from
now on that it does not.

By Theorem~\ref{thm: BHPT1.3} or~\ref{thm: Beaumont-tits}, all elements of $\Scal^+$ have the
same preperiodic set in $\Pbb^2$; in particular all elements of $\rho(\Scal)^+$ have the
same preperiodic set on $\Pi$. Therefore $\rho(\Scal)$
has polynomially bounded growth by Proposition~\ref{prop: BHPT4.10}. Moreover, if there exists  $f\in \Scal^+$ such that $f_\Pi$ is nonspecial, then $\rho(\Scal)$
has linear growth. 

By Proposition~\ref{prop: finite_fiber},
for every $n\ge 0$, $d_{S_\Pi}(n)\le d_S(n)\le M\, d_{S_\Pi}(n).$
Thus $\rho(\Scal)$ and $\Scal$ have comparable growth functions. This proves
Theorem~\ref{thm: Tits1}.
\end{proof}

An immediate consequence of the proof of Theorem~\ref{thm: Tits1}, combined
with Theorem~\ref{thm: Beaumont-tits}, is the following.
\begin{cor}\label{cor: prep-free}
    Let $f,g\in\Poly^2$ be of degree $\ge 2$. Then 
    \begin{align*}
        \Prep(f)=\Prep(g) \Longleftrightarrow \langle f,g \rangle\ \text{is not freely generated by $f$ and $g$}.
    \end{align*}
\end{cor}

\subsection{Centralizers}\label{sect: cent}
In this subsection we prove Theorem~\ref{thm: centralizer}.

\begin{lem}\label{lem:restriction-to-infinity-centralizer}
Let $f\in\Poly^2_d$ for some $d\ge 2$. We have the following assertions.
\begin{enumerate}
    \item If $g,h\in C(f^\infty)$ satisfy $g_\Pi=h_\Pi$, then there exists
$\sigma\in \Sa_f$ such that
$g=\sigma h$.
    In particular, $\#\bigl(\rho^{-1}(u)\cap C(f^\infty)\bigr)\leq \#\Sa_f$ for every $u\in \rho(C(f^\infty)).$
\item Let $g,h \in C(f^n)$ for some $n\geq 1$. If there exists $\sigma\in \APoly^2$ such that $g=\sigma h$, then $\sigma \in C(f^n)$.
\end{enumerate}
\end{lem}

\begin{proof}
(1). Let $g,h\in C(f^\infty)$ satisfy $g_\Pi=h_\Pi$. In
particular, $\deg g=\deg h=:e$. If $e=1$, then
$g,h\in\Sa_f$, and we take
$\sigma=gh^{-1}$. If $e\geq2$, the required $\sigma$ is supplied by
Theorem~\ref{thm: main1}.

(2). Let $g,h \in C(f^n)$ for some $n\geq 1$. Then
\begin{align*}
f^n\sigma h=f^ng=gf^n=\sigma h f^n=\sigma f^n h.
\end{align*}
By right cancellation, $f^n\sigma =\sigma f^n$.
\end{proof}

Let us recall Beaumont's result.
\begin{thm}[Beaumont~\cite{beaumont2025centralizersendomorphismsprojectiveline}]\label{thm: BeaumontC}
    Let $f\in\End(\Pbb^1)^+$ be of degree at least $2$. If $f$ is not conjugate to a power map, then there exists $N\ge 1$ such that $C(f^\infty)=C(f^N)$. If moreover $f$ is nonspecial and $S\subset C(f^\infty)$ is a subsemigroup containing $f$, then there exist finitely many elements $u_1,\ldots,u_r\in S$ such that $S=\{u_1,\ldots,u_r\}\langle f\rangle.$
\end{thm}

\begin{proof}[Proof of Theorem~\ref{thm: centralizer}]
\textbf{(1).}
By Theorem~\ref{thm: BeaumontC} applied to $f_\Pi$,
there exists $N_0\geq 1$ such that $C(f_\Pi^\infty)=C(f_\Pi^{N_0}).$
Set $\Gamma_f:=\Sa_f\cap C(f^\infty).$
This is a finite group. For each $\tau\in \Gamma_f$, choose $m_\tau\geq 1$ such that $\tau f^{m_\tau}=f^{m_\tau} \tau$.
Set
\begin{align*}
   M:=\operatorname{lcm}\{m_\tau\mid \tau\in \Gamma_f\},
\ \ \
L:=\operatorname{lcm}(N_0,M),
\ \ \ \text{and}
\ \ \
E:= \# \Gamma_f. 
\end{align*}

Let $g\in C(f^\infty)$, and choose $n_g\geq1$ such that
$g\in C(f^{n_g})$. Since $L$ is a multiple of $N_0$, we have
$(g f^L)_\Pi=(f^L g)_\Pi$. Lemma~\ref{lem:restriction-to-infinity-centralizer}(1)
gives $\sigma_g\in\Sa_f$ such that
\begin{align}\label{eq: central1}
g f^L=\sigma_g f^L g.    
\end{align}
Both $g f^L$ and $f^L g$ belong to $C(f^{n_g})$, so part~(2) of the same
lemma gives $\sigma_g\in C(f^{n_g})$. Thus $\sigma_g\in\Gamma_f$.
Since $L$ is a multiple of $M$, every element of $\Gamma_f$ commutes with $f^L$. In particular, $\sigma_g f^L=f^L \sigma_g.$
From~\eqref{eq: central1}, an
immediate induction therefore gives
\[
   g f^{EL}=\sigma_g^E f^{EL}g=f^{EL}g.
\]
Hence $g\in C(f^{EL})$. This proves $C(f^\infty)\subset C(f^{EL})$ and thus $C(f^\infty)=C(f^{EL}).$
\medskip

\textbf{(2).} Since $f_\Pi$ is not special, Theorem~\ref{thm: BeaumontC} implies that there exist
finitely many elements $u_1,\dots,u_r\in \rho(C(f^\infty))$
such that $\rho(C(f^\infty))=\{u_1,\dots,u_r\}\langle f_\Pi\rangle.$
Set
\[
F:=\bigcup_{i=1}^r \left(\rho^{-1}(u_i)\cap C(f^\infty)\right).
\]
This is finite by Lemma~\ref{lem:restriction-to-infinity-centralizer}(1). It suffices to show that $C(f^\infty) \subset F\langle f\rangle$ since the other inclusion is trivial.

Let $g\in C(f^n)$ for some $n\geq1$. There exist $i\in\{1,\dots,r\}$
and $m\geq0$ such that $g_\Pi=u_i f_\Pi^m$. Choose
$h_i\in\rho^{-1}(u_i)\cap C(f^\infty)$. The maps $g$ and
$h_i f^m$ have the same restriction to infinity, so
Lemma~\ref{lem:restriction-to-infinity-centralizer}(1) gives
$\sigma\in\Sa_f$ such that
$g=\sigma h_i f^m$. After choosing a common iterate of $f$ commuting with
$g$ and $h_i$, part~(2) of that lemma gives $\sigma\in\Gamma_f$. Moreover,
equality at infinity and the surjectivity of $u_i f_\Pi^m$ imply
$\sigma_\Pi=\id$. Hence $q:=\sigma h_i$ belongs to $C(f^\infty)$ and satisfies
$q_\Pi=u_i$, so $q\in F$.
Therefore $g=q f^m\in F\langle f\rangle$. Since $g\in C(f^\infty)$ was
arbitrary, it follows that $C(f^\infty)\subset F\langle f\rangle$.
\end{proof}

The assumption that $f$ is not affinely conjugate to a homogeneous polynomial is essential in
Theorem~\ref{thm: centralizer}.

\begin{ex}\normalfont\label{ex:homogeneous-centralizer-counterexample}
Consider $f(x,y)=(x^2+y^2,y^2).$
In the affine coordinate $w=x/y$ on $\Pi\simeq \Pbb^1$ we have $f_\Pi(w)=w^2+1,$ which is nonspecial. 

For every $\zeta\in S^1$, let $\tau_\zeta(x,y):=(\zeta x,\zeta y).$
By induction, for every $n\ge 1$ we have $f^n(\zeta x,\zeta y)=\zeta^{2^n}f^n(x,y).$
Hence $\tau_\zeta f^n=f^n \tau_\zeta$ if and only if $\zeta^{2^n-1}=1.$
Therefore, we obtain infinitely many degree-$1$ elements in $C(f^\infty)$.

Fix any
$N\ge 1$, and let $\zeta$ be a primitive $(2^{N+1}-1)$-st root of unity. Since $\gcd(2^N-1,\,2^{N+1}-1)=1,$
we have $\zeta^{2^N-1}\ne 1$, so $\tau_\zeta\notin C(f^N)$. Thus $C(f^\infty)\ne C(f^N)$ for every $N\ge 1.$ This shows that assertion~(1) fails. Since $f_\Pi$ is nonspecial, assertion~(2) would imply assertion~(1), so assertion~(2) fails as well.
\end{ex}

\subsection{Non-density of orbit intersections}\label{sect: arith}

In this subsection we prove Theorem~\ref{thm: arith}.
We briefly recall the properties of Moriwaki heights introduced by
Moriwaki~\cite{Moriwaki00,Moriwaki01}. See also Yuan--Zhang's
book~\cite{YZquasi} for the adelic formulation.

Let $K$ be a finitely generated field over $\Qbb$. For every arithmetic
polarization $\overline{B}$ of $K$ in the sense of Moriwaki, every projective
variety $X$ over $K$, and every integrable adelic line bundle $\overline{L}$
on $X$, we denote by $h_{\overline{L}}^{\overline{B}}$
the associated Moriwaki height. For $X=\Pbb^N$, we write $h_{\Pbb^N}^{\overline{B}}$ for a Weil height
associated with $\Ocal_{\Pbb^N}(1)$ and the arithmetic polarization
$\overline{B}$.

Let $F:\Pbb^2\to\Pbb^2$ be a polarized endomorphism of degree $d\geq2$.
By~\cite[Theorem~6.1.1]{YZquasi}, there exists a canonical nef adelic line
bundle $\overline{\Ocal_{\Pbb^2}(1)}_F$ with underlying line bundle
$\Ocal_{\Pbb^2}(1)$ such that $ F^*\overline{\Ocal_{\Pbb^2}(1)}_F
      =d\,\overline{\Ocal_{\Pbb^2}(1)}_F.$
For every arithmetic polarization $\overline{B}$ of $K$, define $\hhat_F^{\overline{B}}
      :=h_{\overline{\Ocal_{\Pbb^2}(1)}_F}^{\overline{B}}.$
By the nefness of $\overline{\Ocal_{\Pbb^2}(1)}_F$,
see also~\cite[Proposition~6.1.4]{YZquasi}, it satisfies
\[
\hhat_F^{\overline{B}}\geq0 \qquad
   \hhat_F^{\overline{B}}
      =h_{\Pbb^2}^{\overline{B}}+O_{\overline{B}}(1),
   \qquad
   \hhat_F^{\overline{B}}F
      =d\,\hhat_F^{\overline{B}}.
\]

\begin{thm}[Moriwaki]\label{thm: moriwaki}
For every big arithmetic polarization $\overline{B}$ of $K$, we have
\[
   \hhat_F^{\overline{B}}(x)=0
   \quad\Longleftrightarrow\quad
   x\in\Prep(F)
\]
for every $x\in\Pbb^2(\overline K)$.
\end{thm}

We shall also use the following consequence of the fundamental inequality
and the arithmetic Hodge index theorem of Yuan--Zhang; see
\cite[Theorems~5.3.2 and~6.5.8]{YZquasi}. A sequence $(x_i)$ in
$\Pbb^2(\overline K)$ is \emph{generic} if it converges to the generic point
in the Zariski topology.

\begin{thm}[Yuan--Zhang]\label{thm:YZ-common-small}
Let $F,G:\Pbb^2\to\Pbb^2$ be polarized endomorphisms over $K$ of degree at
least $2$. Suppose that there exists a generic sequence $(x_i)$ in
$\Pbb^2(\overline K)$ such that, for every arithmetic polarization
$\overline{B}$ of $K$, $\hhat_F^{\overline{B}}(x_i)\to 0$ and $\hhat_G^{\overline{B}}(x_i)\to 0$.
Then, for every arithmetic polarization $\overline{B}$ of $K$, $\hhat_F^{\overline{B}}=\hhat_G^{\overline{B}}.$
\end{thm}

\begin{proof}[Proof of Theorem~\ref{thm: arith}]
Let $K$ be a finitely generated field over $\Qbb$ over which $f$, $g$, and
$c$ are defined. For fixed $m,n\geq1$, the set
\begin{align*}
   X_{m,n}:=\{z\in\Abb^2\mid f^m(z)=g^n(z)=c(z)\}
\end{align*}
is a closed subvariety defined over $K$. If the union of the
$X_{m,n}(\Cbb)$ were Zariski dense in $\Abb^2(\Cbb)$, then the union of the
$X_{m,n}(\overline K)$ would also be Zariski dense. Thus it suffices to prove
non-density over $\overline K$.

Suppose, towards a contradiction, that
\begin{align*}
S:=\left\{
z\in\Abb^2(\overline K)\mid
f^m(z)=g^n(z)=c(z)\ \text{for some }m,n\geq1
\right\}
\end{align*}
is Zariski dense in $\Abb^2$. Choose a generic sequence $(p_i)$ in $S$.
For every $i$, choose $m_i,n_i\geq1$ such that
\[
   f^{m_i}(p_i)=g^{n_i}(p_i)=c(p_i).
\]
For every $m\geq1$, the set
$\{x\in\Abb^2\mid f^m(x)=c(x)\}$ is a proper closed subset of $\Abb^2$,
because $c$ is not a positive iterate of $f$. Since $(p_i)$ is generic, for
every $M\geq1$ only finitely many $i$ satisfy $m_i\leq M$. Hence
$m_i\to\infty$. Similarly, $n_i\to\infty$.

Fix an arbitrary arithmetic polarization $\overline{B}$ of $K$. To simplify
the notation in the following estimates, write $\hhat_f:=\hhat_f^{\overline{B}}$, $\hhat_g:=\hhat_g^{\overline{B}}$, and $h_{\Pbb^2}:=h_{\Pbb^2}^{\overline{B}}.$
Since $\hhat_f=h_{\Pbb^2}+O_{\overline{B}}(1)$ and
\[
   h_{\Pbb^2}(c(x))
      \leq(\deg c)\,h_{\Pbb^2}(x)+O_{\overline{B}}(1)
\]
for every $x\in\Abb^2(\overline K)$, we obtain
\[
   \hhat_f(c(x))
      \leq(\deg c)\,\hhat_f(x)+O_{\overline{B}}(1).
\]
Applying this inequality to $p_i$ and using
$f^{m_i}(p_i)=c(p_i)$ gives
\begin{align*}
   (\deg f)^{m_i}\hhat_f(p_i)
   &=
   \hhat_f(f^{m_i}(p_i))=\hhat_f(c(p_i))\leq
   (\deg c)\,\hhat_f(p_i)+O_{\overline{B}}(1).
\end{align*}
Consequently, there is a constant $C_{\overline{B}}$, independent of $i$,
such that
\[
   \bigl((\deg f)^{m_i}-\deg c\bigr)\hhat_f(p_i)
      \leq C_{\overline{B}}.
\]
Since $\hhat_f\geq0$ and $m_i\to\infty$, 
$\hhat_f(p_i)\to0$. The same argument gives $\hhat_g(p_i)\to0$.

The arithmetic polarization $\overline{B}$ was arbitrary. Therefore, for
every arithmetic polarization $\overline{B}$ of $K$, $\hhat_f^{\overline{B}}(p_i)\to 0$ and $\hhat_g^{\overline{B}}(p_i)\to 0$.
Theorem~\ref{thm:YZ-common-small} now gives
$\hhat_f^{\overline{B}}=\hhat_g^{\overline{B}}$ for every
$\overline{B}$. In particular, equality holds for big polarizations, and Theorem~\ref{thm: moriwaki},
$\Prep(f)=\Prep(g)$. Corollary~\ref{cor: prep-free} then implies that
$\langle f,g\rangle$ is not freely generated by $f$ and $g$, contradicting the hypothesis. Hence $S$
is not Zariski dense.
\end{proof}

\bibliographystyle{plain}
\bibliography{Mybio}

@article{BTJFA87,
 author = {Bedford, Eric and Taylor, B. A.},
 title = {Fine topology, {{\v{S}}ilov} boundary, and {{$(dd^ c)^ n$}}},
 fjournal = {Journal of Functional Analysis},
 journal = {J. Funct. Anal.},
 issn = {0022-1236},
 volume = {72},
 pages = {225--251},
 year = {1987},
 language = {English},
 doi = {10.1016/0022-1236(87)90087-5},
 zbMATH = {4109104},
 Zbl = {0677.31005}
}

@article{BedfordJonssonAJM00,
 author = {Bedford, Eric and Jonsson, Mattias},
 title = {Dynamics of regular polynomial endomorphisms of {{$\mathbb{C}^k$}}},
 fjournal = {American Journal of Mathematics},
 journal = {Amer. J. Math.},
 issn = {0002-9327},
 volume = {122},
 number = {1},
 pages = {153--212},
 year = {2000},
 language = {English},
 doi = {10.1353/ajm.2000.0001},
 url = {muse.jhu.edu/journals/american_journal_of_mathematics/toc/ajm122.1.html},
 zbMATH = {1416011},
 Zbl = {0941.37027}
}

@article{DujardinPM04,
 author = {Dujardin, Romain},
 title = {The intersection of laminar currents},
 fjournal = {Publicacions Matem{\`a}tiques},
 journal = {Publ. Mat.},
 issn = {0214-1493},
 volume = {48},
 number = {1},
 pages = {107--125},
 year = {2004},
 language = {French},
 doi = {10.5565/PUBLMAT_48104_05},
 url = {https://eudml.org/doc/41492},
 zbMATH = {2074332},
 Zbl = {1048.32021}
}

@article{BS4,
 author = {Bedford, Eric and Lyubich, Mikhail and Smillie, John},
 title = {Polynomial diffeomorphisms of {{$\mathbb{C}^ 2$}}. {IV}: {The} measure of maximal entropy and laminar currents},
 fjournal = {Inventiones Mathematicae},
 journal = {Invent. Math.},
 issn = {0020-9910},
 volume = {112},
 number = {1},
 pages = {77--125},
 year = {1993},
 language = {English},
 doi = {10.1007/BF01232426},
 url = {https://eudml.org/doc/144096},
 zbMATH = {447094},
 Zbl = {0792.58034}
}

@article{BHPTjems24,
 author = {Bell, Jason P. and Huang, Keping and Peng, Wayne and Tucker, Thomas J.},
 title = {A {Tits} alternative for endomorphisms of the projective line},
 fjournal = {Journal of the European Mathematical Society (JEMS)},
 journal = {J. Eur. Math. Soc. (JEMS)},
 issn = {1435-9855},
 volume = {26},
 number = {12},
 pages = {4903--4922},
 year = {2024},
 language = {English},
 doi = {10.4171/JEMS/1376},
 zbMATH = {7927746},
 Zbl = {1552.20224}
}

@article{LevinSymmetries90,
 author = {Levin, G. M.},
 title = {Symmetries on the {Julia} set},
 fjournal = {Mathematical Notes},
 journal = {Math. Notes},
 issn = {0001-4346},
 volume = {48},
 number = {5},
 pages = {1126--1131},
 year = {1990},
 language = {English},
 doi = {10.1007/BF01236299},
 zbMATH = {4194089},
 Zbl = {0724.30021}
}

@article{LevinPrzytycki97,
 author = {Levin, G. and Przytycki, F.},
 title = {When do two rational functions have the same {Julia} set?},
 fjournal = {Proceedings of the American Mathematical Society},
 journal = {Proc. Amer. Math. Soc.},
 issn = {0002-9939},
 volume = {125},
 number = {7},
 pages = {2179--2190},
 year = {1997},
 language = {English},
 doi = {10.1090/S0002-9939-97-03810-0},
 zbMATH = {1017496},
 Zbl = {0870.58085}
}

@misc{beaumont2025centralizersendomorphismsprojectiveline,
      title={On the centralizers of endomorphisms of the projective line}, 
      author={Alonso Beaumont},
      year={2025},
      eprint={2512.13523},
      archivePrefix={arXiv},
      primaryClass={math.DS},
      url={https://arxiv.org/abs/2512.13523},
  note = {arXiv:2512.13523},
}

@misc{beaumont2025uniformtitsalternativeendomorphisms,
      title={A uniform {Tits} alternative for endomorphisms of the projective line}, 
      author={Alonso Beaumont},
      year={2025},
      eprint={2504.14263},
      archivePrefix={arXiv},
      primaryClass={math.NT},
      url={https://arxiv.org/abs/2504.14263},
  note = {arXiv:2504.14263},
}

@article{BT89,
 author = {Bedford, Eric and Taylor, B. A.},
 title = {Uniqueness for the complex {Monge}-{Amp{\`e}re} equation for functions of logarithmic growth},
 fjournal = {Indiana University Mathematics Journal},
 journal = {Indiana Univ. Math. J.},
 issn = {0022-2518},
 volume = {38},
 number = {2},
 pages = {455--469},
 year = {1989},
 language = {English},
 doi = {10.1512/iumj.1989.38.38021},
 zbMATH = {4109107},
 Zbl = {0677.32002}
}

@article{Ye2015,
  author  = {Ye, Hexi},
  title   = {Rational functions with identical measure of maximal entropy},
  journal = {Adv. Math.},
  volume  = {268},
  year    = {2015},
  pages   = {373--395},
  doi     = {10.1016/j.aim.2014.10.001}
}

@article{BakerEremenko1987,
  author  = {Baker, I. N. and Eremenko, A.},
  title   = {A problem on {J}ulia sets},
  journal = {Ann. Acad. Sci. Fenn. Ser. A I Math.},
  volume  = {12},
  number  = {2},
  year    = {1987},
  pages   = {229--236},
  doi     = {10.5186/aasfm.1987.1205}
}

@article{PS89,
 author = {Pugh, Charles and Shub, Michael},
 title = {Ergodic attractors},
 fjournal = {Transactions of the American Mathematical Society},
 journal = {Trans. Amer. Math. Soc.},
 issn = {0002-9947},
 volume = {312},
 number = {1},
 pages = {1--54},
 year = {1989},
 language = {English},
 doi = {10.2307/2001206},
 zbMATH = {4121056},
 Zbl = {0684.58008}
}

@book{Klimek,
 author = {Klimek, Maciej},
 title = {Pluripotential theory},
 fseries = {London Mathematical Society Monographs. New Series},
 series = {Lond. Math. Soc. Monogr., New Ser.},
 volume = {6},
 isbn = {0-19-853568-6},
 year = {1991},
 publisher = {Oxford etc.: Clarendon Press},
 language = {English},
 zbMATH = {52146},
 Zbl = {0742.31001}
}

@article{Beardon1992,
  author  = {Beardon, A. F.},
  title   = {Polynomials with identical {J}ulia sets},
  journal = {Complex Variables Theory Appl.},
  volume  = {17},
  number  = {3--4},
  pages   = {195--200},
  year    = {1992},
  doi     = {10.1080/17476939208814513}
}

@article{Moriwaki00,
 author = {Moriwaki, Atsushi},
 title = {Arithmetic height functions over finitely generated fields},
 fjournal = {Inventiones Mathematicae},
 journal = {Invent. Math.},
 issn = {0020-9910},
 volume = {140},
 number = {1},
 pages = {101--142},
 year = {2000},
 language = {English},
 doi = {10.1007/s002220050358},
 zbMATH = {1443593},
 Zbl = {1007.11042}
}

@article{Moriwaki01,
 author = {Moriwaki, Atsushi},
 title = {The canonical arithmetic height of subvarieties of an abelian variety over a finitely generated field},
 fjournal = {Journal f{\"u}r die Reine und Angewandte Mathematik},
 journal = {J. Reine Angew. Math.},
 issn = {0075-4102},
 volume = {530},
 pages = {33--54},
 year = {2001},
 language = {English},
 doi = {10.1515/crll.2001.005},
 zbMATH = {1559508},
 Zbl = {1013.11032}
}

@book{YZquasi,
 author = {Yuan, Xinyi and Zhang, Shou-Wu},
 title = {Adelic line bundles on quasi-projective varieties},
 fseries = {Annals of Mathematics Studies},
 series = {Ann. Math. Stud.},
 volume = {221},
 isbn = {978-0-691-27172-9; 978-0-691-27173-6; 978-0-691-27870-4},
 year = {2026},
 publisher = {Princeton, NJ: Princeton University Press},
 language = {English},
 zbMATH = {8063230}
}

@article{YZhodge17,
 author = {Yuan, Xinyi and Zhang, Shou-Wu},
 title = {The arithmetic {Hodge} index theorem for adelic line bundles},
 fjournal = {Mathematische Annalen},
 journal = {Math. Ann.},
 issn = {0025-5831},
 volume = {367},
 number = {3-4},
 pages = {1123--1171},
 year = {2017},
 language = {English},
 doi = {10.1007/s00208-016-1414-1},
 zbMATH = {6706027},
 Zbl = {1372.14017}
}

@article{FreireLopesMane1983,
  author  = {Freire, Alexandre and Lopes, Artur and Ma\~n\'e, Ricardo},
  title   = {An invariant measure for rational maps},
  journal = {Bol. Soc. Brasil. Mat.},
  volume  = {14},
  number  = {1},
  year    = {1983},
  pages   = {45--62},
  doi     = {10.1007/BF02584744}
}

@article{Lyubich1983,
  author  = {Lyubich, Mikhail Ju.},
  title   = {Entropy properties of rational endomorphisms of the {R}iemann sphere},
  journal = {Ergodic Theory Dynam. Systems},
  volume  = {3},
  number  = {3},
  year    = {1983},
  pages   = {351--385},
  doi     = {10.1017/S0143385700002030}
}

@article{Brolin1965,
  author  = {Brolin, Hans},
  title   = {Invariant sets under iteration of rational functions},
  journal = {Ark. Mat.},
  volume  = {6},
  year    = {1965},
  pages   = {103--144},
  doi     = {10.1007/BF02591353}
}

@article{Beardon1990,
  author  = {Beardon, A. F.},
  title   = {Symmetries of {J}ulia sets},
  journal = {Bull. London Math. Soc.},
  volume  = {22},
  number  = {6},
  year    = {1990},
  pages   = {576--582},
  doi     = {10.1112/blms/22.6.576}
}

@article{YusukeMagorzata,
 author = {Okuyama, Y{\^u}suke and Stawiska, Ma{\l}gorzata},
 title = {Potential theory and a characterization of polynomials in complex dynamics},
 fjournal = {Conformal Geometry and Dynamics},
 journal = {Conform. Geom. Dyn.},
 issn = {1088-4173},
 volume = {15},
 pages = {152--159},
 year = {2011},
 language = {English},
 doi = {10.1090/S1088-4173-2011-00230-X},
 zbMATH = {6040467},
 Zbl = {1252.37036}
}

@article{AtelaHu1996,
  author  = {Atela, Pau and Hu, Jun},
  title   = {Commuting Polynomials and Polynomials with Same {J}ulia Set},
  journal = {Internat. J. Bifur. Chaos Appl. Sci. Engrg.},
  volume  = {6},
  number  = {12A},
  year    = {1996},
  pages   = {2427--2432},
  publisher = {World Scientific Publishing Company}
}

@article{SchmidtSteinmetz1995,
  author  = {Schmidt, W. and Steinmetz, N.},
  title   = {The Polynomials Associated with a {J}ulia Set},
  journal = {Bull. London Math. Soc.},
  volume  = {27},
  number  = {3},
  year    = {1995},
  pages   = {239--241},
  doi     = {10.1112/blms/27.3.239}
}

@article{JiXie_LocalRigidity,
  author  = {Ji, Zhuchao and Xie, Junyi},
  title   = {Local rigidity of {J}ulia sets},
  journal = {Amer. J. Math.},
  year    = {2023},
  note    = {to appear; arXiv:2302.02562},
  eprint  = {2302.02562},
  archivePrefix = {arXiv},
  primaryClass = {math.DS},
  url = {https://arxiv.org/abs/2302.02562}
}

@article{DujardinFavreGauthier2023,
  author  = {Dujardin, Romain and Favre, Charles and Gauthier, Thomas},
  title   = {When do two rational functions have locally biholomorphic {J}ulia sets?},
  journal = {Trans. Amer. Math. Soc.},
  volume  = {376},
  number  = {3},
  year    = {2023},
  pages   = {1601--1624},
  doi     = {10.1090/tran/8775}
}

@book{FavreGauthier2022,
  author    = {Favre, Charles and Gauthier, Thomas},
  title     = {The Arithmetic of Polynomial Dynamical Pairs},
  series    = {Annals of Mathematics Studies},
  volume    = {214},
  publisher = {Princeton University Press},
  address   = {Princeton, NJ},
  year      = {2022}
}

@misc{ji2023daocurves,
      title={{DAO} for curves}, 
      author={Zhuchao Ji and Junyi Xie},
      year={2023},
      eprint={2302.02583},
      archivePrefix={arXiv},
      primaryClass={math.DS},
      url={https://arxiv.org/abs/2302.02583},
  note = {arXiv:2302.02583},
}

@article{JiXieMultiplierSpectrum,
  author = {Ji, Zhuchao and Xie, Junyi},
  title = {The multiplier spectrum morphism is generically injective},
  year = {2025},
  note = {Published online first; doi:\href{https://doi.org/10.4171/JEMS/1750}{10.4171/JEMS/1750}; arXiv:2309.15382},
  eprint = {2309.15382},
  archivePrefix = {arXiv},
  primaryClass = {math.DS},
  url = {https://arxiv.org/abs/2309.15382},
  journal = {J. Eur. Math. Soc. (JEMS)},
  doi = {10.4171/JEMS/1750},
}

@article{BakerDeMarco2011,
  author  = {Baker, Matthew and DeMarco, Laura},
  title   = {Preperiodic points and unlikely intersections},
  journal = {Duke Math. J.},
  volume  = {159},
  number  = {1},
  year    = {2011},
  pages   = {1--29},
  doi     = {10.1215/00127094-1384773}
}

@article{DMCompositio,
 author = {DeMarco, Laura and Mavraki, Niki Myrto},
 title = {Dynamics on {{$\mathbb{P}^1$}}: preperiodic points and pairwise stability},
 fjournal = {Compositio Mathematica},
 journal = {Compos. Math.},
 issn = {0010-437X},
 volume = {160},
 number = {2},
 pages = {356--387},
 year = {2024},
 language = {English},
 doi = {10.1112/S0010437X23007546},
 zbMATH = {7793888},
 Zbl = {1535.37114}
}

@article{MSDuke,
 author = {Mavraki, Niki Myrto and Schmidt, Harry},
 title = {On the dynamical {Bogomolov} conjecture for families of split rational maps},
 fjournal = {Duke Mathematical Journal},
 journal = {Duke Math. J.},
 issn = {0012-7094},
 volume = {174},
 number = {5},
 pages = {803--856},
 year = {2025},
 language = {English},
 doi = {10.1215/00127094-2024-0041},
 url = {projecteuclid.org/journals/duke-mathematical-journal/volume-174/issue-5/On-the-dynamical-Bogomolov-conjecture-for-families-of-split-rational/10.1215/00127094-2024-0041.full},
 zbMATH = {8050359}
}

@article{Jonsson1999,
  author  = {Jonsson, Mattias},
  title   = {Dynamics of polynomial skew products on {$\mathbb{C}^2$}},
  journal = {Math. Ann.},
  volume  = {314},
  number  = {3},
  pages   = {403--447},
  year    = {1999},
  doi     = {10.1007/s002080050301},
}

@article{FavreGuedj2001,
  author  = {Favre, Charles and Guedj, Vincent},
  title   = {Dynamique des applications rationnelles des espaces multiprojectifs},
  journal = {Indiana Univ. Math. J.},
  volume  = {50},
  number  = {2},
  pages   = {881--934},
  year    = {2001},
  doi     = {10.1512/iumj.2001.50.1880},
}

@article{Ueno2010,
  author  = {Ueno, Kohei},
  title   = {Symmetries of {Julia} sets of nondegenerate polynomial skew products on {$\mathbb{C}^2$}},
  journal = {Michigan Math. J.},
  volume  = {59},
  number  = {1},
  pages   = {153--168},
  year    = {2010},
  doi     = {10.1307/mmj/1272376030},
}

@article{Ueno2020,
  author  = {Ueno, Kohei},
  title   = {Polynomial skew products whose {Julia} sets have infinitely many symmetries},
  journal = {Kyoto J. Math.},
  volume  = {60},
  number  = {2},
  pages   = {451--471},
  year    = {2020},
  doi     = {10.1215/21562261-2019-0040},
  eprint  = {1204.0598},
  archivePrefix = {arXiv},
}

@article{Ritt1923,
  author  = {Ritt, J. F.},
  title   = {Permutable rational functions},
  journal = {Trans. Amer. Math. Soc.},
  volume  = {25},
  number  = {3},
  pages   = {399--448},
  year    = {1923},
  doi     = {10.2307/1989018}
}

@article{Pak20,
 author = {Pakovich, Fedor},
 title = {Finiteness theorems for commuting and semiconjugate rational functions},
 fjournal = {Conformal Geometry and Dynamics},
 journal = {Conform. Geom. Dyn.},
 issn = {1088-4173},
 volume = {24},
 pages = {202--229},
 year = {2020},
 language = {English},
 doi = {10.1090/ecgd/354},
 zbMATH = {7271858},
 Zbl = {1451.30053}
}

@article{Pak21,
 author = {Pakovich, Fedor},
 title = {Commuting rational functions revisited},
 fjournal = {Ergodic Theory and Dynamical Systems},
 journal = {Ergodic Theory Dynam. Systems},
 issn = {0143-3857},
 volume = {41},
 number = {1},
 pages = {295--320},
 year = {2021},
 language = {English},
 doi = {10.1017/etds.2019.51},
 zbMATH = {7282579},
 Zbl = {1461.30059}
}

@article{HsiaTucker,
 author = {Hsia, Liang-Chung and Tucker, Thomas},
 title = {Greatest common divisors of iterates of polynomials},
 fjournal = {Algebra \& Number Theory},
 journal = {Algebra Number Theory},
 issn = {1937-0652},
 volume = {11},
 number = {6},
 pages = {1437--1459},
 year = {2017},
 language = {English},
 doi = {10.2140/ant.2017.11.1437},
 zbMATH = {6763331},
 Zbl = {1392.37103}
}

@misc{noytaptim2024commonzerositeratedmorphisms,
      title={Towards Common Zeros of Iterated Morphisms}, 
      author={Chatchai Noytaptim and Xiao Zhong},
      year={2024},
      eprint={2412.15141},
      archivePrefix={arXiv},
      primaryClass={math.AG},
      url={https://arxiv.org/abs/2412.15141},
  note = {arXiv:2412.15141},
}

@misc{yang2026dynamicalgcdproblemsvariant,
      title={Dynamical {GCD} Problems and a Variant of the Dynamical {Mordell--Lang} Conjecture}, 
      author={She Yang and Xiao Zhong},
      year={2026},
      eprint={2602.18302},
      archivePrefix={arXiv},
      primaryClass={math.NT},
      url={https://arxiv.org/abs/2602.18302},
  note = {arXiv:2602.18302},
}

@misc{abboud2026uniformboundcommonperiodic,
      title={Uniform bound on common periodic points for families of regular plane polynomial automorphisms}, 
      author={Marc Abboud and Yugang Zhang},
      year={2026},
      eprint={2602.10310},
      archivePrefix={arXiv},
      primaryClass={math.DS},
      url={https://arxiv.org/abs/2602.10310},
  note = {arXiv:2602.10310},
}

@incollection{DSbook,
 author = {Dinh, Tien-Cuong and Sibony, Nessim},
 title = {Dynamics in several complex variables: endomorphisms of projective spaces and polynomial-like mappings},
 booktitle = {Holomorphic dynamical systems. Lectures given at the C.I.M.E. summer school, Cetraro, Italy, July 7--12, 2008},
 isbn = {978-3-642-13170-7; 978-3-642-13171-4},
 pages = {165--294},
 year = {2010},
 publisher = {Berlin: Springer},
 language = {English},
 zbMATH = {5879453},
 Zbl = {1218.37055}
}

@incollection{sibonybook,
 author = {Sibony, Nessim},
 title = {Dynamique des applications rationnelles de {$\mathbb{P}^k$}},
 booktitle = {Dynamique et g\'eom\'etrie complexes},
 isbn = {2-85629-078-7},
 pages = {97--185},
 year = {1999},
 publisher = {Soci{\'e}t{\'e} Math{\'e}matique de France},
 language = {French},
 zbMATH = {1908328},
 Zbl = {1020.37026},
  series = {Panoramas et Synth\`eses},
  volume = {8},
  address = {Paris},
}

@article{briendduval2,
 author = {Briend, Jean-Yves and Duval, Julien},
 title = {Deux caract{\'e}risations de la mesure d'{\'e}quilibre d'un endomorphisme de {{$P^k(\mathbb{C})$}}},
 fjournal = {Publications Math{\'e}matiques},
 journal = {Publ. Math. Inst. Hautes \'{E}tudes Sci.},
 issn = {0073-8301},
 volume = {93},
 pages = {145--159},
 year = {2001},
 language = {French},
 doi = {10.1007/s10240-001-8190-4}
}

@article{DinhCommuting,
 author = {Dinh, Tien-Cuong},
 title = {Sur les endomorphismes polynomiaux permutables de {$\mathbb{C}^2$}},
 fjournal = {Annales de l'Institut Fourier},
 journal = {Ann. Inst. Fourier (Grenoble)},
 issn = {0373-0956},
 volume = {51},
 number = {2},
 pages = {431--459},
 year = {2001},
 language = {French},
 doi = {10.5802/aif.1828},
 url = {https://eudml.org/doc/115921},
 zbMATH = {1584199},
 Zbl = {0977.30016}
}

@article{BCZ03,
 author = {Bugeaud, Yann and Corvaja, Pietro and Zannier, Umberto},
 title = {An upper bound for the g.c.d. of {{$a^n-1$}} and {{$b^n -1$}}},
 fjournal = {Mathematische Zeitschrift},
 journal = {Math. Z.},
 issn = {0025-5874},
 volume = {243},
 number = {1},
 pages = {79--84},
 year = {2003},
 language = {English},
 doi = {10.1007/s00209-002-0449-z},
 zbMATH = {1922486},
 Zbl = {1021.11001}
}

@article{DKY2,
 author = {DeMarco, Laura and Krieger, Holly and Ye, Hexi},
 title = {Common preperiodic points for quadratic polynomials},
 fjournal = {Journal of Modern Dynamics},
 journal = {J. Mod. Dyn.},
 issn = {1930-5311},
 volume = {18},
 pages = {363--413},
 year = {2022},
 language = {English},
 doi = {10.3934/jmd.2022012},
 zbMATH = {7594475},
 Zbl = {1503.37101}
}

@article{dujardin2023dynamicalmaninmumfordconjectureplane,
      title={On the dynamical {Manin--Mumford} conjecture for plane polynomial maps}, 
      author={Romain Dujardin and Charles Favre and Matteo Ruggiero},
      year={2026},
      eprint={2312.14817},
      archivePrefix={arXiv},
      primaryClass={math.DS},
      url={https://arxiv.org/abs/2312.14817},
  journal = {J. Eur. Math. Soc. (JEMS)},
  note = {arXiv:2312.14817},
}

@article{zbMATH00825311,
 author = {Okni{\'n}ski, J. and Salwa, A.},
 title = {Generalised {Tits} alternative for linear semigroups},
 fjournal = {Journal of Pure and Applied Algebra},
 journal = {J. Pure Appl. Algebra},
 issn = {0022-4049},
 volume = {103},
 number = {2},
 pages = {211--220},
 year = {1995},
 language = {English},
 doi = {10.1016/0022-4049(95)00104-5},
 zbMATH = {825311},
 Zbl = {0846.20067}
}

@misc{pakovich2026periodiccurvesgeneralendomorphisms,
      title={Periodic curves for general endomorphisms of {$\mathbb{CP}^1\times \mathbb{CP}^1$}}, 
      author={Fedor Pakovich},
      year={2026},
      eprint={2506.09948},
      archivePrefix={arXiv},
      primaryClass={math.DS},
      url={https://arxiv.org/abs/2506.09948},
  note = {arXiv:2506.09948},
}

@article{pakovich-deg4,
 author = {Pakovich, Fedor},
 title = {On iterates of rational functions with maximal number of critical values},
 fjournal = {Journal d'Analyse Math{\'e}matique},
 journal = {J. Anal. Math.},
 issn = {0021-7670},
 volume = {156},
 number = {1},
 pages = {213--251},
 year = {2025},
 language = {English},
 doi = {10.1007/s11854-025-0386-z},
 zbMATH = {8103492},
 Zbl = {1573.30064}
}

@article{DinhSibony2002,
  author  = {Dinh, Tien-Cuong and Sibony, Nessim},
  title   = {Sur les endomorphismes holomorphes permutables de $\mathbb{P}^k$},
  journal = {Math. Ann.},
  volume  = {324},
  number  = {1},
  year    = {2002},
  pages   = {33--70},
  note    = {Preprint version: arXiv:math/0007017 (2000)},
  url     = {https://arxiv.org/abs/math/0007017}
}

@article{Kaufmann2018,
  author  = {Kaufmann, Lucas},
  title   = {Commuting pairs of endomorphisms of $\mathbb{P}^2$},
  journal = {Ergodic Theory Dynam. Systems},
  volume  = {38},
  number  = {3},
  year    = {2018},
  pages   = {1025--1047},
  doi     = {10.1017/etds.2016.54},
  note    = {Published online 19 September 2016; preprint arXiv:1509.06534}
}

@article{Fatou1923,
  author  = {Fatou, Pierre},
  title   = {Sur l'it{\'e}ration analytique et les substitutions permutables},
  journal = {J. Math. Pures Appl. (9)},
  series  = {9},
  volume  = {2},
  year    = {1923},
  pages   = {343--384},
  note    = {Continued in vol.~3 (1924), 1--49},
  url     = {http://eudml.org/doc/234679}
}

@article{Julia1922,
  author  = {Julia, Gaston},
  title   = {M{\'e}moire sur la permutabilit{\'e} des fractions rationnelles},
  journal = {Ann. Sci. \'{E}cole Norm. Sup. (3)},
  series  = {3},
  volume  = {39},
  year    = {1922},
  pages   = {131--215},
  doi     = {10.24033/asens.740}
}

@article{Eremenko1990,
  author  = {Eremenko, Alexandre},
  title   = {Some functional equations connected with the iteration of rational functions},
  journal = {Leningrad Math. J.},
  volume  = {1},
  number  = {4},
  year    = {1990},
  pages   = {905--919},
  note    = {Translated from Algebra i Analiz \textbf{1} (1989), no.~4, 102--116}
}

@article{DNS10JDG,
 author = {Dinh, Tien-Cuong and Nguy{\^e}n, Vi{\^e}t-Anh and Sibony, Nessim},
 title = {Exponential estimates for plurisubharmonic functions},
 fjournal = {Journal of Differential Geometry},
 journal = {J. Differ. Geom.},
 issn = {0022-040X},
 volume = {84},
 number = {3},
 pages = {465--488},
 year = {2010},
 language = {English},
 doi = {10.4310/jdg/1279114298},
 zbMATH = {5774388},
 Zbl = {1211.32021}
}

@article{DS03PolynomialLike,
 author = {Dinh, Tien-Cuong and Sibony, Nessim},
 title = {Dynamics of polynomial-like mappings},
 fjournal = {Journal de Math{\'e}matiques Pures et Appliqu{\'e}es. Neuvi{\`e}me S{\'e}rie},
 journal = {J. Math. Pures Appl. (9)},
 issn = {0021-7824},
 volume = {82},
 number = {4},
 pages = {367--423},
 year = {2003},
 language = {French},
 doi = {10.1016/S0021-7824(03)00026-6},
 zbMATH = {1985982},
 Zbl = {1033.37023}
}

@book{Humphreys1975,
 author = {Humphreys, James E.},
 title = {Linear algebraic groups},
 fseries = {Graduate Texts in Mathematics},
 series = {Grad. Texts Math.},
 issn = {0072-5285},
 volume = {21},
 year = {1975},
 publisher = {Springer, Cham},
 language = {English},
 zbMATH = {3508744},
 Zbl = {0325.20039}
}

@book{Milne2017AlgebraicGroups,
  author    = {Milne, J. S.},
  title     = {Algebraic Groups: The Theory of Group Schemes of Finite Type over a Field},
  publisher = {Cambridge University Press},
  year      = {2017}
}

@book{Hartshorne1977,
  author    = {Hartshorne, Robin},
  title     = {Algebraic Geometry},
  series    = {Graduate Texts in Mathematics},
  volume    = {52},
  publisher = {Springer-Verlag},
  address   = {New York},
  year      = {1977},
  doi       = {10.1007/978-1-4757-3849-0}
}

@article{GuralnickLawther2024,
  author    = {Guralnick, Robert M. and Lawther, Ross},
  title     = {Generic stabilizers in actions of simple algebraic groups},
  journal   = {Mem. Amer. Math. Soc.},
  volume    = {300},
  number    = {1502},
  year      = {2024}
}

@article{SzpiroUllmoZhang1997,
  author  = {Szpiro, Lucien and Ullmo, Emmanuel and Zhang, Shou-Wu},
  title   = {{\'{E}}quir\'{e}partition des petits points},
  journal = {Invent. Math.},
  volume  = {127},
  number  = {2},
  pages   = {337--347},
  year    = {1997},
  doi     = {10.1007/s002220050123}
}

@article{Ullmo1998,
  author  = {Ullmo, Emmanuel},
  title   = {Positivit\'{e} et discr\'{e}tion des points alg\'{e}briques des courbes},
  journal = {Ann. of Math. (2)},
  volume  = {147},
  number  = {1},
  pages   = {167--179},
  year    = {1998},
  doi     = {10.2307/120987}
}

@article{Zhang1998,
  author  = {Zhang, Shou-Wu},
  title   = {Equidistribution of small points on abelian varieties},
  journal = {Ann. of Math. (2)},
  volume  = {147},
  number  = {1},
  pages   = {159--165},
  year    = {1998},
  doi     = {10.2307/120986}
}

@incollection{DujardinICM2022,
  author    = {Dujardin, Romain},
  title     = {Geometric methods in holomorphic dynamics},
  booktitle = {Proceedings of the International Congress of Mathematicians 2022},
  volume    = {5},
  pages     = {3460--3482},
  publisher = {EMS Press},
  address   = {Berlin},
  year      = {2023},
  doi       = {10.4171/ICM2022/60}
}

@misc{zhang2024arithmeticpropertiesfamiliesplane,
  title         = {Arithmetic properties of families of plane polynomial automorphisms},
  author        = {Yugang Zhang},
  year          = {2024},
  eprint        = {2407.15952},
  archivePrefix = {arXiv},
  primaryClass  = {math.DS},
  url           = {https://arxiv.org/abs/2407.15952},
  note = {arXiv:2407.15952},
}

@article{DinhSameJulia2000,
  author   = {Dinh, Tien-Cuong},
  title    = {Remarque sur les fonctions ayant le m{\^e}me ensemble de {Julia}},
  journal  = {Ann. Fac. Sci. Toulouse Math. (6)},
  volume   = {9},
  number   = {1},
  pages    = {55--70},
  year     = {2000},
  mrnumber = {1815940},
  zbl      = {1022.37034},
  url      = {https://www.numdam.org/item/AFST_2000_6_9_1_55_0/}
}

@misc{gong2026nonarchimedeanrigidityuniformitycommon,
  author        = {Gong, Chen and Yap, Jit Wu},
  title         = {Non-{A}rchimedean Rigidity and Uniformity for Common Preperiodic Points},
  year          = {2026},
  eprint        = {2607.19252},
  archivePrefix = {arXiv},
  primaryClass  = {math.DS},
  url           = {https://arxiv.org/abs/2607.19252},
  note = {arXiv:2607.19252},
}

@article{DujardinFavre2017,
  author  = {Dujardin, Romain and Favre, Charles},
  title   = {The dynamical {Manin--Mumford} problem for plane polynomial automorphisms},
  journal = {J. Eur. Math. Soc. (JEMS)},
  volume  = {19},
  number  = {11},
  pages   = {3421--3465},
  year    = {2017},
  doi     = {10.4171/JEMS/743}
}

@article{GauthierTaflinVigny2026,
  author  = {Gauthier, Thomas and Taflin, Johan and Vigny, Gabriel},
  title   = {Sparsity of postcritically finite maps of $\mathbb{P}^k$ and beyond: {A} complex analytic approach},
  journal = {Publ. Math. Inst. Hautes \'{E}tudes Sci.},
  volume  = {143},
  pages   = {1--96},
  year    = {2026},
  doi     = {10.5802/pmihes.1}
}

@article{FavreRiveraLetelier2006,
  author  = {Favre, Charles and Rivera-Letelier, Juan},
  title   = {{\'E}quidistribution quantitative des points de petite hauteur
             sur la droite projective},
  journal = {Math. Ann.},
  volume  = {335},
  number  = {2},
  pages   = {311--361},
  year    = {2006},
  doi     = {10.1007/s00208-006-0751-x}
}

@article{BakerRumely2006,
  author  = {Baker, Matthew H. and Rumely, Robert},
  title   = {Equidistribution of small points, rational dynamics,
             and potential theory},
  journal = {Ann. Inst. Fourier (Grenoble)},
  volume  = {56},
  number  = {3},
  pages   = {625--688},
  year    = {2006},
  doi     = {10.5802/aif.2196}
}

@article{ChambertLoir2006,
  author  = {Chambert-Loir, Antoine},
  title   = {Mesures et {\'e}quidistribution sur les espaces de {Berkovich}},
  journal = {J. Reine Angew. Math.},
  volume  = {595},
  pages   = {215--235},
  year    = {2006},
  doi     = {10.1515/CRELLE.2006.049}
}

@article{Yuan2008,
  author  = {Yuan, Xinyi},
  title   = {Big line bundles over arithmetic varieties},
  journal = {Invent. Math.},
  volume  = {173},
  number  = {3},
  pages   = {603--649},
  year    = {2008},
  doi     = {10.1007/s00222-008-0127-9}
}

@article{ChambertLoirThuillier2009,
  author  = {Chambert-Loir, Antoine and Thuillier, Amaury},
  title   = {Mesures de {Mahler} et {\'e}quidistribution logarithmique},
  journal = {Ann. Inst. Fourier (Grenoble)},
  volume  = {59},
  number  = {3},
  pages   = {977--1014},
  year    = {2009},
  doi     = {10.5802/aif.2454}
}

@article{GauthierGoodHeights2026,
  author  = {Gauthier, Thomas},
  title   = {Good height functions on quasi-projective varieties:
             equidistribution and applications in dynamics},
  journal = {Ann. Fac. Sci. Toulouse Math. (6)},
  volume  = {35},
  number  = {1},
  pages   = {57--94},
  year    = {2026},
  doi     = {10.5802/afst.1841}
}

@article{Bilu1997,
  author  = {Bilu, Yuri},
  title   = {Limit distribution of small points on algebraic tori},
  journal = {Duke Math. J.},
  volume  = {89},
  number  = {3},
  pages   = {465--476},
  year    = {1997},
  doi     = {10.1215/S0012-7094-97-08921-3}
}

@article{MavrakiYe2023,
  author  = {Mavraki, Niki Myrto and Ye, Hexi},
  title   = {Quasi-adelic measures and equidistribution on
             {$\mathbb{P}^1$}},
  journal = {Ergodic Theory Dynam. Systems},
  volume  = {43},
  number  = {8},
  pages   = {2732--2779},
  year    = {2023},
  doi     = {10.1017/etds.2022.49}
}

@misc{KuhneEquidistribution2021,
  author        = {K{\"u}hne, Lars},
  title         = {Equidistribution in families of abelian varieties
                   and uniformity},
  year          = {2021},
  eprint        = {2101.10272},
  archivePrefix = {arXiv},
  primaryClass  = {math.NT},
  note = {arXiv:2101.10272},
}

@misc{Zhang2026IteratedCentralizer,
  author        = {Zhang, Yugang},
  title         = {A general endomorphism of {$\mathbb{P}^k$} has trivial iterated centralizer},
  year          = {2026},
  eprint        = {2609.09858},
  archivePrefix = {arXiv},
  primaryClass  = {math.DS},
  note          = {arXiv:2609.09858},
  url           = {https://arxiv.org/abs/2609.09858}
}

@misc{JiSongXie2026,
  author        = {Ji, Zhuchao and Song, Jiarui and Xie, Junyi},
  title         = {Exceptional maps and abelian points in backward orbits},
  year          = {2026},
  eprint        = {2609.24701},
  archivePrefix = {arXiv},
  primaryClass  = {math.NT},
  note          = {arXiv:2609.24701},
  url           = {https://arxiv.org/abs/2609.24701}
}

@misc{Vigny2026HenonSibony,
  author        = {Vigny, Gabriel},
  title         = {Rigidity of the {Julia} set for {H{\'e}non--Sibony} maps},
  year          = {2026},
  eprint        = {2605.20935},
  archivePrefix = {arXiv},
  primaryClass  = {math.DS},
  note          = {arXiv:2605.20935},
  url           = {https://arxiv.org/abs/2605.20935}
}

@misc{BianchiHe2026Joining,
  author        = {Bianchi, Fabrizio and He, Yan Mary},
  title         = {Joining rigidity for rational maps},
  year          = {2026},
  eprint        = {2609.14890},
  archivePrefix = {arXiv},
  primaryClass  = {math.DS},
  note          = {arXiv:2609.14890},
  url           = {https://arxiv.org/abs/2609.14890}
}

@article{Luo2021Inhomogeneity,
  author  = {Luo, Yusheng},
  title   = {On the inhomogeneity of the {Mandelbrot} set},
  journal = {Int. Math. Res. Not. IMRN},
  volume  = {2021},
  number  = {8},
  pages   = {6051--6076},
  year    = {2021},
  doi     = {10.1093/imrn/rnz035},
  url     = {https://doi.org/10.1093/imrn/rnz035}
}
\end{document}